\documentclass[12 pt]{article}
	\usepackage[english]{babel}
	\usepackage[utf8]{inputenc}  
	\usepackage[T1]{fontenc}
	\usepackage{amssymb}
	\usepackage[mathscr]{euscript}
	\usepackage{stmaryrd}
	\usepackage{amsmath}
	\usepackage{tikz}
	\usepackage[all,cmtip]{xy}
	\usepackage{amsthm}
	\usepackage{varioref}
	\usepackage{geometry}
	\usepackage{lmodern}
	\usepackage[unicode]{hyperref}
	 \usepackage{fancyhdr}
	\usepackage{xargs}
	\usepackage{ifthen}
	\usepackage{caption}
	\usepackage{float}
	\theoremstyle{plain}
	\newcounter{n}
	\numberwithin{n}{section}
	\newtheorem{theo}[n]{Theorem}
	\labelformat{theo}{Theorem~#1}
	\labelformat{rmq}{remarque~#1}
	\newtheorem{cor}[n]{Corollary}
	\newtheorem{lm}[n]{Lemma}
	\labelformat{lm}{lemma~#1}
	
	\newtheorem{prop}[n]{Proposition}
	\labelformat{prop}{proposition~#1}
	{\theoremstyle{definition}
	\newtheorem{sett}[n]{Setting}
	
	\newtheorem{nt}[n]{Notation}
    \newtheorem{rmq}[n]{Remark}	
	\newtheorem{df}[n]{Definition}}
	
	\newtheorem*{thmnonnum}{Theorem}
	\newtheorem{theou}[n]{Theorem}
	
 \usetikzlibrary{patterns}
	\renewcommand\epsilon{\varepsilon}
	\renewcommand\phi{\varphi}
	\newcommand\R{\mathbb{R}}
	\newcommand\s{\mathbb{S}}
	\newcommand{\Int}{\mathrm{int}}

\newcommand\QQ{\mathbb{Q}}
\newcommand\und{\underline}
\newcommand\Card{{\mathrm{Card}}}
\newcommand\dimd{{d}}

\newcommand\degk{{k}}
\newcommand\Ln{\mathrm{Ln}}

\newcommand\Tr{{\mathrm{Tr}}}

\newcommand\torsion[1][\psi]{{\mathcal T_{#1}}}

\newcommand\moinsun[1]{{{#1}^-}}
\newcommand\plusun[1]{{{#1}^+}}
\newcommand\zk{{\und{k}}}

\newcommand{\punct}[1]{{#1^\circ}}
\newcommand\sphereambiante{{M}}	
\newcommand\sphamb{{M}}
\newcommand\ambientspace{{\punct \sphamb}}
\newcommand\spamb{{\punct \sphamb}}
\newcommandx\bouleambiante[1][1= M]{B(#1)}
\newcommandx\bM[1][1= M]{B(#1)}

\newcommand\voisinageinfini{\punct{B_\infty}}

\newcommandx\nbdinf[1][1= ]{\ifthenelse{\equal{#1}{}}{\punct{B_\infty}}{\punct{B_\infty(#1)}}}

\newcommand\voisinageinfinideux{{\punct{ B_{\infty, 2}}}}

\newcommandx\sommets[2][1=\Gamma,2= \empty]{\ifthenelse{\equal{#2}{}} {V(#1)}{V^*(#1)} }

\newcommand\sommetsinternes[1][\Gamma]{V_i(#1)}

\newcommand\sommetsexternes[1][\Gamma]{V_e(#1)}

\newcommand\aretes[1][\Gamma]{E(#1)}

\newcommand\aretesinternes[1][\Gamma]{E_i(#1)}

\newcommand\aretesexternes[1][\Gamma]{E_e(#1)}

\newcommand\graphes{\mathcal G_k}

\newcommand\graphesnum{\widetilde{\mathcal G_k}}

\newcommandx\confignoeud[3][1=\Gamma, 2=\psi,3=\empty ]{C_{#1}#3(#2)}

\newcommand\configR{C_2(\R^n)}

\newcommand\config{C_2}

\newcommand\configM{\config(\spamb)}

\newcommand\unitaire{{U}}

\newcommand{\N}{{\mathfrak{N}}}

\newcommand\injections{\mathcal I}

\renewcommand\d{\mathrm d}

\newcommand\orientationarete[1]{{\Omega_{#1}}}

\newcommand\orientationconfig[1][\Gamma]{{\Omega(#1)}}

\newcommand\Propint{{A}}

\newcommand\Propext{{B}}

\newcommandx\formearete[2][1=e,2=\sigma]{{\omega^F_{#1,#2}}}

	\newcommand\Z[1][F]{{Z_{\degk}^{#1}}}

\newcommand{\inj}{{\injections(\R^n,\R^{n+2})}}

\newcommand\indices{{\und{2k}}}

\newcommand\mur[1][r]{{\mathbb D_{\mu}( #1)}}

\newcommand\DR[1][r]{{L_{\theta}( #1)}}

\newcommand\Dr[2]{{L_0^#1(#2)}}

\renewcommand\P{{B}}

\newcommand\Dir{{\omega}}

\newcommand\ys{\mathcal Y_s}

	\newcommand\Ba{{\mathcal B}}
	\newcommand\Bb{{\tilde{\mathcal B}}}
	\newcommand\Lk{{\mathcal L}}
	\newcommand\lk{{\mathrm{lk}}}

	\newcommand\bA{{\zeta}}
	\newcommand\bb{{ a }}
	\newcommand\ba{{ z }}
\newcommand{\Pii}[1][i]{P_{#1,\sigma}}
\newcommand\poids{{ \lambda }}

\newcommand{\crc}{circle (0.1)}
\newcommand{\edgi}{\draw[->, >=latex]}
\newcommand{\edge}{\draw[->, >=latex, dashed]}

\newcommand{\Sym}{\mathfrak{S}}

\newcommand{\Num}{\mathrm{Num}}

\tikzset{
    partial ellipse/.style args={#1:#2:#3}{
        insert path={+ (#1:#3) arc (#1:#2:#3)}
    }
}
\newcommand\figurea{{
\begin{tikzpicture}\draw[color=white] (-4, -0.75) rectangle (4.5, 2.5);
\begin{pgflowlevelscope}{\pgftransformscale{0.75}}
\draw (0,0) [partial ellipse=180:110:4.4cm and 2.8cm];
\draw (0,0) [partial ellipse=0:70:4.4cm and 2.8cm];
\draw (0, 2.598) ellipse (1.5 and 0.3);
\draw[dashed] (-1.5, 2.598)-- (-1.5, 0);
\draw[dashed] (1.5, 2.598) --(1.5, 0);
\draw[dashed] (0, 0) ellipse (1.5 and 0.3);
\draw[dashed] (0, 0) [partial ellipse = 0: 180: 4.4 cm and 0.88 cm];
\draw(0, 0) [partial ellipse = 0: -180: 4.4 cm and 0.88 cm];
\draw [<->] (5, 0) -- (5, 2.598);
\draw (6.75, 1.2) node {$h_r = \sqrt{\left(\frac{18}r\right)^2-r^2}$};
\draw [<->] (0, 3) -- (1.5, 3);
\draw (0.75, 3.2) node {$r$} ;
\draw [<->] (-4.4, 3.1) -- (0, 3.1);
\draw (-2, 3.4) node {$R_r= \frac{18}r$};
\draw[->] (-6, -1) -- ++(0, 3);
\draw[->] (-6, -1) -- ++(3,0);
\draw[->] (-6, -1) -- ++(1.5, 1.5);
\draw (-6.3, 1.5) node{$h_x$};
\draw (-3.75, -1.25) node {$x_1$};
\draw (-5, 0.4) node {$x_2$};
\end{pgflowlevelscope}
\end{tikzpicture}
}}

\newcommand\figureba{{\begin{tikzpicture} \draw[white] (-4,0) rectangle (5,3);
\begin{pgflowlevelscope}{\pgftransformscale{0.7}}
\draw (0,0) [partial ellipse=180:110:4.4cm and 2.8cm];
\draw (0,0) [partial ellipse=0:70:4.4cm and 2.8cm];
\draw (0, 2.598) ellipse (1.5 and 0.3);
\draw[dashed] (-1.5, 2.598)-- (-1.5, 0);
\draw[dashed] (1.5, 2.598) --(1.5, 0);
\draw[dashed] (0, 0) ellipse (1.5 and 0.3);
\draw[dashed] (0, 0) [partial ellipse = 0: 180: 4.4 cm and 0.88 cm];
\draw(0, 0) [partial ellipse = 0: -180: 4.4 cm and 0.88 cm];
\fill[red] (-0.6, 3) circle (0.05);
\draw (-0, 3) node {\scriptsize $y$};
\draw[red, ->, >= latex] (-0.6, 3) -- (-2, 3);
\draw[red] (-4, 3) -- (-2, 3);
\draw[red, dashed] (-4.6, 3) -- (-4, 3);
\draw (-3.3, 2.7) node { $D^0(y, -e_1)$};
\end{pgflowlevelscope}
\end{tikzpicture}}}

\newcommand\figurebb{{\begin{tikzpicture} \draw[white] (-4,0) rectangle (5,3);
\begin{pgflowlevelscope}{\pgftransformscale{0.75}}
\draw (0,0) [partial ellipse=180:110:4.4cm and 2.8cm];
\draw (0,0) [partial ellipse=0:70:4.4cm and 2.8cm];
\draw (0, 2.598) ellipse (1.5 and 0.3);
\draw[dashed] (-1.5, 2.598)-- (-1.5, 0);
\draw[dashed] (1.5, 2.598) --(1.5, 0);
\draw[dashed] (0, 0) ellipse (1.5 and 0.3);
\draw[dashed] (0, 0) [partial ellipse = 0: 180: 4.4 cm and 0.88 cm];
\draw(0, 0) [partial ellipse = 0: -180: 4.4 cm and 0.88 cm];
\fill[red] (-0.8, 1.55) circle (0.05) ;
\draw[red, ->, >= latex] (-.8,1.55) -- (-1.1, 1.55);
\draw[red] (-1.1,1.55) -- (-1.3, 1.55);
\fill[red] (-1.3, 1.55) circle (0.05);
\fill[red] (-3.5, 1.55) circle (0.05);
\draw[red] (-4.5,1.55) -- (-5, 1.55);
\draw[red,->, > = latex] (-3.5,1.55) -- (-4.5, 1.55);
\draw[red, dashed] (-5,1.55) -- (-6.5, 1.55);
\draw[dotted] (0, 1.4) ellipse (1.5 and 0.3);
\draw[dotted] (0, 1.4) ellipse (3.7 and 0.55);
\draw (-0.6, 1.45) node {\scriptsize $y$};
\draw (-1.05, 1.75) node {\scriptsize $x_c^-(y)$};
\draw (-3.2, 1.45) node {\scriptsize $x_s^-(y)$};
\draw (-5.2, 1.9) node { $D^0(y, -e_1)$};
\end{pgflowlevelscope}
\end{tikzpicture}}}

\newcommand\figurebc{{\begin{tikzpicture}\draw[white] (-4,0) rectangle (5,3);
\begin{pgflowlevelscope}{\pgftransformscale{0.75}}
\draw (0,0) [partial ellipse=180:110:4.4cm and 2.8cm];
\draw (0,0) [partial ellipse=0:70:4.4cm and 2.8cm];
\draw (0, 2.598) ellipse (1.5 and 0.3);
\draw[dashed] (-1.5, 2.598)-- (-1.5, 0);
\draw[dashed] (1.5, 2.598) --(1.5, 0);
\draw[dashed] (0, 0) ellipse (1.5 and 0.3);
\draw[dashed] (0, 0) [partial ellipse = 0: 180: 4.4 cm and 0.88 cm];
\draw(0, 0) [partial ellipse = 0: -180: 4.4 cm and 0.88 cm];
\fill[red] (4.5, 1.55) circle (0.05) ;
\draw (4.7, 1.55) node {$y$};
\fill[red] (1.3, 1.55) circle (0.05) ;
\draw[red, ->, >= latex] (-1.3,1.55) --  (1.3, 1.55)-- (0, 1.55);
\fill[red] (-1.3, 1.55) circle (0.05);
\fill[red] (-3.5, 1.55) circle (0.05);
\fill[red] (3.5, 1.55) circle (0.05);
\draw[red, ->, >= latex] (3.5,1.55) -- (4.5, 1.55)-- (4, 1.55);
\draw[red, <-, >= latex] (-4.3, 1.55) --(-3.5,1.55) -- (-5, 1.55) ;
\draw[red, dashed] (-5,1.55) -- (-6.5, 1.55);
\draw[dotted] (0, 1.4) ellipse (1.5 and 0.3);
\draw[dotted] (0, 1.4) ellipse (3.7 and 0.55);
\draw (-1.01, 1.79) node {\scriptsize $x_c^-(y)$ };
\draw (1.05, 1.79) node {\scriptsize $x_c^+(y)$};
\draw (3.2, 1.4) node {\scriptsize $x_s^+(y)$};
\draw (-3.2, 1.4) node {\scriptsize $x_s^-(y)$};
\draw (-5.2, 1.8) node { $D^0(y, -e_1)$};
\end{pgflowlevelscope}\end{tikzpicture}}}

\newcommand\figurebd{{
\begin{tikzpicture}\draw[white] (-4,0) rectangle (5,3);
\begin{pgflowlevelscope}{\pgftransformscale{0.75}}
\draw (0,0) [partial ellipse=180:110:4.4cm and 2.8cm];
\draw (0,0) [partial ellipse=0:70:4.4cm and 2.8cm];
\draw (0, 2.598) ellipse (1.5 and 0.3);
\draw[dashed] (-1.5, 2.598)-- (-1.5, 0);
\draw[dashed] (1.5, 2.598) --(1.5, 0);
\draw[dashed] (0, 0) ellipse (1.5 and 0.3);
\draw[dashed] (0, 0) [partial ellipse = 0: 180: 4.4 cm and 0.88 cm];
\draw(0, 0) [partial ellipse = 0: -180: 4.4 cm and 0.88 cm];
\fill[red] (4.5, 1) circle (0.05) ;
\draw (4.8, 1) node {$y$};
\fill[red] (2.6, 1) circle (0.05) ;
\fill[red] (-2.6, 1) circle (0.05);
\draw[red, -> , >= latex](2.6,1) -- (4.5, 1)-- (3.3, 1);
\draw[red, ->,>= latex] (-3.5, 1)--(-2.6,1) -- (-5, 1);
\draw[red, dashed] (-5,1) -- (-6.5, 1);
\draw[dotted] (0, 1.4) ellipse (1.5 and 0.3);
\draw[dotted] (0, 1.4) ellipse (3.7 and 0.55);
\draw (-5.2, 1.3) node { $D^0(y, -e_1)$};
\draw (2.6, 1.2) node {\scriptsize $x_s^+(y)$};
\draw (-2.6, 1.2) node {\scriptsize $x_s^-(y)$};
  \end{pgflowlevelscope}\end{tikzpicture}
}}
\newcommand\figureca{{
\begin{tikzpicture}\draw[white] (-6,-1) rectangle (1,3);
\begin{pgflowlevelscope}{\pgftransformscale{1}}
\draw (0,0) [partial ellipse=180:110:4.4cm and 2.8cm];
\draw (0, 2.598) [partial ellipse = 90 : 270 : 1.5 and 0.3];
\draw[dashed, opacity = .3] (-1.5, 2.598)-- (-1.5, 0);
\draw (0, 0) [partial ellipse = 90 : 270 : 1.5 and 0.3];
\draw[dashed,opacity = .3] (0, 0) [partial ellipse = 90: 180: 4.4 cm and 0.88 cm];
\draw(0, 0) [partial ellipse = -90: -180: 4.4 cm and 0.88 cm];
\draw[dotted,opacity = .3] (0, .7) [partial ellipse = 90 : 270 :2.1 and 0.35];
\draw[dotted,opacity = .3] (0, .7) [partial ellipse = 90: 270 : 3.8 and 0.6];
\draw[red] (-2.1, 1.65) --(-2.1,0.7);
\fill[red] (0, .95) circle (0.05) ;
\draw[color = red, <-, >=latex] ( -1, .95)--(0, 0.95)--(-1.6, .95);
\draw (-0, 1.15) node {\large $y$};
\draw[red] (0, .7)[partial ellipse = 180: 138: 2.1 and 0.38];
\draw[red] (0,0.7) [partial ellipse=180:131:3.2cm and 1.25 cm];
\draw[red] (0, .7) [partial ellipse= 180:151:3.2 and 0.5];
\draw[color = red] (-4.9, 0.95) -- (-2.8, .95);
\draw[color = red, dashed] (-4.9, 0.95)--(-5.9, 0.95);
\end{pgflowlevelscope}
\end{tikzpicture}
}}

\newcommand\figurecb{{
\begin{tikzpicture}\draw[white] (-7  ,-2) rectangle (0, 2);
\begin{pgflowlevelscope}{\pgftransformscale{1}}
\draw (0, 0) [partial ellipse =170 : 190 : 5 and 5];
\draw (0, 0) [partial ellipse =90 : 270 : 1 and 1];
\fill[red][red] (0, 0.7) circle (0.05);
\draw[red, <-, >= latex] (-1, 0.7)--(0, 0.7) -- (-1.885, 0.7);
\draw[red] (-3.93, 0.7) -- (-6, 0.7);
\draw[red, dashed] (-6, .7) -- (-7, .7);
\draw[red] (0,0)[partial ellipse = 159.60: 180: 2 and 2];
\draw[red] (0,0) [partial ellipse = 170: 180 : 4 and 4];
\draw[red, dotted, opacity = 1] (-2, 0)--(-4, 0);
\draw[dashed] (0, 0 ) [partial ellipse = 170: 165: 5 and 5];
\draw[dashed] (0, 0 ) [partial ellipse = 190: 195: 5 and 5];
\draw (0.1, 0.9) node {\scriptsize $y$};
\draw (-1.9, 0.9) node {\scriptsize $x^c_L(y)$};
\draw (-2.2,-0.2 ) node {\scriptsize $x^c_\Sigma(y)$};
\draw (-3.8,-0.2 ) node {\scriptsize $x^s_\Sigma(y)$};
\draw (-3.8,0.9 ) node {\scriptsize $x^s_L(y)$};
\draw (-5.4, .9 ) node {\scriptsize $x_s^-(y)$};
\draw (-0.4,0.5 ) node {\scriptsize $x_c^-(y)$};
\draw[dashed] (-1.885, 0.7) -- (0,0); 
\draw[dashed] (-3.93, 0.7) -- (0,0); 
\fill[white] (-1.4, 0.5) rectangle (-1.8, 0.1);
\draw[ <-] (0,0)[partial ellipse = 159.60: 180: 1.8 and 1.8];
\draw[ <-] (0,0)[partial ellipse = 170: 180 : 3.8 and 3.8];
\draw (-3.6, 0.3) node {\scriptsize $\eta_s$};
\draw (-1.6, 0.25) node {\scriptsize $\eta_c$};
\end{pgflowlevelscope}
\end{tikzpicture}
}}

\newcommand\figured{{
\begin{tikzpicture}[scale = .5]
\draw (0, 0) circle (1) circle (5);
\draw[dotted] (0, 0) circle (2) circle (4);
\fill[red] (5, 4.2) circle (0.1);
\draw (6.25, 4.5) node {\tiny $y^{(1)}\in\mathcal Y_s^1$};
\fill[red] (5.5, 2.1) circle (0.1);
\draw (6.8, 2.4) node {\tiny $y^{(2)}\in\mathcal Y_s^2$};
\draw[red, ->, >= latex] (-2.7, 4.2) -- (2.7,4.2)-- (0, 4.2);
\draw (0, 4.6) node{\tiny $D^1(y^{(1)}, -e_1)$};
\draw[red] (-4.55, 2.1) -- (-3.39,2.1);
\fill[white] (-3.3, 1.35) rectangle (-4.1, 2.1) ; 
\fill[white] (3.3, 1.35) rectangle (4.1, 2.1) ; 
\fill[white] (2.3, -.5) rectangle (1.1, .45) ; 
\fill[white] (-2.3, -.5) rectangle (-1.1, .35) ; 
\draw (-3.5, 1.7) node{\tiny $x_{s,2}^-(y^{(2)})$} (3.5, 1.7) node{\tiny $x_{s,2}^+(y^{(2)})$};
\draw (-2.2, -0.15) node{\tiny $x_{c,2}^-(y^{(3)})$} (2.79, -0.15) node{\tiny $x_{c,2}^+(y^{(3)})$};
\draw[red] (4.55, 2.1) -- (3.39,2.1);
\draw[red, <-, >= latex] (0,0) [partial ellipse = 148.4: 31.6 : 4 and 4];
\fill[red] (6, -0.6) circle (0.1);
\draw[red] (4.95, -.6) -- (3.95, -.6);
\draw[red] (-4.95, -.6) -- (-3.95, -.6);
\draw[red] (-1.91, -.6) -- (-0.81, -.6) (.81, -.6)  -- (1.91,-.6);
\draw[red, <-, >= latex] (0,0) [partial ellipse = -171.5: -8.5 : 4 and 4];
\draw[red, ->, >= latex] (0,0) [partial ellipse = -162.5: -17.5 : 2 and 2];
\draw (-1, -3.25) node{\tiny $\gamma_s(y^{(3)})$};
\draw (1, -2.2) node{\tiny $\gamma_c(y^{(3)})$};
\draw (7.2, -0.3) node{\tiny $y^{(3)}\in\mathcal Y_s^3$};
\draw[red] (-4.55, 2.1) -- (-3.39,2.1);
\end{tikzpicture}
}}

\newcommand\figureXYZ{{\begin{tikzpicture}
\draw[white] rectangle (4,4) -- (5.5,4);
\draw [->] (0,0) --(0,4.2);
\draw [->] (0,0) --(4.2,0);
\fill  (0,0) rectangle (1,1) (2,0) rectangle (4,1) (3,1) rectangle (4,2) (0,2) rectangle (1,4) (1,3) rectangle (2,4);
\fill (0.9, 2) -- (1,2) -- (2,3) --(2, 3.1)-- (0.9, 3.1) -- cycle;
\fill (2, 1) -- (3, 2)-- (3.1, 2) -- (3.1, 0.9)-- (2,0.9) -- cycle;
\fill [opacity = 0.4] (1,0) rectangle (2,1);
\fill [ opacity = 0.4] (0,1) rectangle (1,2);
\draw  (1.5, 0.5) node {$Y_1$} (0.5, 1.5) node {$Y_2$};
\draw[color=white]  (0.5, 0.5) node{$X_0$} (1, 3) node{$X_2$} (3, 1) node {$X_1$};
\fill [color =white, opacity = 0.8] (1, 1) -- (2, 1) -- (3, 2) -- (4, 2) -- (4,4) -- (2, 4) -- (2, 3) -- (1, 2)-- cycle;
\draw  (2.3, 2.5) node{$W$};
\draw (1, 0) -- (1, -0.2) (2, 0) -- (2, -0.2) (3, 0) -- (3, -0.2);
\draw (1, -0.4) node {$1$} (2,-0.4) node {$2$} (3, -0.4) node {$3$} ;
\draw (0,1) -- ( -0.2,1) ( 0,2) -- (-0.2,2) (0,3) -- (-0.2,3);
\draw ( -0.4, 1) node {$1$} (-0.4, 2) node {$2$} (-0.4, 3) node {$3$};
\draw (0,0) rectangle (4,4);
\draw[<->] (-0.7,3) --(-.7, 4);
\draw (-1.5, 3.5) node {$y\in E_3$};
\draw[<->] (-.7,0) --(-.7, 2);
\draw (-1.5, 1) node {$y\in N_2$};
\draw[<->] (1,-1) --(4, -1);
\draw (2.2, -1.4) node {$x\in E_1$};
\end{tikzpicture}}}

\newcommand\figuresigma{{
\begin{tikzpicture}
\draw[opacity = 0.6] (0,0) [partial ellipse=180:110:4.4cm and 2.8cm];
\draw[opacity = 0.6] (0,0) [partial ellipse=0:70:4.4cm and 2.8cm];
\draw[opacity = 0.6] (0, 2.598) ellipse (1.5 and 0.3);
\draw[dashed,opacity = 0.6] (-1.5, 2.598)-- (-1.5, 0);
\draw[dashed,opacity = 0.6] (1.5, 2.598) --(1.5, 0);
\draw[dashed,opacity = 0.6] (0, 0) ellipse (1.5 and 0.3);
\draw[dashed,opacity = 0.6] (0, 0) [partial ellipse = 0: 180: 4.4 cm and 0.88 cm];
\draw[opacity = 0.6](0, 0) [partial ellipse = 0: -180: 4.4 cm and 0.88 cm];
\fill[ opacity = 0.2, color = red] (1.35, 0.1) -- (4.05, 0.3)  (0,0) [partial ellipse=5.5:70:4.1cm and 2.8cm] --(1.35, 0.1) ;
\fill[ opacity = 0.8, color =red ] (1.35, 0.1) (0,0) [partial ellipse=70: 61:4.1cm and 2.8cm] --  (2, 0.15) --(1.35, 0.1) ;
\fill[ opacity = 0.8, color =red, pattern= north west lines ] (2, 0.15) (0,0) [partial ellipse=55.5: 5.5:3.5cm and 2.4 cm] --  (2, 0.15) ;
\draw (1.1, 1.4) node{$\mathbb D$};
\draw (3.5,0) node{$\Sigma'$};
\draw[thick, -> ] (1.5, 0.2) --++ (0.4, 0.05);
\draw[thick, -> ] (1.5, 0.4) --++ (0.4, 0.05);
\draw[thick, -> ] (1.5, 0.6) --++ (0.4, 0.05);
\draw[thick, -> ] (1.5, 0.8) --++ (0.4, 0.05);
\draw[thick, -> ] (1.5, 1) --++ (0.4, 0.05);
\draw[thick, -> ] (1.5, 1.2) --++ (0.4, 0.05);
\draw[thick, -> ] (1.5, 1.4) --++ (0.4, 0.05);
\draw[thick, -> ] (1.5, 1.6) --++ (0.4, 0.05);
\draw[thick, -> ] (1.5, 1.8) --++ (0.4, 0.05);
\draw[thick, -> ] (1.5, 2.0) --++ (0.4, 0.05);
\draw[thick, -> ] (1.5, 2.2) --++ (0.4, 0.05);
\draw[thick, -> ] (1.5, 2.4) --++ (0.4, 0.05);
\draw[thick, -> ] (1.9, 2.2) --++ (0.4, 0.05);
\draw[thick, -> ] (2.3, 2) --++ (0.4, 0.05);
\draw[thick, -> ] (2.7, 1.75) --++ (0.4, 0.05);
\draw[thick, -> ] (3, 1.3) --++ (0.4, 0.05);
\draw[thick, -> ] (3.45, .8) --++ (0.4, 0.05);
\draw[thick, -> ] (3.8, .4) --++ (0.4, 0.05);
\end{tikzpicture}
}}

\newcommand\figuresigmab{{
\begin{tikzpicture}
\fill[opacity =0.8,color = red] (0,0) -- (0, 3) (0,0) [partial ellipse = 90 : 70 : 3cm and 3cm] --(1,0)-- (0,0);
\fill[opacity = 0.2, color = red] (0, 0) [partial ellipse = 70:0: 3 cm and 3 cm] -- (1,0) -- (1, 1.44);
\draw[thick, -> ] (1, 0.2) --++ (0.4, 0);
\draw[thick, -> ] (1, 0.4) --++ (0.4, 0);
\draw[thick, -> ] (1, 0.6) --++ (0.4, 0);
\draw[thick, -> ] (1, 0.8) --++ (0.4, 0);
\draw[thick, -> ] (1, 1) --++ (0.4, 0);
\draw[thick, -> ] (1, 1.2) --++ (0.4, 0);
\draw[thick, -> ] (1, 1.4) --++ (0.4, 0);
\draw[thick, -> ] (1, 1.6) --++ (0.4, 0);
\draw[thick, -> ] (1, 1.8) --++ (0.4, 0);
\draw[thick, -> ] (1, 2.0) --++ (0.4, 0);
\draw[thick, -> ] (1, 2.2) --++ (0.4, 0);
\draw[thick, -> ] (1, 2.4) --++ (0.4, 0);
\draw[thick, -> ] (1, 2.6) --++ (0.4, 0);
\draw[thick, -> ] (0, 0.2) --++ (-0.4, 0);
\draw[thick, -> ] (0, 0.4) --++ (-0.4, 0);
\draw[thick, -> ] (0, 0.6) --++ (-0.4, 0);
\draw[thick, -> ] (0, 0.8) --++ (-0.4, 0);
\draw[thick, -> ] (0, 1) --++ (-0.4, 0);
\draw[thick, -> ] (0, 1.2) --++ (-0.4, 0);
\draw[thick, -> ] (0, 1.4) --++ (-0.4, 0);
\draw[thick, -> ] (0, 1.6) --++ (-0.4, 0);
\draw[thick, -> ] (0, 1.8) --++ (-0.4, 0);
\draw[thick, -> ] (0, 2.0) --++ (-0.4, 0);
\draw[thick, -> ] (0, 2.2) --++ (-0.4, 0);
\draw[thick, -> ] (0, 2.4) --++ (-0.4, 0);
\draw[thick, -> ] (0, 2.6) --++ (-0.4, 0);
\draw[thick, -> ] (0, 2.8) --++ (-0.4, 0);
\draw[thick, -> ] (0.7, 0.2) --++ (0.2, 0);
\draw[thick, -> ] (0.7, 0.4) --++ (0.2, 0);
\draw[thick, -> ] (0.7, 0.6) --++ (0.2, 0);
\draw[thick, -> ] (0.7, 0.8) --++ (0.2, 0);
\draw[thick, -> ] (0.7, 1) --++ (0.2, 0);
\draw[thick, -> ] (0.7, 1.2) --++ (0.2, 0);
\draw[thick, -> ] (0.7, 1.4) --++ (0.2, 0);
\draw[thick, -> ] (0.7, 1.6) --++ (0.2, 0);
\draw[thick, -> ] (0.7, 1.8) --++ (0.2, 0);
\draw[thick, -> ] (0.7, 2.0) --++ (0.2, 0);
\draw[thick, -> ] (0.7, 2.2) --++ (0.2, 0);
\draw[thick, -> ] (0.7, 2.4) --++ (0.2, 0);
\draw[thick, -> ] (0.7, 2.6) --++ (0.2, 0);
\draw[thick, -> ] (0.3, 0.2) --++ (-0.2, 0);
\draw[thick, -> ] (0.3, 0.4) --++ (-0.2, 0);
\draw[thick, -> ] (0.3, 0.6) --++ (-0.2, 0);
\draw[thick, -> ] (0.3, 0.8) --++ (-0.2, 0);
\draw[thick, -> ] (0.3, 1) --++ (-0.2, 0);
\draw[thick, -> ] (0.3, 1.2) --++ (-0.2, 0);
\draw[thick, -> ] (0.3, 1.4) --++ (-0.2, 0);
\draw[thick, -> ] (0.3, 1.6) --++ (-0.2, 0);
\draw[thick, -> ] (0.3, 1.8) --++ (-0.2, 0);
\draw[thick, -> ] (0.3, 2.0) --++ (-0.2, 0);
\draw[thick, -> ] (0.3, 2.2) --++ (-0.2, 0);
\draw[thick, -> ] (0.3, 2.4) --++ (-0.2, 0);
\draw[thick, -> ] (0.3, 2.6) --++ (-0.2, 0);
\draw[thick, -> ] (0, 3) -- ++ (-0.3, 0.2);
\draw[thick, -> ] (0.22, 2.95) -- ++ (-0.2, 0.3);
\draw[thick, -> ] (0.5, 2.925) -- ++ (0, 0.34);
\draw[thick, -> ] (0.7, 2.885) -- ++ (0.2, 0.3);
\draw[thick, -> ] (0.85, 2.85) -- ++ (0.35, 0.1);
\draw[thick, -> ] (0.4, 2.65) -- ++ (-0.15, 0.15);
\draw[thick, -> ] (0.6,2.65) -- ++ (0.15, 0.15);
\end{tikzpicture}
}}

\newcommand\figuredz{
\begin{tikzpicture}
\draw[white] (-.5, 0) rectangle (1.8, 2.8);
\begin{pgflowlevelscope}{\pgftransformscale{0.6}}
\draw (0,0) [partial ellipse = 0: 180 : 1 and 1];
\draw (0,0) [partial ellipse = 0: 180 : 5 and 5];
\draw[dotted] (30: 1) -- (30:5);
\draw[dotted] (60: 1) -- (60:5);
\draw (25: 3) node{$\Sigma'$};
\draw (65: 3) node{$\Sigma'^+$};
\draw (-1, 1) node{$\partial N_1$};
\draw (130: 4.52) node{$\partial N_1$};
\draw[dashed] (0.86, 0.5)-- (1.94, 0.5);
\draw[dashed] (0, 0) [partial ellipse = 15: 0 :2 and 2 ] (0,0) [partial ellipse = 0:7.5:4 and 4];
\draw[dashed] (3.97, 0.5)-- (4.97, 0.5);
\draw[dashed] (2.5,4.33) -- (-2.5, 4.33);
\fill (0.86,0.5) circle (0.1); 
\draw (0.61, 0.25) node {$z_1^0$};
\fill (2.5, 4.33) circle (0.1);
\draw (3, 4.6) node {$(a_1^0)^+$};
\end{pgflowlevelscope}
\end{tikzpicture}
}

\title{Generalized Bott-Cattaneo-Rossi invariants in terms of Alexander polynomials}
	\author{David Leturcq\footnote{Institut Fourier, Université-Grenoble-Alpes (now JSPS Fellow at RIMS, Kyoto university)}}
	\date{ }
\begin{document}
\maketitle
\begin{abstract}
The Bott-Cattaneo-Rossi invariant $(Z_\degk)_{\degk\in \mathbb N\setminus\{0,1\}}$ is an invariant of long knots $\R^n\hookrightarrow\R^{n+2}$ for odd $n$, which reads as a combination of integrals over configuration spaces.
In this article, we compute such integrals and prove explicit formulas for (generalized) $Z_k$ in terms of Alexander polynomials, or in terms of linking numbers of some cycles of a hypersurface bounded by the knot. Our formulas, which hold for all null-homologous long knots in homology $\R^{n+2}$ at least when $n\equiv 1\mod 4$, conversely express the Reidemeister torsion of the knot complement in terms of $(Z_\degk)_{\degk\in\mathbb N\setminus\{0,1\}}$. 
Our formula extends to the even-dimensional case, 
where $Z_k$ will be proved to be well-defined in an upcoming article.
\end{abstract}
\section{Introduction}

\textbf{Invariants from configuration space integrals }

Knot invariants defined as combinations of integrals over configuration spaces
emerged after the seminal work of Witten \cite{[Wit]} on the perturbative
expansion of the Chern-Simons theory.
Knot invariants defined from spatial configurations of unitrivalent graphs were formally
defined by Guadagnini, Martellini and Mintchev \cite{GMM}, Bar-Natan \cite{BarNatanPCST}, Altschüler and Freidel \cite{Altschuler_1997},
Bott and Taubes \cite{Bott2017}, and others,
for knots in $\R^3$.
Their definition applies to \emph{long knots in $\R^3$}, which are embeddings of $\R$ into $\R^3$
that coincide with the vertical line outside a ball.
Poincaré duality allowed Dylan Thurston \cite{thurston1999integral} and Sylvain Poirier \cite{Poirier_2002} to see the mentioned combinations of
configuration space integrals as algebraic counts of configurations
of uni-trivalent graphs such that the univalent vertices lie on the knot 
and such that the edge directions are in some generic
finite set.
Work of Kontsevich \cite{Kontsevich1994}, Kuperberg 
and Thurston \cite{kuperberg1999perturbative} 
allowed Lescop \cite{[Lescop],[Lescop2]} to
generalize these constructions to knots in homology $\R^3$, where the
constraint "an edge must have a direction in some finite set" is
replaced with
"the endpoints of an edge must lie in some propagator in a finite
set of propagators", where a \emph{propagator} is a $4$-dimensional chain of
the two-point configuration space of the ambient manifold that
satisfies some boundary conditions.
The above knot invariants thus become combinations of algebraic intersections of preimages
of propagators in some configuration spaces.

\textbf{(Generalized) Bott-Cattaneo-Rossi invariants}

In \cite{[Bott]}, Bott introduced a similar isotopy invariant $Z_2$ of knots $\s^n\hookrightarrow \R^{n+2}$, when $n$ is odd and $n\geq3$. 
The invariant $Z_2$ reads as a linear combination of configuration space integrals associated to graphs 
by integrating a wedge product over the edges of pull-backs under Gauss maps of volume forms on the spheres $\s^{n\pm1}$, 
where such a Gauss map is the direction of an edge in $\R^n$ or in $\R^{n+2}$. The involved graphs have 
four vertices of two kinds, and four edges of two kinds.

This invariant was generalized to a whole family $(Z_{k})_{k\in\mathbb N\setminus\{0,1\}}$ of isotopy invariants of long knots $\R^n\hookrightarrow \R^{n+2}$, for odd $n\geq3$, by Cattaneo and Rossi in \cite{Cattaneo2005} and by Rossi in his thesis \cite{[Rossi]}. The invariant $Z_k$ involves graphs with $2k$ vertices of two kinds and $2k$ edges of two kinds.

In \cite{article1}, we introduced more flexible definitions for the invariants $Z_k$ 
where propagating forms dual to propagators as above replace pull-backs via Gauss maps dual to direction constraints. 
Our definitions allowed us to generalize these invariants in the larger setting of long knots inside
\emph{asymptotic (integral) homology}\footnote{These asymptotic homology $\R^{n+2}$ are precisely defined in the beginning of Section \ref{S11}.} $\R^{n+2}$ 
when $n$ is odd $\geq3$. 
They also allow us to compute $Z_k$ as a combination of algebraic intersections of preimages of propagators under restriction maps.
In \cite{article3}, we notice that, for any $k\in\mathbb N\setminus\{0,1\}$, our definition of $Z_k$ is also valid for $n=1$ 
and for long knots in rational homology $\R^3$. 
To this end, we prove that $Z_k$ coincides with the composition
$Z_k = - w'_C \circ \mathcal Z_k$ for long knots in asymptotic rational homology $\R^3$, where the invariant $\mathcal Z_k$ is the perturbative expansion of the Chern-Simons theory, as defined in \cite{[Lescop],[Lescop2]} and where the 
weight system $w'_{C}$ coincides with the Conway weight system $w_C$ defined by Bar-Natan and 
Garoufalidis in \cite{BNG} on the connected unitrivalent diagrams and vanishes on trivalent diagrams and on non-trivial 
products of diagrams.
In another upcoming article \cite{article4}, we prove that the generalized Bott-Cattaneo-Rossi invariants are also 
well-defined for long knots of parallelizable asymptotic homology $\R^{n+2}$ when $n\geq 2$ is even. 

\textbf{Results of this article}

Here, we construct suitable propagators associated to hypersurfaces bounded by the knot
and we exactly compute 
the involved algebraic intersection numbers.
Our computations lead us to explicit formulas for the generalized $Z_k$ for all null-homologous long knots (at least) when 
$n\not\equiv3\mod4$ in terms of linking numbers of some cycles of a hypersurface as above (Corollary \ref{Zklemma}).

These formulas allow us to prove the following main theorem of this article, 
which holds for the virtually rectifiable long knots of Definition \ref{rectifiabilitydef}.
For even-dimensional knots, the consistency of the definition of $Z_k$ relies on the result of \cite{article4}.
\footnote{Regardless of these results, the theorem below holds for even-dimensional long knots with $Z_k^{F_*}(\psi)-Z_k^{F_*^0}(\psi_0)$ instead of $Z_k(\psi)$, where $\psi_0$ is the trivial knot, and $F_*$ and $F_*^0$ are special family of propagators. See Section \ref{Section2} for more details.}

\begin{thmnonnum}[Theorem \ref{Reidth}]
Let $\mathbb K$ denote $\mathbb Z$ if $n\geq2$, and $\mathbb Q$ if $n=1$.
Let $\psi$ be a null-homologous long knot of an asymptotic $\mathbb K$-homology $\R^{n+2}$.
If $\psi$ is virtually rectifiable (in particular if $n\equiv 1\mod 4$, 
or if $n$ is even and if $\spamb$ is parallelizable), we have the following equality in $\QQ[[h]]$ : 
\[\sum\limits_{\degk\geq 2}Z_{\degk}(\psi)h^{\degk} = 
(-1)^n\Ln\left(\torsion(e^{h})\right)
=
\sum\limits_{\dimd=1}^{n}(-1)^{n+d+1}\left(
\Ln\left(\Delta_{\dimd,\Sigma}(e^{h})\right)- \Delta_{d,\Sigma}'(1)h
\right),\]
where $\torsion(t)$ is the Reidemeister torsion, as normalized in Definition \ref{Reidemeister}, and the $(\Delta_{d, \Sigma}(t))_{d\in\{1,\ldots, n\}}$ are the Alexander polynomials of Theorem \ref{Levine-th}.
\end{thmnonnum}

This formula determines the invariants $Z_k$ as explicit functions of the Alexander polynomials of the knot. To our knowledge, our induced explicit determination of $(Z_k)_{k\geq2}$ is the first complete computation of an invariant defined from spatial configuration space integrals in degree higher than five.
In \cite[Corollary 4.9]{[Watanabe]}, Watanabe proved that the BCR invariants are finite type invariants 
with respect to some operations on \emph{long ribbon knots} of $\R^{n+2}$ for odd $n\geq3$.
His study allowed him to prove 
that BCR invariants $Z_{k}$ are not trivial for $k$ even $\geq2$ and $n$ odd or for $(k,n)=(3,2)$, and that they are related to the Alexander 
polynomial for long ribbon knots of $\R^{n+2}$. He obtained an exact formula for $Z_2$ in terms of the 
Alexander polynomial for any long ribbon knot. 
The Watanabe expression of $Z_k$ for long ribbon knots was found by using the 
Habiro-Kanenobu-Shima finite type invariant theory of \cite{[HKS]}.
It contains some indeterminacies when $k\geq4$.
Our formula for $(Z_k)_{k\geq2}$ extends this result (\cite[Corollary 4.9]{[Watanabe]}) 
and its analogue (\cite[Theorem 6.2]{[Watanabe2]}) for long handle $(p,q)$-knots
to virtually rectifiable long knots, where no finite type invariant theory is known, 
and it does not contain indeterminacies anymore. 
In particular, it proves \cite[Conjecture 6.3]{[Watanabe2]} and it
partially answers \cite[Problems 6.4 \& 6.5]{[Watanabe2]}: 
the formula 
\[\exp\left(\sum\limits_{k\geq2} Z_k(\psi) h^k\right)
= \torsion(t=e^h)^{(-1)^n} = f(t)\]
 defines a rational fraction $f(t)\in \mathbb Q(t)$ such that $f(t^{-1}) = (f(t))^{(-1)^{n-1}}$, $f'(1)=0$ and $f(1)=1$.

\textbf{An application : the classical Alexander polynomial in terms of the perturbative expansion of the Chern-Simons theory}

In \cite{article3}, 
our formula and \cite[Theorem 2.13]{article3} allow us to prove\footnote{This is \cite[Corollary 2.17]{article3}.} 
\[ \sum\limits_{k\geq 1} \left(w'_{C} \circ \mathcal Z_k\right) h^k = \Ln(\Delta_\psi(e^h)) 
\] for any null-homologous long knot of an asymptotic rational homology $\R^3$, 
where $\Delta_\psi(t)$ is the Alexander polynomial and $\mathcal Z = (\mathcal Z_k)_{k\geq0}$ is 
the forementioned invariant 
defined in \cite{[Lescop],[Lescop2]}. 
This generalizes a result
proved by Bar-Natan and Garoufalidis \cite{BNG}\footnote{
The result of \cite{BNG} applies to the Kontsevich integral $Z_{K}$ 
instead of $\mathcal Z$ but these two invariants are connected by a result of Lescop \cite{LesJKTR}. In \cite[Proposition 2.18]{article3}, we proved that 
this result
implies that $w'_C\circ \mathcal Z_k=w'_C\circ Z_{K,k}$ in any degree $k$.
} for knots of $\s^3$, using
crossing change formulas.
Our proof extends to null-homologous long knots in asymptotic rational homology $\R^3$, 
where crossing change formulas (and the value on the unknot) do not suffice to determine a knot invariant. In degree $2$, this result was proved by Lescop in \cite[Theorem 17.36]{[Lescop2]} for usual closed knots in rational homology spheres.

\textbf{Structure of this article}

In Section \ref{S.1}, we recall the definition\footnote{In the odd dimensional case, this is a review of \cite[Section 2]{article1}.} of the invariant $\Z[]$ in terms of intersection numbers of preimages of \emph{propagators}, 
where propagators are special chains in the two-point configuration space of the ambient asymptotic homology $\R^{n+2}$, which are presented in Definition \ref{df-prop}. 
We precisely state all the forementioned theorems in this section.
Section \ref{Section2} describes how to obtain the formula for $\Z[]$ in terms of linking numbers (Theorem \ref{th0}), for the null-homologous \emph{rectifiable knots} of Definition \ref{rectifiabilitydef}, using propagators related to some 
hypersurface bounded by the knot.
The details of the construction of such propagators for any null-homologous rectifiable knot are presented in Section \ref{Section 3}.
In Section \ref{Section6}, we derive the formula of Theorem \ref{Reidth} for the Reidemeister torsion from Corollary \ref{Zklemma}.
In Section \ref{Section4}, we prove that any odd-dimensional long knot is rectifiable up to connected sum with three copies of itself, when $n\equiv 1\mod 4$, and that any even-dimensional long knot in a parallelizable asymptotic integral homology $\R^{n+2}$ is rectifiable up to connected sum with one copy of itself.

I thank my advisor Christine Lescop for her help with the redaction of this article.
\section{Definition of \texorpdfstring{$(Z_k)_{k\geq2}$}{Zk} and main statements}\label{S.1}

In this article, $n$ is a positive integer. Let $\mathbb K$ denote $\mathbb Z$ if $n\geq2$, or $\mathbb Q$ if $n=1$. 
\subsection{Parallelized asymptotic homology \texorpdfstring{$\R^{n+2}$}{Rn+2} and long knots}\label{S11}

 Let $M$ be an $(n + 2)$-dimensional closed smooth oriented manifold, 
 such that $H_*(M; \mathbb K) = H_*(\s^{n+2}; \mathbb K)$. 
 Such a manifold is called a \emph{$\mathbb K$-homology $(n+2)$-sphere}. 
 In such a homology sphere, choose a point $\infty$ and a closed ball $B_\infty(\sphereambiante)$ around this point. 
 Fix an identification of this ball $B_\infty(M)$ with the complement 
 $B_\infty$ of the open unit ball of $\R^{n+2}$ in $\s^{n+2}=\R^{n+2}\cup \{\infty\}$, 
 such that this smooth identification extends 
 from a neighborhood of $B_\infty(M)$ to a neighborhood of $B_\infty$ in $\s^{n+2}$. 
 Let $\ambientspace$ denote the manifold $\sphereambiante \setminus\{\infty\}$ and 
 let $\voisinageinfini(\sphereambiante)$ denote the punctured ball 
 $B_\infty(\sphereambiante)\setminus\{\infty\}$, which is identified with 
 the complement $\voisinageinfini$ of the open unit ball in $\R^{n+2}$. 
 Let $\bouleambiante$ denote the closure of $\ambientspace\setminus\voisinageinfini$, 
 so that the manifold $\ambientspace$ can be seen as 
 $\ambientspace = \bouleambiante \cup \voisinageinfini$, where $\voisinageinfini\subset \R^{n+2}$. 
 The manifold $\punct M$ endowed with the decomposition 
 $\ambientspace = \bouleambiante \cup \voisinageinfini$ is called 
 an \emph{asymptotic $\mathbb K$-homology $\R^{n+2}$}.

\emph{Long knots} of such a space $\ambientspace$ are smooth embeddings $\psi\colon \R^n\hookrightarrow \ambientspace$ such that $\psi(x)=(0,0,x)\in \voisinageinfini$ when $||x||\geq 1$, and $\psi(x) \in \bouleambiante$ when $||x||\leq 1$.

\begin{df}\label{paral-def}
A \emph{parallelization} of an asymptotic $\mathbb K$-homology $\R^{n+2}$ is a 
smooth bundle isomorphism 
$\tau\colon \ambientspace\times\R^{n+2}\rightarrow T\ambientspace$ 
that coincides with the canonical trivialization 
$\tau_0\colon \R^{n+2}\times\R^{n+2}\rightarrow T\R^{n+2}$ of $T\R^{n+2}$ on $\voisinageinfini \times \R^{n+2}$. 
An asymptotic $\mathbb K$-homology $\R^{n+2}$ with such a parallelization is 
called a \emph{parallelized asymptotic $\mathbb K$-homology $\R^{n+2}$}. 
An asymptotic $\mathbb K$-homology $\R^{n+2}$ that admits such a 
parallelization is called \emph{parallelizable}. 
Given a parallelization $\tau$ and a point $x\in \ambientspace$, $\tau_x$ 
denotes the isomorphism 
$\tau(x,\cdot)\colon \R^{n+2} \rightarrow T_x\ambientspace$.
\end{df}

Some partial results of parallelizability are given in Proposition \ref{conn-sum2}.

\subsection{BCR diagrams}
In this section, we describe the diagrams involved in the definition of the 
invariant $Z_k$, which are the BCR diagrams of \cite[Section 2.2]{article1}.

\begin{df}\label{Def-BCR}
A \emph{BCR diagram} is an oriented connected graph $\Gamma$, defined by a set 
$\sommets$ of vertices, decomposed into 
$\sommets = \sommetsinternes\sqcup\sommetsexternes$, 
and a set $\aretes$ of ordered pairs of distinct vertices, 
decomposed into $\aretes = \aretesinternes\sqcup\aretesexternes$, 
whose elements are called \emph{edges}\footnote{In particular, our graphs have neither loops nor multiple edges with the same orientation.}, 
where the elements of $\sommetsinternes$ are called \emph{internal vertices}, 
those of $\sommetsexternes$ \emph{external vertices}, 
those of $\aretesinternes$ \emph{internal edges}, 
and those of $\aretesexternes$ \emph{external edges}, and such that, 
for any vertex $v$ of $\Gamma$,  one of the five following properties holds: 
\begin{enumerate}
\item $v$ is external, 
with two incoming external edges and one outgoing external edge, 
and exactly one of the incoming edges comes from a univalent vertex. 
\item $v$ is internal and trivalent, 
with one incoming internal edge, one outgoing internal edge, 
and one incoming external edge, which comes from a univalent vertex.
\item $v$ is internal and univalent, with one outgoing external edge. 
\item $v$ is internal and bivalent, 
with one incoming external edge and one outgoing internal edge.
\item $v$ is internal and bivalent, 
with one incoming internal edge and one outgoing external edge.
\end{enumerate}

\end{df}
In the following, internal edges are depicted by solid arrows, 
external edges by dashed arrows, 
internal vertices by black dots, 
and external vertices by white dots, as in Figure \ref{BCR4}, 
where the five behaviors of Definition \ref{Def-BCR} appear.

\begin{figure}[H]
\centering
\begin{tikzpicture}
\fill (0:1) \crc (45:1) \crc (0:2) \crc (90:2) \crc (135:2) \crc (180:1) \crc (225:1) \crc (225:2) \crc (270:1) \crc (315:1) \crc;
\draw (90:1) \crc (135:1) \crc;
\draw [<-, >=latex] (45:1) ++(-67.5:.1) -- ++(-67.5:.6); 
\draw [<-, >=latex] (0:1) ++(-112.5:.1) -- ++(-112.5:.6); 
\draw [<-, >=latex, dashed] (-45:1) ++(-157.5:.1) -- ++(-157.5:.6); 
\draw [<-, >=latex] (-90:1) ++(-201.5:.1) -- ++(-201.5:.6); 
\draw [<-, >=latex] (-135:1) ++(-246.5:.1) -- ++(-246.5:.6); 
\draw [<-, >=latex, dashed] (180:1) ++(67.5:.1) -- ++(67.5:.55); 
\draw [<-, >=latex, dashed] (135:1) ++(22.5:.1) -- ++(22.5:.55); 
\draw [<-, >=latex, dashed] (90:1) ++(-22.5:.1) -- ++(-22.5:.6); 
\edge (0:1.9) -- (0:1.1); 
\edge (90:1.9) -- (90:1.1); 
\edge (135:1.9) -- (135:1.1); 
\edge (-135:1.9) -- (-135:1.1); 
\end{tikzpicture}
\caption{A BCR diagram of degree 6}\label{BCR4}
\end{figure}
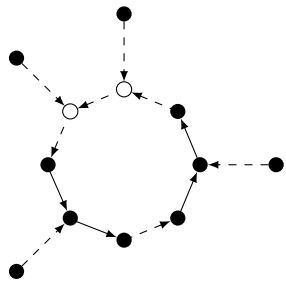
%\begin{figure}[H]
%\centering
%\begin{tikzpicture}
%\draw (90:0.7) node {$1$};
%\draw (0.265,0) node {$v$};
%\draw (0, 0) circle (0.1) ;
%\draw [->, >= latex, dashed](-90: 1)-- (-90: 0.1) ;
%\draw [<-, >= latex, dashed](40: 1)-- (40: 0.1) ;
%\draw [->, >= latex, dashed](140: 1)-- (140: 0.1) ;
%\fill (140:1.1) circle (0.1);
%\end{tikzpicture} \ \ \ \ \ \ \ \ 
%\begin{tikzpicture}
%\draw (90:0.7) node {$2$};
%\fill (0, 0) circle (0.1) ;
%\fill (140:1.1) circle (0.1);
%\draw (0.265,0) node {$v$};
%\draw [->, >= latex, thick](-90: 1)-- (-90: 0.1) ;
%\draw [<-, >= latex, thick](40: 1)-- (40: 0.1) ;
%\draw [->, >= latex, dashed](140: 1)-- (140: 0.1) ;
%\end{tikzpicture} \ \ \ \ \ \ \ \ 
%\begin{tikzpicture}
%\draw (90:0.7) node {$3$};
%\fill (0, 0) circle (0.1) ;
%\draw (0.265,0) node {$v$};
%\draw [<-, >= latex, dashed](-90: 1)-- (-90: 0.1) ;
%\end{tikzpicture} \ \ \ \ \ \ \ \ 
%\begin{tikzpicture}
%\draw (90:0.7) node {$4$};
%\fill (0, 0) circle (0.1) ;
%\draw (0.265,0) node {$v$};
%\draw [->, >= latex, dashed](-90: 1)-- (-90: 0.1) ;
%\draw [<-, >= latex, thick](40: 1)-- (30: 0.1) ;
%\end{tikzpicture}  \ \ \ \ \ \ \ 
%\begin{tikzpicture}
%\draw (90:0.7) node {$5$};
%\fill (0, 0) circle (0.1) ;
%\draw (0.265,0) node {$v$};
%\draw [->, >= latex, thick](-90: 1)-- (-90: 0.1) ;
%\draw [<-, >= latex, dashed](40: 1)-- (30: 0.1) ;
%\end{tikzpicture}
%\caption{The five possible behaviors near a vertex of a BCR diagram}\label{fig-BCR}
%\end{figure}
Definition \ref{Def-BCR} implies that any BCR diagram consists of 
one cycle with some legs attached to it, 
where \emph{legs} are external edges that come from a (necessarily internal) univalent vertex, 
and where the graph is a cyclic sequence of pieces as in Figure \ref{fig-BCR2} 
with as many pieces of the first type than of the second type. 
In particular, any BCR diagram has an even number of vertices.

\begin{figure}[H]
\centering
\begin{tikzpicture}[yscale = -1]
\draw[white] (0, 1) circle(0.1);
\draw [->, dashed, >= latex] (-0.5, 0)-- (-0.1,0);
\fill (0,0) circle (0.1) ;
\draw  (.1, 0)-- (0.5,0);
\end{tikzpicture}\ \ \ \ \ \ \ \ \ \ 
\begin{tikzpicture}[yscale = -1]
\draw[white] (0, 1) circle(0.1);
\draw [->, >= latex] (-0.5, 0)-- (-0.1,0);
\fill (0,0) circle (0.1);
\draw [dashed] (.1, 0)-- (0.5,0);
\end{tikzpicture}\ \ \ \ \ \ \ \ \ \ 
\begin{tikzpicture}[scale=-1]
\fill (1,0) circle (0.1) (3,0) circle (0.1);
\fill (1,1) circle (0.1) (3,1) circle (0.1);
\draw [->, dashed,>= latex] (1, 0.9)-- (1,0.1);
\draw [->,dashed,  >= latex] (3, 0.9)-- (3,0.1);
\draw [ >= latex] (.5, 0)-- (0.9,0);
\draw [<-,  >= latex] (1.1, 0)-- (1.7,0);
\draw [<-, >= latex] (2.3, 0)-- (2.9,0);
\draw [dotted] (1.7, 0)-- (2.3,0);
\draw [<-,  >= latex](3.1, 0)-- (3.5,0);
\end{tikzpicture}\ \ \ \ \ \ \ \ \ \ 
\begin{tikzpicture}[scale=-1]
\draw (1,0) circle (0.1) (3,0) circle (0.1) ;
\fill (1,1) circle (0.1) (3,1) circle (0.1) ;
\draw [->,dashed, >= latex] (1, 0.9)-- (1,0.1);
\draw [->,dashed, >= latex] (3, 0.9)-- (3,0.1);
\draw [<-,dashed, >= latex] (.5, 0)-- (0.9,0);
\draw [>=latex, <-, dashed] (1.1, 0)-- (1.7,0);
\draw [<-, dashed,>= latex] (2.3, 0)-- (2.9,0);
\draw [dotted] (1.7, 0)-- (2.3,0);
\draw [dashed, <-, >=latex] (3.1, 0)-- (3.5,0);
\end{tikzpicture}
\caption{ }\label{fig-BCR2}
\end{figure}

For any positive integer $k$, set $\und k = \{1, \ldots, k\}$. 
\begin{df}\label{Def-numb}
Define the \emph{degree} of a BCR diagram $\Gamma$ as 
$\mathrm{deg}(\Gamma) = \frac12\mathrm{Card}(\sommets)$, 
and let $\graphes$ denote the set of all BCR diagrams of degree $k$. 
Note that a degree $k$ BCR diagram has $2k$ edges.

A \emph{numbering} of a BCR diagram $\Gamma$ of degree $k$ is 
a bijection $\sigma \colon \aretes\rightarrow \und {2k}$. 
A \emph{numbered BCR diagram} is a pair $(\Gamma, \sigma)$ 
where $\Gamma$ is a BCR diagram and $\sigma$ is a numbering of $\Gamma$. 
Let $\graphesnum$ denote the set of numbered BCR diagrams up to numbered graph isomorphisms.
\end{df}

An \emph{orientation} of a finite set $X$ is the data of a total order on $X$, up to an even permutation of $X$.
If $o$ is an orientation, $-o$ denote the opposite orientation. 
Note the following easy lemma.
\begin{lm}\label{orientationar}

Given a degree $k$ BCR diagram $\Gamma$, let $E_\theta(\Gamma)$ denote the set of edges which are
either 
an external edge going to an internal vertex, or an external edge of the cycle going to an external vertex.
For such an edge, define the edge $e'$ as follows:
\begin{itemize}
\item if $e$ is going to an internal vertex $v$, $e'$ is the internal edge going out of $v$,
\item if $e$ is an external edge of the cycle going to an external vertex $v$, 
$e'$ is the leg going to $v$.
\end{itemize}
We have $\Card(E_\theta(\Gamma)) = k$. Furthermore,
given any total order $(e_1, \ldots, e_k)$ on $E_\theta(\Gamma)$, 
$(e_1, e'_1, e_2, e'_2, \ldots, e_k, e'_k)$ is a total order on $\aretes$.
The orientation induced on $\aretes$ does not depend on the choice of the total order on $E_\theta(\Gamma)$, and is called the \emph{canonical orientation} of $\aretes$.
\end{lm}

\subsection{One-point and two-point configuration spaces}

Let $X$ be a $d$-dimensional closed smooth oriented manifold, 
let $\infty$ be a point of $X$, and set $\punct X= X\setminus\{\infty\}$. 
We give a short overview of a compactification $C_2(\punct X)$ of 
the two-point configuration space 
$C_2^0(\punct X)=\{(x,y) \in (\punct X)^2 \mid x\neq y\}$, 
as defined in \cite[Section 2.2]{[Lescop]}.

If $P$ is a submanifold of a manifold $Q$, such that $P$ is transverse to $\partial Q$ and $\partial P= P\cap \partial Q$, 
its \emph{normal bundle} $\mathfrak NP$ is the bundle whose fibers are $\mathfrak N_xP = T_xQ/T_xP$. 
A fiber $U\mathfrak N_xP$ of the \emph{unit normal bundle $U\mathfrak NP$ of $P$} 
is the quotient of $\mathfrak N_xP \setminus\{0\}$ by the dilations\footnote{Dilations are homotheties with positive ratio.}.  

Here, we use the blow-up in differential topology, 
which replaces a compact submanifold $P$ of a compact manifold $Q$ as above 
with its unit normal bundle $U\mathfrak NP$. 
The obtained manifold is a smooth compact manifold with boundary and ridges. 
It is diffeomorphic to the complement in $Q$ of an open tubular neighborhood 
of $P$. Its interior is $Q\setminus(\partial Q\cup P)$, 
and its boundary is $U\mathfrak NP\cup (\partial Q\setminus  \partial P)$ as a set.

Define the \emph{one-point configuration space} $C_1(\punct X)$ as the blow-up of $X$ at $\{\infty\}$. 
It is a compact manifold with interior $\punct X$ and 
with boundary the unit normal bundle $\s^{d-1}_\infty X$ to $X$ at $\infty$. 

Blow up the point $(\infty,\infty)$ in $X^2$. 
In the obtained manifold, blow up the closures of the sets $\{\infty\}\times\punct X$, $\punct X\times \{\infty\}$ and $\Delta_{\punct X}=\{(x,x)\mid x\in\punct  X\}$.

The obtained manifold $C_2(\punct X)$ is compact and 
it comes with a canonical map $p_b\colon\config(\punct X)\rightarrow X^2$. 
Its interior is canonically diffeomorphic to the open configuration space $C_2^0(\punct X) = \{(x,y) \in (\punct X)^2 \mid x\neq y\}$, 
and $C_2(\punct X)$ has the same homotopy type as $C_2^0(\punct X)$. 
The manifold $C_2(\punct X)$ is called the \emph{two-point configuration space of $\punct X$}.
Its boundary is the union of\begin{itemize}
\item the closed part $p_b^{-1}(\{(\infty,\infty)\})$,
\item the unit normal bundles to $\punct X \times\{\infty\}$ 
and $\{\infty\}\times\punct X$, which are $\punct X \times \s_\infty^{d-1}X$
and $\s_\infty^{d-1}X\times \punct X$, 
\item the unit normal bundle to the diagonal $\Delta_{\punct X}$, 
which is identified with the unit tangent bundle $\unitaire\punct X$ 
via the map $[(u,v)]_{(x,x)}\in U\mathfrak N_{(x,x)}\Delta_{\punct X} \mapsto \left[v-u\right]_x\in \unitaire_x\punct X$.\end{itemize}

The following lemma can be proved as \cite[Lemma 2.2]{[Lescop]}, which is the case $d=3$.
\begin{lm}\label{Gauss}
When $\punct X=\R^d$, the Gauss map 
\[\begin{array}{lll} C_2^0(\R^d) & \rightarrow & \s^{d-1}\\ (x,y) & \mapsto & \frac{y-x}{||y-x||}\end{array}\] extends to a smooth map $G\colon C_2(\R^d)\rightarrow \s^{d-1}$. 
\end{lm}

We now define an analogue of $G$ on the boundary of $\configM$ 
for any parallelized asymptotic $\mathbb K$-homology $\R^{n+2}$.

\begin{df}\label{configM}Let $(\ambientspace, \tau)$ be a parallelized asymptotic $\mathbb K$-homology $\R^{n+2}$. 
Identify the sphere $\s^{n+1}_\infty M$ with $\s^{n+1}$ in such a way that 
$u\in\s^{n+1}$ is the limit when $t$ approaches $+\infty$ of the map 
$\left(t\in \left[1, +\infty\right[ \mapsto t.u\in \voisinageinfini\subset\R^{n+2}\right)$.
The boundary of $\configM$ is the union of \begin{itemize}
\item the closed part $\partial_{\infty,\infty}\configM = p_b^{-1}(\{\infty\times\infty\})$, 
which identifies with the similar part $\partial_{\infty, \infty}\config(\R^{n+2})\subset\config(\R^{n+2})$,
\item an open\footnote{as a subset of $\partial\configM$.} face $\partial_{\infty, \ambientspace} \configM =p_b^{-1}(\{\infty\}\times\ambientspace)= \s^{n+1}_\infty M\times \ambientspace=\s^{n+1}\times\ambientspace$,
\item an open face $\partial_{\ambientspace, \infty}\configM = p_b^{-1}(\ambientspace\times\{\infty\})= \ambientspace\times\s^{n+1}$,
\item an open face $\partial_\Delta\configM = p_b^{-1}(\Delta_{\ambientspace}) = \unitaire\ambientspace$.
\end{itemize}

Define the smooth map $G_\tau \colon\partial\configM\rightarrow\s^{n+1}$ 
by the following formula: 
\[G_\tau(c) =\left\{\begin{array}{lll} 
G(c) &\text{if $c\in \partial_{\infty,\infty}\configM= \partial_{\infty,\infty}C_2(\R^{n+2})$,}\\
 -u &  \text{if $c= (u, y)\in\partial_{\infty,\ambientspace}\configM = \s^{n+1}\times \ambientspace$,}\\
u & \text{if $c=(x, u)\in\partial_{\ambientspace,\infty}\configM=\ambientspace \times \s^{n+1}$,}\\
\frac{\tau_x^{-1}(u)}{||\tau_x^{-1}(u)||} & \text{if  $c= [u]_x\in \unitaire_x\ambientspace\subset\unitaire\ambientspace=\partial_{\Delta}\configM$.}
\end{array}\right.\]
In order to simplify the notations, 
for any configuration in one of the three above open faces, 
we write $c= (x, y, u)$ where $(x, y) = p_b(c)$, 
and $u$ denotes the coordinate in the previous definition, 
which is either in $\s^{n+1}$ or in $U_x\punct M$.
\end{df}

\subsection{Conventions about orientations}
From now on, homology groups are taken with rational coefficients 
unless otherwise mentioned, all the manifolds are oriented, 
and their boundaries are oriented with the "outward normal first" convention. 
The ordered products of manifolds are naturally oriented, 
and this orientation does not depend on the order if the manifolds are even-dimensional. 
The fibers of the normal bundle of an oriented submanifold $P$ 
of a manifold $Q$ are oriented in such a way that the orientation of $\N_xP$ 
followed by the orientation of $T_xP$ is the orientation of $T_xQ$. 
The orientation of $\N_x P$ is called the \emph{coorientation} of $P$.
The preimages of submanifolds are oriented in such a way that 
coorientations are preserved. 
The intersection $A_\cap=\bigcap\limits_{i=1}^r A_r$ of submanifolds is 
oriented in such a way that $\mathfrak NA_\cap$ is oriented as 
$\bigoplus\limits_{i=1}^r\mathfrak NA_i$. 
If $A$ is an oriented manifold, $-A$ denotes the same manifold, 
with opposite orientation.

\subsection{Configuration spaces}\label{confspace}

Let $\Gamma$ be a BCR diagram, 
and let $\ambientspace$ be an asymptotic $\mathbb K$-homology $\R^{n+2}$. 
Fix a long knot $\psi\colon \R^n \hookrightarrow \ambientspace$. 
Let $\confignoeud[\Gamma][\psi][^0]$ denote the open \emph{configuration space} 
\[\confignoeud[\Gamma][\psi][^0]= 
\{c\colon \sommets \hookrightarrow \ambientspace 
\mid \text{There exists }c_i \colon \sommetsinternes\hookrightarrow \R^n 
\text{ such that } c_{|\sommetsinternes} = \psi\circ c_i
\}.\]

An element $c$ of $\confignoeud[\Gamma][\psi][^0]$ is called a \emph{configuration}. 
Note that $c_i$ is uniquely determined by $c$. 
By definition, the images of the vertices under a configuration are pairwise distinct, 
and the images of the internal vertices are on the knot.

This configuration space is a non-compact smooth manifold. 
It admits a compactification $\confignoeud$, 
which is defined in \cite[Section 2.4, pp. 51-61]{[Rossi]}.

\begin{theo}[Rossi]\label{cgamma}There exists a smooth compact manifold $\confignoeud$ with boundary and ridges such that: 
\begin{itemize}
\item the interior of $\confignoeud$ is canonically diffeomorphic to $\confignoeud[\Gamma][\psi][^0]$,
\item for any two internal vertices $v$ and $w$, the map $\big(c\in \confignoeud[\Gamma][\psi][^0] \mapsto (c_i(v), c_i(w))\in \configR\big)$ extends to a smooth map $p^{i}_{v,w}\colon \confignoeud\rightarrow \configR$,
\item for any two vertices $v$ and $w$, the map $\big(c\in \confignoeud[\Gamma][\psi][^0] \mapsto (c(v), c(w))\in \configM\big)$ extends to a smooth map $p^{e}_{v,w}\colon \confignoeud\rightarrow \configM$,
\item for any vertex $v$, the map $\big(c\in \confignoeud[\Gamma][\psi][^0] \mapsto c(v)\in C_1(\ambientspace)\big)$ extends to a smooth map $p_v\colon \confignoeud \rightarrow C_1(\ambientspace)$.
\end{itemize} 
\end{theo}

\begin{df}

For any edge $f$ of $\Gamma$, which goes from a vertex $v$ to a vertex $w$, $C_f$ denotes the configuration space $\configR$ if $f$ is internal, and $\configM$ if $f$ is external, and $p_f\colon\confignoeud\rightarrow C_f$ denotes the map $p_{v,w}^{i}$ if $f$ is internal, and $p_{v,w}^{e}$ if $f$ is external.
\end{df}
If $n$ is odd, 
orient the space $\confignoeud$ as follows.\footnote{Here, we describe the orientation in a more natural way than in \cite[Section 2.4]{article1}, but it is immediate to see that the two definitions are equivalent.} 
For any $i\in\und n$, let $\d Y_i^{v}$ denote the 
$i$-th coordinate form of the internal vertex $v$ (parametrized by $\R^n$) 
and for any $i\in\und{n+2}$, let $\d X_i^{v}$ denote 
the $i$-th coordinate form of the external vertex $v$ (in an oriented chart of $\ambientspace$).
Split any external edge $e$ in two halves: 
the tail $e_-$ and the head $e_+$. 
Define a form $\orientationarete{e_\pm}$ for any half-edge $e_\pm$ of an external edge $e$, as follows: \begin{itemize}
\item for the head $e_+$ of an edge that is not a leg, going to an external vertex $v$, $\Omega_{e_+} = \d X_v^1$,
\item for the head $e_+$ of a leg going to an external vertex $v$, $\Omega_{e_+} = \d X_v^2$,
\item for the tail $e_-$ of an edge coming from an external vertex $v$, $\Omega_{e_-} = \d X_v^3\wedge\cdots\wedge\d X_v^{n+2}$,
\item for any external half-edge $e_\pm$ adjacent to an internal vertex $v$, $\Omega_{e_\pm} = \d Y_v^1\wedge\ldots\wedge\d Y_v^n$.
\end{itemize}

Let $N_T(\Gamma)$ denote the number of trivalent vertices, and define the \emph{sign} of a BCR diagram as $\epsilon(\Gamma) = (-1)^{N_T(\Gamma)+\Card(\aretesexternes)}$. The orientation of $\confignoeud$ is $\orientationconfig=\epsilon(\Gamma)\bigwedge\limits_{e\in\aretesexternes}\orientationarete{e}$,
where $\orientationarete{e}= \orientationarete{e_-}\wedge\orientationarete{e_+}$ for any external edge $e$.

If $n$ is even, orient $\confignoeud$ as $\epsilon(\Gamma)(\psi(\R^n))^{\sommetsinternes} \times (\spamb)^{\sommetsexternes}$.

\begin{lm}\label{orex}
For $\degk\geq2$, let $\Gamma_\degk$ be the degree $\degk$ BCR diagram of Figure \ref{Gk}. 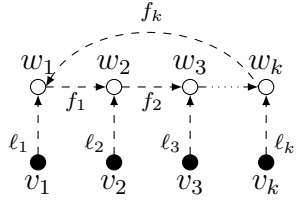
\begin{figure}[H]
\centering 
\begin{tikzpicture}[yscale=-1]
\draw (1,0) circle (0.1) ++(0,-0.3) node {$w_{1}$};
\draw [dashed, ->, >= latex] (1.1,0) -- (1.9,0) ;
\draw [dashed, <-, >= latex] (1.1,0) to[bend right=60] (3.9,0) ;
\draw [dotted] (3.2,0) -- (3.7,0) ;
\draw [dashed, ->, >= latex] (2.1,0) -- (2.9,0) ;
\draw [dashed] (3.1,0) -- (3.2,0) ;
\draw [dashed, ->, >= latex] (3.7,0) -- (3.9,0) ;
\draw (2,0) circle (0.1) ++(0,-0.3) node {$w_{2}$};
\draw (3,0) circle (0.1) ++(0,-0.3) node {$w_{ 3}$};
\draw (4,0) circle (0.1) ++(0,-0.3) node {$w_{ k}$};
\fill (1, 1) circle (0.1);
\fill (2, 1) circle (0.1);
\fill (3, 1) circle (0.1);
\fill (4, 1) circle (0.1);
\draw (1,1) circle (0.1) ++(0,0.3) node {$v_{1}$};
\draw (2,1) circle (0.1) ++(0,0.3) node {$v_{ 2}$};
\draw (3,1) circle (0.1) ++(0,0.3) node {$v_{ 3}$};
\draw (4,1) circle (0.1) ++(0,0.3) node {$v_{ \degk}$};
\draw [dashed, ->, >= latex] (1,0.9) -- (1,0.1) ;
\draw [dashed, ->, >= latex] (2,0.9) -- (2,0.1) ;
\draw [dashed, ->, >= latex] (3,0.9) -- (3,0.1) ;
\draw [dashed, ->, >= latex] (4,0.9) -- (4,0.1) ;
\draw (1.5,.2) node {\scriptsize $f_{1}$};
\draw (2.5,.2) node {\scriptsize $f_{2}$};
\draw (2.5,-1) node {\scriptsize $f_{\degk}$};
\draw (.75,.75) node {\scriptsize $\ell_{ 1}$};
\draw (1.75,.75) node {\scriptsize $\ell_{ 2}$};
\draw (2.75,.75) node {\scriptsize $\ell_{3}$};
\draw (4.25,.75) node {\scriptsize $\ell_{\degk}$};
\end{tikzpicture} \caption{The graph $\Gamma_\degk$}
\label{Gk}
\end{figure}

The space $(-1)^{k+n(k+1)}C_{\Gamma_\degk}(\psi)$ is oriented by the coordinates $(c_i(v_{j}),c(w_{j}))_{j\in \zk}\in (\R^n\times \spamb)^\degk$, and the order of Definition \ref{orientationar} on $E(\Gamma_k)$ is 
$(-1)^{k-1}(f_1, \ell_1, \ldots, f_k, \ell_k)$.

\end{lm}

\begin{proof}
First note that $\epsilon(\Gamma_\degk)=(-1)^k$ since there are $2k$ external edges and $k$ trivalent vertices. The result for even $n$ follows.
Let us now assume that $n$ is odd.
For any $j\in \zk$, let $\plusun{j}$ denote the integer \[\plusun{j}= \begin{cases} j+1 & \text{if $j<\degk$,}\\ 1 & \text{if $j=\degk$,}\end{cases}\] 
so that for any $j\in \zk$, $\ell_j$ is the leg from $v_j$ to $w_j$, and $f_j$ is the external edge of the cycle from $w_j$ to $w_{\plusun{j}}$ as in Figure \ref{Gk}. 
Set $\d Y_{v_j} = \bigwedge\limits_{i=1}^n \d Y_{v_j}^i$, $\d\overline{ X_{w_j}} = \bigwedge\limits_{i=3}^{n+2} \d X_{w_j}^i$ and $\d X_{w_j} = \bigwedge\limits_{i=1}^{n+2} \d X_{w_j}^i$,
so that $\orientationarete{\ell_j} = \d Y_{v_j}\wedge \d X_{w_j}^2$ and 
$\orientationarete{f_j} =  \d \overline{X_{w_j}}  \wedge \d X_{w_{\plusun{j}}}^1$.
Therefore, \begin{eqnarray*}
\orientationconfig[\Gamma_k] &= & (-1)^k\bigwedge\limits_{j=1}^\degk \left(\d Y_{v_j} \wedge \d X_{w_j}^2 \wedge \d \overline{X_{w_j}}  \wedge \d X_{w_{\plusun{j}}}^1 \right) \\
&=& (-1)^k\d Y_{v_1}\wedge \d X_{w_1}^2\wedge \d \overline{X_{w_1}}\wedge 
\bigwedge\limits_{j=2}^\degk \left(\d X_{w_{j}}^1 \wedge\d Y_{v_j}   \wedge \d X_{w_j}^2 \wedge \d \overline{ X_{w_j} }\right)
\wedge\d X_{w_1}^1 \\
&=& (-1)^{k+1}\bigwedge\limits_{j=1}^\degk \left( \d X_{w_{j}}^1 \wedge\d Y_{v_j}   \wedge \d X_{w_j}^2 \wedge \d \overline{ X_{w_j} }\right) 
=- \bigwedge\limits_{j=1}^\degk \left(\d Y_{v_j} \wedge \d X_{w_j}\right).
\end{eqnarray*}
The second assertion is immediate, since the canonical order obtained from  Definition \ref{orientationar} is 
$(f_k, \ell_1, f_1, \ell_2, \ldots, f_{k-1}, \ell_k)$, which differs from 
$(f_1, \ell_1, \ldots, f_k, \ell_k)$ by a cycle of length $k$.
\end{proof}

\subsection{Rectifiability and virtual rectifiability}\label{Alexander}

Let $\inj$ denote the space of linear injections $\R^n\hookrightarrow \R^{n+2}$, 
and let $\iota_0$ be the standard injection $x \in \R^n\mapsto (0,0,x) \in \R^{n+2}$. 
Let $\mathbb D^{n}$ be the unit ball of $\R^{n}$, and see $\pi_n(\inj,\iota_0)$ as 
the set $[(\R^n, \R^n \setminus \mathbb D^n), (\inj, \iota_0)]$ of homotopy classes 
of maps $\R^n\rightarrow \inj$ that map $\R^n\setminus\mathbb D^n$ to $\iota_0$ among such maps.
\begin{lm}\label{obstructionlemma}
Let $\spamb$ be a parallelizable asymptotic integral homology $\R^{n+2}$ and let $\psi$ be a long knot of $\spamb$. 
 For any parallelization $\tau$ of $\ambientspace$, 
 the tangent map $T\psi$ induces a map 
 $\iota(\tau,\psi)\colon x\in \R^n\mapsto (\tau_{\psi(x)})^{-1}\circ T_x\psi\in \inj$.

The class \[[\iota(\tau, \psi)]\in \pi_n(\inj, \iota_0)\] is independent of the parallelization $\tau$.\end{lm}
\begin{proof}
Let $\tau$ and $\tau'$ be two parallelizations of $\spamb$. 
\cite[Theorem 6.2]{article1}, whose proof and result remain inchanged for even $n$, 
implies the existence of a smooth family $(\tau_t)_{t\in [0,1]}$ of parallelizations of $\spamb$ 
such that $\tau_0=\tau$ and that $\tau_1$ and $\tau'$ coincide on $\psi(\R^n)$. 
This yields a homotopy $(\iota(\tau_t,\psi))_{t\in[0,1]}$ from $\iota(\tau,\psi)$ to $\iota(\tau',\psi)$. 
\end{proof}

\begin{df}\label{rectifiabilitydef} 
If $n=1$, we say that any long knot $\psi\colon\R\hookrightarrow\spamb$ 
of an asymptotic rational homology $\R^3$ is \emph{rectifiable} by definition.
If $n\geq2$ and if $\psi\colon \R^n\hookrightarrow \spamb$ is a long knot of 
a parallelizable asymptotic integral homology $\R^{n+2}$, 
the class $\iota(\psi)=[\iota(\tau, \psi)]\in \pi_n(\inj, \iota_0)$ is called 
the \emph{rectifiability obstruction of $\psi$}.
The long knot $\psi$ is \emph{rectifiable} if $\spamb$ is parallelizable and if its rectifiability obstruction is zero. 
\end{df}

For long knots in a possibly non-parallelizable asymptotic $\mathbb K$-homology $\R^{n+2}$, 
we have the following weaker definition.
\begin{df}\label{defiota}
Let $\ambientspace$ be an asymptotic $\mathbb K$-homology $\R^{n+2}$, 
and let $r$ be a positive integer.
A long knot $\psi\colon \R^n\hookrightarrow \ambientspace$ is \emph{$r$-rectifiable} if 
the connected sum $\psi^{(r)}=\psi\sharp \cdots \sharp \psi$ of 
$r$ copies of $\psi$ is rectifiable.
A long knot $\psi$ is \emph{virtually rectifiable} if it is $r$-rectifiable for some $r\geq1$.
\end{df}
In Section \ref{Section4}, we establish the following lemma.
\begin{lm}\label{threct}
If $n \equiv 1 \mod 4$, then any long knot in an asymptotic $\mathbb K$-homology $\R^{n+2}$ is $4$-rectifiable.
If $n$ is even, then any long knot in a parallelizable asymptotic integral homology $\R^{n+2}$ is $2$-rectifiable.
\end{lm}

\subsection{Propagators and first definition of \texorpdfstring{$Z_k$}{Zk} }
In this article, an \emph{embedded rational $\dimd$-chain} $A$ 
of a manifold $X$ is a finite rational combination of compact oriented 
$\dimd$-submanifolds with boundary and ridges 
with pairwise disjoint interiors of $X$. 
The \emph{support} $\mathrm{Supp}(A)$ of $A$ is 
the union of these submanifolds, 
and the \emph{interior} $\Int(A)$ of $A$ is the union of their interiors.
Let $w_A \colon \Int(A) \rightarrow \mathbb Q$ denote the function that maps a point of $\Int(A)$ to the coefficient
of the submanifold in which he lies.

In this section, 
we define $Z_k$ for long knots in parallelized asymptotic homology $\R^{n+2}$,
for an integer $k\geq2$.
Let $(\ambientspace, \tau)$ be a parallelized asymptotic $\mathbb K$-homology $\R^{n+2}$.

\begin{df}\label{df-prop}

An \emph{internal propagator}\footnote{In \cite{article1}, propagators were called propagating chains.} 
is an embedded rational $(n + 1)$-chain $A$ in
$\configR$ such that $\partial A = \frac12{(G_{|\partial\configR})}^{-1}(\{ - x_A, +x_A\})$ for some $x_A \in \s^{n-1}$.

An \emph{external propagator} of $(\ambientspace, \tau)$ is an embedded rational $(n + 3)$-chain $B$ in
$\configM$ such that $\partial B =\frac12 {(G_\tau)}^{-1}(\{- x_B, +x_B\})$ for some $x_B \in \s^{n+1}$.

For any $\degk\geq1$, a $\degk$-family 
$F_*= (A_i, B_i)_{i\in\indices}$ of propagators of $(\ambientspace, \tau)$ is 
the data of $2\degk$ internal propagators $(A_i)_{i\in\indices}$ 
and $2\degk$ external propagators $(B_i)_{i\in\indices}$ of $(\ambientspace, \tau)$.

\end{df}
A simple example of an external propagator of $\R^{n+2}$ with its canonical trivialization is given by the chain
$B= \frac{(-1)^{n-1}}2 \overline{\left\{ (x, x+t x_B) \mid x\in\R^{n+2}, t\in \R^*\right\}}\subset C_2(\R^{n+2})$ 
for some $x_B\in \s^{n+1}$. 
Similarly, for any $x_A\in \s^{n-1}$, the chain
$A= \frac{(-1)^{n-1}}2 \overline{\left\{ (x, x+t x_A) \mid x\in\R^{n}, t\in \R^*\right\}}
\subset C_2(\R^n)$
is an internal propagator. 
When $\spamb$ is a general asymptotic $\mathbb K$-homology $\R^{n+2}$, 
the existence of such propagators follows from \cite[Section 4.2]{article1}.
\footnote{In \cite{article1}, the existence is stated for 
odd-dimensional asymptotic $\mathbb Z$-homology $\R^{n+2}$, but the argument is exactly the 
same for asymptotic $\mathbb Q$-homology $\R^{3}$ or for even-dimensional asymptotic $\mathbb Z$-homology $\R^{n+2}$.}
We recall the discrete definition of the invariant $\Z[]$ from our previous article \cite[Sections 2.7-2.8]{article1}.

Let $\psi$ be a long knot of $\spamb$.
Consider a $\degk$-family $F_*=(\Propint_i,\Propext_i)_{i\in\indices}$ of propagators of $(\ambientspace, \tau)$.
For any BCR diagram $\Gamma$, let $P_\Gamma$ be the product map
 \[ \begin{array}{llll}P_\Gamma\colon& \confignoeud &\rightarrow &
\prod\limits_{e\in\aretesinternes}\configR\times\prod\limits_{e\in\aretesexternes}\configM
=\prod\limits_{e\in\aretes}C_e\\
&c & \mapsto &  (p_e(c))_{e\in\aretes}
\end{array} .\]
For any edge $e$ of a numbered BCR diagram $(\Gamma,\sigma)$, 
let $X_{e,\sigma} $ denote the chain $A_{\sigma(e)}$ if $e$ is internal, 
or the chain $B_{\sigma(e)}$ if $e$ is external.

The $\degk$-family $F_*$  is in \emph{general position for $\psi$} if, 
for any numbered BCR diagram $(\Gamma,\sigma)\in \graphesnum$, 
and for any $c\in\confignoeud$ such that $P_\Gamma(c) \in\prod\limits_{e\in\aretes}X_{e,\sigma}$, \begin{itemize}
\item for any edge $e$, $p_e(c) \in \Int(X_{e,\sigma})$,
\item we have the following equality between oriented vector spaces 
\[ (-1)^{(n-1)k}\epsilon(c) T_{P_\Gamma(c)}\left(\prod\limits_{e\in\aretes}C_e\right)
= T_cP_\Gamma(T_c\confignoeud) + \prod\limits_{e\in\aretes}T_{p_e(c)}X_{e,\sigma},\] 
where $\epsilon(c) =\pm1$ is called the \emph{sign} of the intersection point $c$,
and where the orientations of the product is
given by any order on $\aretes$ compatible with the orientation of Lemma \ref{orientationar}
\end{itemize}
The $(-1)^{(n-1)k}$ sign above is such that
the sign $\epsilon(c)$ is obtained from a comparison of the orientations of $T_c \confignoeud$
and 
$\prod\limits_{e\in \aretes}\big(  \N_{p_e(c)} X_{e,\sigma}  \big)$.
%for any family in general position and any $c$ such that 
%$P_\Gamma(c) \in\prod\limits_{e\in\aretes}X_{e,\sigma}$,  

In the following, set $D_{e,\sigma}^{F_*} = {p_e}^{-1}(X_{e,\sigma})$. This defines a chain of $\confignoeud$.
The algebraic intersection of the chains $(D_{e,\sigma}^{F_*})_{e\in\aretes}$ is defined as
\[ I_{F_*}(\Gamma,\sigma, \psi) = \sum\limits_{c\in {P_\Gamma}^{-1}\big(\prod\limits_{e\in\aretes}X_{e,\sigma}\big)}
\epsilon(c) \prod\limits_{e\in \aretes} w_{X_{e,\sigma}}(p_e(c)).\]
The proof of \cite[Lemma 11.4]{[Lescop2]} directly adapts to the higher-dimensional case and guarantees 
the existence of $\degk$-families of propagators in general position for any $\psi$ and any $\degk$. 
See \cite[Theorem 4.3]{article1} for an outline. We now define our generalized BCR invariant $Z_k$.

\begin{df}
Recall that $\mathbb K$ denotes $\mathbb Z$ if $n\geq2$, and $\mathbb Q$ if $n=1$.
Let $(\spamb, \tau)$ be a parallelized asymptotic $\mathbb K$-homology $\R^{n+2}$, and let $\psi$ be a long knot of $\spamb$.
Let $F_*=(\Propint_i,\Propext_i)_{i\in\indices}$ be a $\degk$-family of propagators of $(\ambientspace, \tau)$ in general position for $\psi$. Set \[Z_k^{F_*}(\psi, \tau) = \frac1{(2\degk)!} \sum\limits_{(\Gamma,\sigma)\in\graphesnum} I_{F_*}(\Gamma,\sigma,\psi).\]
\end{df}

For odd $n\geq 3$, \cite[Theorems 2.10 and 2.13]{article1} imply the following theorem.
The result for $n=1$ follows from \cite[Corollary 2.15]{article3}, 
and from the same arguments than in the proof of \cite[Theorem 2.13]{article1}.

\begin{theo}\label{th*}
If $n$ is odd,
then $Z_k^{F_*}(\psi)$ depends neither on the choice of the propagators, 
nor on the parallelization. It defines an 
invariant $Z_k$ of long knots
such that for any diffeomorphism $\phi$ of $\spamb$ that restricts to $\nbdinf$ as the identity map, 
$Z_k(\phi \circ \psi) = Z_k(\psi)$.
Furthermore, $Z_k$ is trivial if $k$ is odd.
\end{theo}
The results of \cite{[Watanabe]} imply that the last statement above is an "if and only if"
when $\spamb =\R^{n+2}$.

In a work in preparation \cite{article4}, we prove the following result, which is an even-dimensional analogue of the above theorem.

\begin{theou}\label{thu*}
If $n$ is even, and if the map $\iota(\tau, \psi)$ 
of Lemma \ref{obstructionlemma}
is the constant map\footnote{Lemma \ref{rect-lemma} implies the existence of such a parallelization for any
rectifiable long knot.} with value $\iota_0$,
then $Z_k^{F_*}(\psi)$ depends neither on the choice of the family $F_*$ of propagators of $(\spamb, \tau)$, 
nor on the parallelization $\tau$ such that $\iota(\tau, \psi)$ is the constant map with value $\iota_0$.
It defines an invariant $Z_k(\psi)$ of rectifiable long knots
such that for any diffeomorphism $\phi$ of $\spamb$ that restricts to $\nbdinf$ as the identity map, 
$Z_k(\phi \circ \psi) = Z_k(\psi)$.
Furthermore, $Z_k$ is trivial if $k$ is even.
\end{theou} 
In the computations of this article, we use the two above theorems as a definition of the \emph{degree $k$ generalized Bott-Cattaneo-Rossi} (BCR for short) \emph{invariant} $Z_k(\psi)$ for rectifiable long knots. 
In the next subsection, we explain how the invariants extend to any odd-dimensional long knot 
(without assuming that $\spamb$ is parallelizable),
and to any even-dimensional long knot (without assuming that $\psi$ is rectifiable) if $\spamb$ is parallelizable.

Let $\psi_0$ denote the trivial knot.
In Section \ref{Section2}, we are going to compute $Z_k^{F_*}(\psi)- Z_k^{F_*^0}(\psi_0)$ for any rectifiable knot $\psi$, using two particular families of propagators $F_*$ and $F_*^0$, defined in Section \ref{Section 3}.
These computations do not use Theorem \ref{thu*}, and, what is
actually proved in this article when $n$ is even is the formula of
Theorem \ref{Reidth} with the difference $Z_k^{F_*}(\psi)- Z_k^{F_*^0}(\psi_0)$ above instead of $Z_k(\psi)$. The
given formulation of Theorem \ref{Reidth}
is then a straightforward consequence of Theorem \ref{thu*}, which implies
that $Z_k(\psi_0)=0$.

\subsection{Connected sum and general definition of \texorpdfstring{$\Z[]$}{Zk}}
In this section, we review the connected sum defined in \cite[Section 2.9]{article1}.

Let $\punct{M_1}$ and $\punct{M_2}$ be two asymptotic $\mathbb K$-homology $\R^{n+2}$, respectively decomposed as $B(M_1)\cup \voisinageinfini$ and $B(M_2)\cup\voisinageinfini$.
Let $\voisinageinfinideux$ be the complement in $\R^{n+2}$ of the two open balls $\mathbb B_1$ and $\mathbb B_2$ of radius $\frac14$ and with respective centers $(0, 0, \ldots, 0, -\frac12)$ and $(0, 0, \ldots, 0, \frac12)$.
For $i\in\{1,2\}$ and $x$ in $\partial B(M_i)\subset \R^{n+2}$, define the map $\phi_i(x) = \frac14x +(-1)^i(0, \ldots, 0 ,\frac12)$, 
which is a diffeomorphism from $\partial B(M_i)$ to $\partial \mathbb B_i$.

Set $\punct{M_1}\sharp \punct{M_2}= \voisinageinfinideux\cup B(M_1) \cup B(M_2)$, where $B(M_i)$ is glued to $\partial \mathbb B_i$ via the map $\phi_i$. Set $B(\punct{M_1}\sharp \punct{M_2}) = B(\punct{M_1}) \cup B(\punct{M_2}) \cup \overline{(\voisinageinfinideux\setminus\voisinageinfini)}$. This defines an asymptotic $\mathbb K$-homology $\R^{n+2}$, with two canonical injections $\iota_i \colon B(M_i) \hookrightarrow B(\punct{M_1}\sharp \punct{M_2})$ for $i\in\{1,2\}$. 

If $\punct{M_1}$ and $\punct{M_2}$ contain two long knots $\psi_1$ and $\psi_2$, define the long knot $\psi_1\sharp\psi_2$ of $\punct{M_1}\sharp\punct{M_2}$ by the following formula, for any $x\in \R^n$:

\[(\psi_1\sharp \psi_2)(x)= \left\{\begin{array}{lll}
\iota_2(\psi_2(4.x_1, \ldots, 4.x_{n-1}, 4.x_n-2)) & \text { if $||x-(0, \ldots, 0, \frac12)|| \leq \frac14$,}\\
\iota_1(\psi_1(4.x_1, \ldots, 4.x_{n-1}, 4.x_n+2)) & \text { if $||x-(0, \ldots, 0, -\frac12)|| \leq \frac14$,}\\
(0, 0, x)\in \voisinageinfinideux &\text{otherwise.}
\end{array}
\right.
\]

Similarly, any two parallelizations $\tau_1$ and $\tau_2$ of $\punct{M_1}$ and $\punct{M_2}$ induce a parallelization $\tau_1\sharp \tau_2$ of $\punct{M_1}\sharp\punct{M_2}$, which is well-defined up to homotopy. 
In particular, if $\punct{M_1}$ and $\punct{M_2}$ are parallelizable, then $\punct{M_1}\sharp\punct{M_2}$ is also parallelizable in the sense of Definition \ref{paral-def}. 
\begin{theo}\label{conn-sum}
For any long knots $\psi_1\colon \R^n \hookrightarrow \punct{M_1}$ and $\psi_2\colon \R^n \hookrightarrow \punct{M_2}$, \[\Z[](\psi_1\sharp\psi_2)= \Z[](\psi_1)+\Z[](\psi_2).\]
\end{theo}
\begin{proof}
When $n\geq3$ is odd, this is the result of \cite[Theorem 2.17]{article1}. 
The proof extends without any change to the even-dimensional case assuming Theorem \ref{thu*}. 
For $n=1$, it follows from \cite[Theorem 2.14]{article3} and \cite[Theorem 16.10]{[Lescop2]}.

\end{proof}

According to \cite[Proposition 5.5]{[Lescop2]},
any asymptotic rational homology $\R^3$ is parallelizable in the sense of Definition \ref{paral-def}.
In general, we have the following \cite[Proposition 2.18]{article1}.
\begin{prop}\label{conn-sum2}
For any positive odd integer $n$, the connected sum of any asymptotic integral homology $\R^{n+2}$ with itself is parallelizable in the sense of Definition \ref{paral-def}.
\end{prop}
The above proposition allows us to extend $\Z[]$ as follows in the odd-dimensional case.
Lemma \ref{threct} allows us to extend $\Z[]$ similarly in the even-dimensional case.

\begin{df}\label{Zkext}
Let $\psi$ be a long knot in an odd-dimensional asymptotic $\mathbb K$-homology $\R^{n+2}$
or in an even-dimensional parallelizable asymptotic integral homology $\R^{n+2}$.
Define $\Z[](\psi) = \frac12 \Z[](\psi\sharp \psi),$ where $\Z[]$ is the invariant of Theorem \ref{th*}.
\end{df}

\cite[Prop 2.20]{article1} implies\footnote{The proof is 
for odd $n$, it also works for even $n$.} 
that Theorem \ref{conn-sum} is still valid for this extended $\Z[]$.

\subsection{Linking number}\label{sectionlk}
We use the following definition and basic properties of the linking number. 

\begin{df}\label{d27}Let $X^{\dimd}$ and $Y^{n+1-\dimd}$ be two disjoint rational cycles of 
our homology $(n+2)$-sphere $M$, with $\dimd\in\und n$. Let $W_X$ and $W_Y$ be two chains with 
respective boundaries $X$ and $Y$, such that $W_X$ and $W_Y$ are transverse to each other. 
The \emph{linking number} of $X$ and $Y$ is defined as the algebraic intersection 
number $\langle W_X, Y\rangle_M$, so that \[\mathrm{lk}(X,Y) = (-1)^{\dimd+1} \langle X, W_Y\rangle_M=  \langle W_X, Y\rangle_M.\]
Furthermore, $\mathrm{lk}(X^\dimd, Y^{n+1-\dimd}) = (-1)^{n\dimd+1} \mathrm{lk}(Y^{n+1-\dimd},X^\dimd)$.

\end{df}
These linking numbers will appear in the computation of our invariant $Z_k$ because of the following lemma, which relates external propagators to linking numbers.

\begin{lm}\label{lk}
Let $X^\dimd$ and $Y^{n+1-\dimd}$ be two disjoint cycles of $\ambientspace$. For any external propagator $B$, \[\mathrm{lk} (X, Y) = (-1)^{d+1} \langle X\times Y, B\rangle_{\configM}.\]
\end{lm}
\begin{proof}
%The class of the cycle $X\times Y$ is an element of $H_{n+1}(\configM)$. \cite[Lemma 3.3]{article1} implies that $H_{n+3}(\configM)=0$. Therefore, the intersection number $\langle X\times Y, B\rangle_{\configM}$ only depends on the homology class $[X\times Y]$. 
Let $W_X$ and $W_Y$ be chains such that $\partial W_X=X$ and $\partial W_Y=Y$ as above. For the proof, assume that $W_X$ and $W_Y$ are manifolds (the general case follows easily). 
Let $W'_X$ be obtained from $W_X$ by removing a little ball $\mathbb D^{\dimd+1}_x$ with boundary $\s^\dimd_x$ around each point $x\in W_X\cap Y$, 
so that $\partial(W'_X\times Y) = X\times Y- \sum\limits_{x\in W_X\cap Y} \s^\dimd_x\times Y$. 
Since $W'_X\cap Y = \emptyset$, 
we have $(W'_X\times Y) \cap \partial B= \emptyset$, 
and
\[ \langle X\times Y, B\rangle  = \sum\limits_{x\in W_X\cap Y}\langle \s^\dimd_x\times Y, B\rangle.\]

For any $x\in W_X\cap Y$, assume that $\mathbb D_x^{\dimd+1}$ meets $W_Y$ transversely along an interval $[x, x']$. 
For any $x\in W_X\cap Y$, remove a little ball $\mathbb D^{n+2-\dimd}_{x'}$ with boundary $\s^{n+1-\dimd}_{x'}$ around the point $x'$ (which is the intersection of $\s^\dimd_x$ and $W_Y$) from $W_Y$. 
The same argument as above yields 
\[\langle X\times Y, B\rangle  = \sum\limits_{x\in W_X\cap Y} \langle \s^\dimd_x\times\s^{n+1-\dimd}_{x'}, B\rangle.\]

It suffices to prove that $ (-1)^{\dimd+1} \langle \s_x^\dimd\times \s_{x'}^{n+1-\dimd} , B \rangle_{\configM}$ is  the sign of the intersection point $x$ in $W_X \cap Y$.
Since the balls $\mathbb D^{\dimd+1}_x$ and $\mathbb D^{n+2-\dimd}_{x'}$ can be taken arbitrarily small, assume without loss of generality that $\ambientspace = \R^{n+2}$, \[\begin{array}{rl} \s_x^\dimd&= \{(x_1,\ldots, x_{\dimd+1}, 0,\ldots, 0)\mid {(x_1)}^2+\ldots+{(x_{\dimd+1})}^2=1\}, \\
\s_{x'}^{n+1-\dimd}&= \{(0, \ldots, x_{\dimd+1}, \ldots, x_{n+2}) \mid (x_{\dimd+1}-1)^2+{(x_{\dimd+2})}^2+\ldots +{(x_{n+2})}^2=1\}, \\ 
\text{and  } B &= \frac12\left(G^{-1}(\{-e_{n+2}\}) + G^{-1}(\{+e_{n+2}\})\right),\end{array}\] where $e_{n+2}$ is the last vector of the canonical basis of $\R^{n+2}$.
Now, the intersection number $\langle \s_x^\dimd\times \s_{x'}^{n+1-\dimd} , B \rangle_{C_2(\R^{n+2})}$ is the degree of the Gauss map $\s_x^\dimd\times \s_{x'}^{n+1-\dimd} \rightarrow \s^{n+1}$, which is $(-1)^\dimd$. Since $\langle \mathbb B_x^{\dimd+1} , \s^{n+1-\dimd}_{x'} \rangle_{\R^{n+2}}=-1$, this concludes the proof.
\end{proof}

\subsection{Seifert (hyper)surfaces and matrices} 

From now on, we will often use the coordinates $(x_1,x_2,\overline x)\in \R\times \R\times \R^n$ for points $x\in\nbdinf\subset \R^{n+2}$.
Set $N^0 = \{x \in \R^{n+2}\mid {x_1}^2+{x_2}^2 \leq 1\text{ or } ||x|| \geq1\}$.
\begin{df}\label{defSf}
A \emph{Seifert (hyper)surface} for a long knot $\psi\colon \R^n\hookrightarrow \ambientspace$ is an oriented connected $(n+1)$-submanifold $\Sigma$ of $\ambientspace$ such that $\partial\Sigma = \psi(\R^n)$, such that
$\Sigma\cap\bouleambiante$ is compact, and such that
 there exists a neighborhood $N$ of $\psi(\mathbb R^n)$ such that \begin{itemize}
\item $\voisinageinfini \subset N$, 
\item $N\cap\bM$ is a tubular neighborhood of $\psi(\R^n)\cap \bM=\psi(\mathbb D^n)$, 
\item there exists a diffeomorphism 
$\Theta \colon N^0 \rightarrow N$ that restricts to $\nbdinf$ as the identity such that
for any $x\in\R^n$, $\psi(x)= \Theta(0,0,x)$, and such that
$\Sigma\cap N= \Theta(\{(r\cos(\theta), r\sin(\theta), \overline x)\in N^0 \mid r\geq0,\overline x\in\R^n \})$ for some $\theta\in\R$.
\end{itemize}
\end{df}

\begin{lm}\label{lmS}
Any null-homologous long knot $\psi\colon\R^n\hookrightarrow \spamb$ 
of an asymptotic rational homology $\R^{n+2}$ admits Seifert surfaces as in Definition \ref{defSf}.
\end{lm}
\begin{proof}

There exists a bundle isomorphism from $\mathbb D^2 \times\mathbb D^n$ to a tubular
neighborhood of $\psi(\mathbb D^n)$
that reads $(d,b) \mapsto (f(b)(d),b)$ on $\mathbb D^2 \times \s^{n-1} (\subset
\nbdinf \cap N^0)$ for some map $f \colon \s^{n-1} \to SO(2)$.
The homotopy class of $f$ is obviously trivial when $n \neq 2$.
When $n=2$, since the homotopy class of $f$ is determined by the
self-intersection of the (null-homologous) knot $\psi(\R^2) \cup
\{\infty\}$ in $M=\spamb\cup\{\infty\}$, it is also trivial.
Therefore, there exists a diffeomorphism $\Theta_1$ from $N^0$ onto a
neighborhood $N$ of the image of $\psi$, which restricts to $\mathbb D^2
\times\mathbb  D^n$ as a bundle isomorphism as above, and which is the
identity map on $\nbdinf$.
$\Theta_1$ yields a smooth map $g\colon \partial N \rightarrow
\s^1$ and a longitude $\ell(\Theta_1) = g^{-1}(\{1\})$. 
According to
\cite[Theorem 34.2]{Steenrod} the map $g$ 
extends to $E=\overline{\spamb \setminus N}$
as soon as some obstruction in $H^2(E, \partial E ;
\pi_1(\s^1))= H_n(E;\mathbb Z)$, which 
is the class
$[\ell(\Theta_1)]\in H_n(E; \mathbb Z)$ in our case, vanishes.
When $n\neq1$, 
the group $H_n(E; \mathbb Z)$ is trivial, and when $n=1$, 
we can choose $\Theta_1$ such that $\ell(\Theta_1)$ is zero.\qedhere

\end{proof}

Let $\psi$ be a null-homologous long knot, let $\Sigma$ be a Seifert surface for $\psi$, and, for any $\dimd\in\und n$, let $b_\dimd$ denote the $\dimd$-th Betti number of $\Sigma$. Let $\Sigma^+$ denote a parallel surface obtained from $\Sigma$ by slightly pushing it in the positive normal direction. For any cycle $z$ in $\Sigma$, let $z^+$ denote the image of $z$ in the parallel surface $\Sigma^+$. 
\begin{df}\label{bases duales}
Let $\Sigma$ be a Seifert surface for a long knot $\psi$.
For any $\dimd\in\und n$, let $([\bb_i^\dimd])_{i\in\und{b_\dimd}}$ and $([\ba_i^\dimd])_{i\in\und{b_\dimd}}$ be two bases of $H_\dimd(\Sigma)$.
We say that the bases $\Ba=([\bb_i^\dimd])_{i,\dimd}$ and $\Bb=([\ba_i^\dimd])_{i,\dimd}$ of the reduced homology $\overline H_*(\Sigma)$ are \emph{dual} to each other if, for any $\dimd\in\und n$, and any $ (i , j) \in(\und{b_\dimd})^2$, $\langle [\bb_i^{\dimd}],[ \ba_j^{n+1-\dimd}]\rangle_{\Sigma}=\delta_{i,j}$, where $\langle\cdot,\cdot\rangle_{\Sigma}$ denotes the intersection pairing of $\Sigma$ and $\delta_{i,j}$ denotes the Kronecker delta.
Given such \emph{a pair $(\Ba,\Bb)$ of dual bases of $\overline{H}_*(\Sigma)$}, for any $d\in\und n$, define the \emph{Seifert matrices} \[V_\dimd^+(\Ba,\Bb) = (\mathrm{lk}(\bb_i^\dimd, (\ba_j^{n+1-\dimd})^+))_{1\leq i,j\leq b_\dimd}\] \[V_\dimd^-(\Ba,\Bb)= (\mathrm{lk}((\bb_i^\dimd)^+, \ba_j^{n+1-\dimd}))_{1\leq i,j\leq b_\dimd}\] and set 
$\Lk_{\degk,\nu}(\Ba, \Bb) = \frac1{\degk}\sum\limits_{\dimd\in\und n}
(-1)^{d+1}
\Tr(V_\dimd^+(\Ba, \Bb)^{\nu}V_\dimd^-(\Ba, \Bb)^{\degk-\nu})\) for any $k\geq 2$ and any $\nu\in\{0,\ldots, \degk\}$.

The duality of the bases implies that $V_\dimd^-(\Ba, \Bb)-V_\dimd^+(\Ba, \Bb) = I_{b_d(\Sigma)}$.
\end{df}

Note that pairs of dual bases as above exist thanks to Poincaré duality.  
The following immediate lemma describes how the $\Lk_{\degk,\nu}$ behave under connected sum.

\begin{lm}\label{sigmadd}
Let $\Sigma_1$ and $\Sigma_2$ be Seifert surfaces for two long knots $\psi_1$ and $\psi_2$. For $i\in\{1,2\}$, let $(\Ba_i,\Bb_i)$ be a pair of dual bases of $\overline H_*(\Sigma_i)$ as in Definition \ref{bases duales}.

There exists a natural Seifert surface $\Sigma_{1,2}$ for the connected sum $\psi_1\sharp \psi_2$ and a pair $(\Ba_{1,2},\Bb_{1,2})$ of dual bases of $\overline H_*(\Sigma_{1,2})$ such that, for any $\dimd\in\und n$, \[V_\dimd^\pm(\Ba_{1,2}, \Bb_{1,2}) = \begin{pmatrix} V_\dimd^\pm(\Ba_1, \Bb_1) & 0 \\ 0 & V_\dimd^\pm(\Ba_2, \Bb_2) \end{pmatrix}.\]
In particular, for any $k\geq 2$ and any $\nu\in\{0,\ldots, k\}$, \[\Lk_{\degk,\nu}(\Ba_{1,2}, \Bb_{1,2}) = \Lk_{\degk,\nu}(\Ba_{1}, \Bb_{1})+ \Lk_{\degk,\nu}(\Ba_{2}, \Bb_{2}).\]
\end{lm}

\subsection{A formula for \texorpdfstring{$\Z[]$}{Zk} in terms of linking numbers}

In Section \ref{Section2}, we prove the following theorem, where $k$ is an integer with $k\geq2$.
\begin{theo}\label{th0}
Recall that $\mathbb K$ denotes $\mathbb Z$ if $n\geq2$, and $\mathbb Q$ if $n=1$.
 For any null-homologous rectifiable knot $\psi$ in an asymptotic $\mathbb K$-homology $\R^{n+2}$, any Seifert surface $\Sigma$ for $\psi$, and any pair $(\Ba,\Bb)$ of dual bases of $\overline{H}_*(\Sigma)$, \[\Z[](\psi)=\sum\limits_{\nu=1}^{\degk-1} \poids_{\degk,\nu}\Lk_{\degk,\nu}(\Ba,\Bb),\]
where $\poids_{\degk,\nu} = \frac1{(\degk-1)!} \Card(\{\sigma\in\Sym_{\degk-1} \mid \Card\{i\in \und{\degk-2} \mid \sigma(i)<\sigma(i+1)\} = \nu-1
\}),$ and where $\Lk_{\degk,\nu}(\Ba,\Bb)$ is introduced in Definition \ref{bases duales}.
\end{theo}
Note the symmetry $\lambda_{k, \nu} = \lambda_{k, k-\nu}$ for any $k$ and $\nu$.

Theorem \ref{th0} yields the following more general corollary. 

\begin{cor}\label{Zklemma}
For any null-homologous virtually rectifiable knot $\psi$ in an asymptotic $\mathbb K$-homology $\R^{n+2}$, any Seifert surface $\Sigma$ for $\psi$, and any pair $(\Ba,\Bb)$ of dual bases of $\overline{H}_*(\Sigma)$,\[\Z[](\psi)=\sum\limits_{\nu=1}^{\degk-1} \poids_{\degk,\nu}\Lk_{\degk,\nu}(\Ba,\Bb).\]
In particular, this formula holds for any null-homologous long knot when $n\equiv 1 \mod 4$, 
and for any long knot in an even-dimensional parallelizable asymptotic integral homology $\R^{n+2}$.
\end{cor}
\begin{proof}Let $\psi$ be a null-homologous virtually rectifiable knot, and let $r$ be an integer such that $\psi^{(r)}$ is rectifiable. 
Let $\Sigma$, $\Ba$ and $\Bb$ be as in the theorem.
From Theorem \ref{conn-sum}, we get that 
$\Z[](\psi^{(r)})= r\Z[](\psi)$. For any $\degk\geq2$ and $\nu\in\und{\degk-1}$, 
Lemma \ref{sigmadd} yields $\Lk_{\degk,\nu}(\Ba_r, \Bb_r)= r\Lk_{\degk,\nu}(\Ba, \Bb)$, 
where $(\Ba_r,\Bb_r)$ is the pair of dual bases of the reduced homology of a Seifert surface for $\psi^{(r)}$
induced by $(\Ba, \Bb)$. 
On the other hand, Theorem \ref{th0} implies that $\Z[](\psi^{(r)})=\sum\limits_{\nu=1}^{\degk-1} \poids_{\degk,\nu} \Lk_{\degk,\nu}(\Ba_r,\Bb_r)$.
This concludes the proof of Corollary \ref{Zklemma} for virtually rectifiable long knots. 
The last assertion follows from Lemma \ref{threct}.\end{proof}

\subsection{Alexander polynomials and Reidemeister torsion}

We use the following formula of \cite[p.542]{[Levine]} as a definition of Alexander polynomials.
\begin{theo}[Levine]\label{Levine-th}
The \emph{Alexander polynomials} of the Seifert surface $\Sigma$ are defined, for $\dimd\in\und n$, by the formula\footnote{With the notations of \cite[Theorem 1]{[Levine]}, our Alexander polynomial $\Delta_{\dimd, \Sigma}$ is the product $\prod\limits_{i\in\und{b_\dimd}}\lambda_i^\dimd$.} \[\Delta_{\dimd,\Sigma}(t) = 
\det\left(t^{\frac12}V_\dimd^-(\Ba, \Bb)- t^{-\frac12}V_\dimd^+(\Ba,\Bb)\right)\] and do not depend on the choice of the pair $(\Ba, \Bb)$ of dual bases of $\overline H_*(\Sigma)$.
Note that $\Delta_{\dimd,\Sigma}(1)=1$.

If $\Sigma$ and $\Sigma'$ are two Seifert surfaces for $\psi$, then there exists an integer $a\in \mathbb Z$ such that\footnote{This follows from the fact that the $(\lambda_i^\dimd)$ from \cite{[Levine]} are defined from the knot up to multiplication by a monomial.} $\Delta_{\dimd, \Sigma'} (t) = t^{\frac{a}2}. \Delta_{\dimd, \Sigma}(t) $. 
\end{theo}

\begin{lm}\label{211}

For any Seifert surface $\Sigma$, and any $\dimd\in\und n$, $\Delta_{\dimd,\Sigma}(t^{-1})= \Delta_{n+1-\dimd,\Sigma}(t)$. 
%
%Furthermore, for any pair $(\Ba, \Bb)$ of dual bases of the reduced homology of $\Sigma$, \[\sum\limits_{d\in \und n}  \left( \Tr(V_\dimd^+(\Ba,\Bb)) + \Tr(V_\dimd^-(\Ba,\Bb))\right) =0 ,\]
%and
% \[\sum\limits_{\dimd\in\und n} \Tr( V_{\dimd}^-(\Ba,\Bb))= \frac{1-\chi(\Sigma)}2.\]
\end{lm}

\begin{proof}There exists a pair $(\Ba_0, \Bb_0)$ of dual bases $\Ba_0= ([\bb_i^\dimd])_{i,\dimd}$ and $\Bb_0 =([\ba_i^\dimd])_{i,\dimd}$ of $\overline H_*(\Sigma)$ such that $\bb_i^{\dimd} = \ba_i^{\dimd}$ for $\dimd>\frac{n+1}2$, and such that $\ba_i^{\dimd} = (-1)^{n\dimd}\bb_i^{\dimd}$ for $\dimd<\frac{n+1}2$. 
It follows that for $\dimd\neq \frac{n+1}2$, ${}^TV_\dimd^\pm(\Ba_0, \Bb_0) = - V_{n+1-\dimd}^\mp(\Ba_0, \Bb_0)$. This implies that $\Delta_{\dimd,\Sigma}(t^{-1})= \Delta_{n+1-\dimd,\Sigma}(t)$.
If $n$ is odd, let $\Ba'_0$ be the basis defined from $\Ba_0$ by replacing $\bb_i^{\frac{n+1}2}$ with $ (-1)^{\frac{n+1}2} \ba_i^{\frac{n+1}2}$ for any $i\in\{1,\ldots, b_{\frac{n+1}2}\}$, and let $\Bb'_0$ be the basis defined from $\Bb_0$ by replacing $\ba_i^{\frac{n+1}2}$ with $\bb_i^{\frac{n+1}2}$, so that $(\Ba'_0, \Bb'_0)$ is a pair of dual bases of $\overline{H}_*(\Sigma)$.
We have ${}^TV_{\frac{n+1}2}^\pm(\Ba_0, \Bb_0) =-V_{\frac{n+1}2}^\mp(\Ba'_0, \Bb'_0)$, hence $\Delta_{\frac{n+1}2, \Sigma}(t^{-1}) = \Delta_{\frac{n+1}2, \Sigma}(t)$. The first point of the lemma follows.\qedhere
%
%This implies that for any $\dimd\in\und n$, $\Delta_{\dimd,\Sigma}'(1) + \Delta_{n+1-\dimd,\Sigma}'(1)=0$, so that, for 
%any pair $(\Ba, \Bb)$ of dual bases of the reduced homology of $\Sigma$,
%  \[\sum\limits_{d\in \und n} \left( \Tr(V_\dimd^+(\Ba,\Bb)) + \Tr(V_\dimd^-(\Ba,\Bb))\right) =0 .\]
%
%On the other hand, since $V_\dimd^+(\Ba,\Bb)-V_\dimd^-(\Ba,\Bb)= (-1)^dI_{b_\dimd}$ for any $\dimd\in\und n$, we have \[\sum\limits_{d\in \und n} \left( \Tr(V_\dimd^+(\Ba,\Bb)) - \Tr(V_\dimd^-(\Ba,\Bb))\right) =\chi(\Sigma)-1.\]
%The two equations above imply the second and third point of the lemma. \qedhere
\end{proof}

The same method allows us to derive the following lemma, which will be used in Section \ref{Section2}.
Note that this lemma and Corollary \ref{Zklemma} are consistent with the forementioned property 
that $Z_k= 0$  when $k\equiv n \mod 2$.
\begin{lm}\label{symkn}
With the notations of Definition \ref{bases duales}, if $k \equiv n \mod 2$, then 
$\sum\limits_{\nu=1}^{\degk-1} \poids_{\degk,\nu}\Lk_{\degk,\nu}(\Ba,\Bb)=0.$
%$Z_k(\psi) = 0$.
\end{lm}
\begin{proof}
With the notations of the proof of Lemma \ref{211}, we have $\mathcal L_{k, k-\nu}(\Ba_0, \Bb_0) = (-1)^{k+n+1}\mathcal L_{k, \nu}(\Ba_0, \Bb_0)$, for any $\nu\in \und{k-1}$ when $k\neq \frac{n+1}2$,
and $\mathcal L_{k, k-\nu}(\Ba_0, \Bb_0) = (-1)^{k+n+1}\mathcal L_{k, \nu}(\Ba'_0, \Bb'_0)$, for any $\nu\in \und{k-1}$ when $k=\frac{n+1}2$. 
Since $\lambda_{k,\nu} = \lambda_{k, k-\nu}$, the lemma follows.
\end{proof}

\subsection{The Reidemeister torsion in terms of BCR invariants}

We use \cite[Theorem p.131]{[Milnorcyc]} to compute the Reidemeister torsion of null-homologous knots (i.e. the Reidemeister torsion of the knot complement) as follows. 

\begin{df}\label{Reidemeister}
The \emph{Reidemeister torsion} $\mathcal T_\Sigma(t)$ of a null-homologous long knot $\psi$ is defined\footnote{Here, we normalize the Reidemeister torsion without taking into account the degree $0$ part, which is independent of the knot. This normalization gives $\mathcal T_{\psi_0}(t) = 1$ for the trivial knot. 
For $n=1$, $\mathcal T_\psi(t)$ is the Alexander polynomial. This is not the usual normalization of the Reidemeister torsion.}
as follows. We set \[\mathcal T_\Sigma(t)= \prod\limits_{\dimd\in\und n}  \left(\Delta_{\dimd,\Sigma}(t)\right)^{(-1)^{\dimd+1}}\in \mathbb Q(t),\] where $\Sigma$ is a Seifert surface for $\psi$. We have $\mathcal T_\Sigma(1)=1$ and $\mathcal T_\Sigma(t^{-1})=(\mathcal T_\Sigma(t))^{(-1)^{n+1}}$, so that the torsion does not depend on the surface $\Sigma$ if $n$ is odd.
If $n$ is odd, we set $\torsion(t) = \mathcal T_\Sigma(t)$ for any Seifert surface, and
if $n$ is even, we let $\torsion(t)$ be $\mathcal T_\Sigma(t). t^{- \mathcal T_\Sigma'(1)}$, so 
that $\torsion'(1)=0$ and $\torsion(t)$ does not depend on $\Sigma$.
\end{df}
In Section \ref{Section6}, we 
%prove the identity 
%$\Ln(\torsion(e^h))=(-1)^n
%\sum\limits_{k\geq2}\sum\limits_{\nu=1}^{\degk-1} \poids_{\degk,\nu}\Lk_{\degk,\nu}(\Ba,\Bb)$ 
%for any dual bases $(\Ba, \Bb)$ of the homology of a 
%Seifert surface for a null-homologous long knot $\psi$. 
%This equality relies on computations between formal series without any intervention of the topology. 
%It allows us to 
rephrase Corollary \ref{Zklemma} as follows, using some computations on formal series associated to 
the coefficients $(\poids_{k, \nu})$.

\begin{theo}\label{Reidth}Recall that $\mathbb K$ denotes $\mathbb Z$ if $n\geq2$, and $\mathbb Q$ if $n=1$.
Let $\psi$ be a null-homologous long knot of an asymptotic $\mathbb K$-homology $\R^{n+2}$.
If $\psi$ is virtually rectifiable (and in particular for any $\psi$ if $n\equiv 1\mod 4$, 
or if $n$ is even, and if $\spamb$ is parallelizable), we have the following equality in $\QQ[[h]]$ : \[
\sum\limits_{\degk\geq 2}
Z_{\degk}(\psi)h^{\degk} = 
(-1)^{n}
\Ln\left(\torsion(e^{h})\right).
%=
%\sum\limits_{\dimd=1}^{n}(-1)^{d+1}\left(
%\Ln\left(\Delta_{\dimd,\Sigma}(e^{h})\right) - \Delta_{\dimd, \Sigma}'(1)\right)
\]
\end{theo}

The following lemma will be used later. 

\begin{lm}\label{dernierlemme}
Let $\Sigma$ be a Seifert surface for a null-homologous long knot $\psi$, 
and let $(\Ba,\Bb)$ be a pair of dual bases of its reduced homology. We have 
\[ \sum\limits_{d=1}^n (-1)^d \Tr\left( V_d^-(\Ba, \Bb) \right) = \frac{\chi(\Sigma)-1}2- \mathcal T_\Sigma'(1)\]
so that $ \mathcal T_\Sigma'(1)\in\mathbb Z$ .
\end{lm}
\begin{proof}
Recall that $V_d^-(\Ba, \Bb) -V_d^+(\Ba, \Bb)= I_{b_d(\Sigma)}$ for any $d\in\{1,\ldots, n\}$.
The definition of $\mathcal T_\Sigma(t)$ yields
\[ \mathcal T_\Sigma'(1) = \frac12 \sum\limits_{d=1}^n (-1)^{d+1} \Tr\left( V_d^-(\Ba, \Bb) + V_d^+(\Ba, \Bb) \right),\]
and we have
\[\chi(\Sigma)-1 = \sum\limits_{d=1}^n (-1)^d b_d(\Sigma)
= \sum\limits_{d=1}^n (-1)^d \Tr\left( V_d^-(\Ba, \Bb) - V_d^+(\Ba, \Bb)\right).\]
The formula of the lemma then follows from a linear combination of the two above equalities. 
When $n$ is odd, $\mathcal T_\Sigma'(1) =0 $.
Now assume that $n$ is even. Note that $\Sigma\cap \bM$ is an odd-dimensional compact orientable manifold with boundary homeomorphic to $\s^n$,
so that $\chi(\Sigma)= 1$.
Since $M$ is an integral homology sphere, 
the Seifert matrices are integral matrices 
when the bases are chosen in the image of $H_*(\Sigma; \mathbb Z) \rightarrow H_*(\Sigma; \mathbb Q)$.
This yields $\mathcal T_\Sigma'(1)=\sum\limits_{d=1}^n (-1)^{d+1} \Tr\left( V_d^-(\Ba, \Bb) \right) \in\mathbb Z$.
\end{proof}
\section{Computing \texorpdfstring{$\Z[]$}{Zk} from admissible propagators}\label{Section2}
\subsection{Admissible propagators}\label{Section31}

Let $\ambientspace$ be a fixed parallelizable asymptotic $\mathbb K$-homology $\R^{n+2}$ and 
let $\psi\colon\R^n\hookrightarrow \punct M$ be a fixed null-homologous long knot. 
Let $\big(\psi_0\colon x\in\R^n\mapsto (0,0,x)\in\R^{n+2}\big)$ denote the trivial knot.
Fix a real number $R\geq 3$.

For $r\in[1, R]$, let $N_r^0$ denote the following neighborhood of the trivial knot $N_r^0 = \{x\in\R^{n+2} \mid {x_1}^2+{x_2}^2\leq r^2 \text{ or } ||x||\geq\frac{2R^2}r\}$.

Choose a neighborhood $N_R$ of $\psi(\R^n)$ in $\ambientspace$ such that 
$N_R\cap \voisinageinfini= N_R^0\cap \voisinageinfini$, and such that 
$N_R\cap\bouleambiante$ is a tubular neighborhood of $\psi(\R^n)\cap\bouleambiante$. 
Fix a diffeomorphism
%\footnote{Such a $\Theta$ exists according to Lemma \ref{lmS}.}
$\Theta\colon N_R^0\rightarrow N_R$ such that\begin{itemize}
\item $\Theta$ reads as the identity map on 
$ N_R^0\cap \voisinageinfini$, 
\item $\Theta$ induces a bundle isomorphism from $N_R^0\cap \bouleambiante$
to $N_R\cap \bouleambiante$,
\item $\Theta\circ \psi_0=\psi$,
\item there exists a Seifert surface with $\Sigma\cap N_R = \Theta\left(\{x\in N_R^0\mid x_1\geq0, x_2=0\}\right)$.
\end{itemize}

In Section \ref{Section4-1}, we prove the following lemma.
\begin{lm}\label{rect-lemma}
For $n\geq2$ and any rectifiable knot $\psi$ of 
a parallelizable asymptotic integral homology $\R^{n+2}$,
there exists a parallelization $\tau$ such that 
the map $\iota(\tau,\psi)$ of Lemma \ref{obstructionlemma} 
is the constant map with value $\iota_0$.
\end{lm}
This allows us to prove the following.
\begin{lm}\label{gauss-rect}The null-homologous knot $\psi$ is rectifiable if and only if 
there exists a parallelization $\tau$ of $\ambientspace$ such that, 
for any $x\in N_R^0$, $T_x\Theta\circ (\tau_0)_x = \tau_{\Theta(x)}$, 
where $\tau_0$ denotes the canonical parallelization of $\R^{n+2}$ 
introduced in Definition \ref{paral-def}.
\end{lm}

\begin{proof}
Let us first assume $n\geq2$.

If there exists $\tau$ as in the lemma, then $T_{\psi_0(x)}\Theta\circ T_x\psi_0 = \tau_{\Theta(\psi_0(x))}\circ \iota_0 $, so $\tau_{\psi(x)}\circ \iota_0 = T_x\psi$. 
Therefore, $\iota(\tau,\psi) = \iota_0$ and $\psi$ is rectifiable.

Let us prove the converse.
Suppose that $\psi$ is rectifiable, and let $\tau_1$ be a parallelization given by Lemma \ref{rect-lemma} 
such that $\iota(\tau_1, \psi)$ is the constant map of value $\iota_0$.
Recall that if $(X,A)$ and $(Y,B)$ are two topological pairs, $[(X,A), (Y,B)]$ denotes the set of 
homotopy classes of maps from $X$ to $Y$ that map $A$ to $B$. 
Let $\mathbb D^n$ denote the unit ball of $\R^n$.
 The map $\left(x\in N_R^0 \mapsto (\tau_1)_{\Theta(x)}^{-1}\circ T_x\Theta\circ (\tau_0)_x\in GL_{n+2}^+(\R)\right)$ induces 
 an element $\kappa(\tau_1)$ of the set $[(N_R^0, N_R^0\cap \voisinageinfini), (GL_{n+2}^+(\R), I_{n+2}) ]\cong[(N_R^0\cap\bouleambiante, N_R^0\cap\partial \bouleambiante), (GL_{n+2}^+(\R), I_{n+2}) ]$. 
Since $N_R^0\cap \bouleambiante$ is a disk bundle over $\psi_0(\R^n)\cap B(\s^{n+2})=\psi_0(\mathbb D^n)$, the restriction map 
\begin{align*} [(N_R^0\cap\bouleambiante,  N_R^0\cap \partial\bouleambiante),  (GL_{n+2}^+(\R), I_{n+2})]& \\ \rightarrow [(\psi_0(\mathbb D^n), \psi_0(\partial\mathbb D^n)),(GL_{n+2}^+(\R), I_{n+2})]\cong &[(\mathbb D^n,\s^{n-1}),(GL_{n+2}^+(\R), I_{n+2})]\end{align*} 
is an isomorphism, and $[(N_R^0, N_R^0\cap \voisinageinfini), (GL_{n+2}^+(\R), I_{n+2}) ] \cong \pi_n(GL_{n+2}^+(\R), I_{n+2}) $. 
The fibers of the fibration $\left(g\in GL^+_{n+2}(\R) \mapsto g_{|\{0\}^2\times\R^n} \in \inj\right)$ have the homotopy type of $SO(2)$. 
Since $n\geq 2$, the long exact sequence of this fibration provides an injective map $\pi_n(GL_{n+2}^+(\R), I_{n+2}) \rightarrow \pi_n(\inj, \iota_0)$. 
The induced injection $[(N_R^0, N_R^0\cap \voisinageinfini), (GL_{n+2}^+(\R), I_{n+2}) ]\hookrightarrow \pi_n(\inj, \iota_0)$ maps $\kappa(\tau_1)$ to $[\iota(\tau_1,\psi)]$.
By definition of $\tau_1$, $\iota(\tau_1, \psi) = \iota_0$, and $\kappa(\tau_1)$ is trivial. 
This proves the existence of a parallelization $\tau\colon N_R\times \R^{n+2} \rightarrow TN_R$ 
homotopic to $\tau_1$ such that, for any $x\in N_R^0$, 
$\tau_{\Theta(x)}^{-1}\circ T_x\Theta\circ (\tau_0)_x = I_{n+2}$. 
Since ${\tau_1}_{|N_R\times \R^{n+2}}$ extends to $\ambientspace\times \R^{n+2}$, 
and since $\tau$ is homotopic to $\tau_1$, 
$\tau$ also extends to  $\ambientspace\times \R^{n+2}$.

If $n=1$, every long knot is rectifiable by definition. In this case, 
the parallelization $T\Theta\circ (\tau_0)_{|N_R^0\times \R^{n+2}}$ extends to $\spamb$ thanks to the last assertion of \cite[Proposition 18.3]{[Lescop2]}.
\end{proof}
The long exact sequence in the previous proof also yields the following lemma, since $GL_{n+2}^+(\R)$ and $SO(n+2)$ have the same homotopy type.

\begin{lm}\label{lmpin}
For $n\geq 3$, $\pi_n(\inj, \iota_0)$ is isomorphic to $ \pi_n(SO(n+2), I_{n+2})$, and $\pi_2(\mathcal I(\R^2, \R^4), \iota_0)$ is infinite cyclic.
\end{lm}

Suppose now that $\psi$ is rectifiable.
\begin{nt}\label{notations} 
Let $\tau$ be a parallelization as given by Lemma \ref{gauss-rect}. 
For any $r\in [1,R]$, let $E_r$ denote the closure of $\ambientspace \setminus N_r$. With the identification induced by $\Theta$, $\ambientspace$ reads $N_r^0\cup E_r$.

In $N_R=N_R^0$, use the coordinates $x= (x_1, x_2, \overline x)\in \R\times\R\times\R^n$.
 For $r\in [1,R]$, $\mur$ denotes the disk $\{(x_1,x_2, \overline 0) \mid  {x_1}^2+{x_2}^2\leq r^2\}\subset N_r$. For $\theta\in \R$, set $L_\theta^\pm(r) = \{\pm(\rho\cos(\theta), \rho\sin(\theta), \overline 0) \mid \rho \geq \frac{2R^2}r\}$, $\DR = L_\theta^+(r) +L_\theta^-(r)$, and orient these half-lines by $\d \rho$. 
\end{nt}

Let $\Theta_2$ denote the diffeomorphism $C_2(N_R^0) \rightarrow C_2(N_R)$ induced by $\Theta$.
We now define the main objects of this article, which will be used to compute $\Z[]$.

\begin{df}\label{adm-prop}Fix $\theta \in \R$. 
Fix two isotopic Seifert surfaces $\Sigma^+$ and $\Sigma^-$ with disjoint interiors such that 
$\Sigma^\pm\cap N_R = \Theta(({}_\theta\Sigma^\pm)^0\cap N_R^0)$, where $({}_\theta\Sigma^\pm)^0
\subset \R^{n+2}$ is the hypersurface 
$\{\pm(r\cos(\theta), r\sin(\theta), \overline x) \mid\overline x\in \R^n, r\geq 0\}$.
%(Such surfaces exist because $\Theta$ has been chosen as in Lemma \ref{lmS}). 
For any $r\in[1,R]$, let $\Sigma(r)$ denote the submanifold $E_r\cap (\Sigma^+\cup \Sigma^-)$. 
The restriction of the Gauss map of $C_2(\R^{n+2})$ to $C_2(N_1^0)$ 
and the identification $C_2(N_1)\cong C_2(N_1^0)$ given by $\Theta_2$ 
induce a map $G_0\colon C_2(N_1)\rightarrow \s^{n+1}$.

An external propagator $B$ is called $R$-\emph{admissible} 
(with respect to $(\Sigma^+, \Sigma^-, \psi)$) if:
\begin{itemize}
\item $B\cap p_b^{-1}(N_1\times N_1) = \frac12 {G_0}^{-1}(\{-(\cos(\theta), \sin(\theta), \overline 0),
 + (\cos(\theta), \sin(\theta), \overline 0)\})$.
\item If $p_2$ denotes 
the smooth map $p_2\colon \configM \rightarrow C_1(\ambientspace)$ that extends 
the map $\big((x,y)\in C_2^0(\ambientspace)\mapsto y \in \ambientspace\subset C_1(\ambientspace)\big)$, 
and if $c\in B\cap \overline{p_b^{-1}(\psi(\R^n)\times \punct M)}$, then $p_2(c)$ lies in the closure 
$\overline{ \Sigma^-\cup\Sigma^+}$ of $\Sigma^-\cup\Sigma^+$ in $C_1(\ambientspace)$.
\item For any $r\in ]1,R-1]$, 
\[  B\cap p_b^{-1}(N_r\times E_{r+1}) =\frac12\left( \mur[r] \times \Sigma(r+1)
%caspair
-
\overline{p_b^{-1}(\DR[r]\times E_{r+1})}\right).
\]
\item If $T$ denotes the smooth map of $\configM$ such that for any $(x,y)\in C_2^0(\ambientspace)$, $T(x,y)=(y,x)$, then $T(B)=B$. In particular, for any $r\in ]1,R-1]$, \[B\cap p_b^{-1}(E_{r+1}\times N_r) = \frac12\left( \Sigma(r+1) \times \mur[r]
%caspair
+(-1)^{n+1}
\overline{p_b^{-1}( E_{r+1}\times\DR[r])}\right).\]
\end{itemize}
\end{df}

In Section \ref{Section 3}, we prove the following technical lemma. 

\begin{lm}\label{th-prop}

Fix an integer $\degk\geq1$ and a null-homologous rectifiable long knot $\psi$, and define $N_R$ and $N_R^0$ as above.
For any pairwise distinct numbers $(\theta_i)_{i\in\und{2\degk}}$ in $[0,\pi[$, 
and any pairwise isotopic Seifert surfaces $(\Sigma_i^-,\Sigma_i^+)_{i\in\und{2\degk}}$ with disjoint interiors
 such that $\Sigma_i^\pm\cap N_R = ({}_{\theta_i}\Sigma^\pm)^0\cap N_R^0 $, there exist external propagators $\left(B^R(\Sigma_i^+, \Sigma_i^-, \psi)\right)_{i\in\und{2\degk}}$, such that for any $i\in\und{2\degk}$, 
the propagator $B^R(\Sigma_i^+, \Sigma_i^-, \psi)$ is $R$-admissible with respect to $(\Sigma_i^+, \Sigma_i^-,  \psi)$.

Furthermore, we can fix such propagators $\left(B^R(({}_{\theta_i}\Sigma^+)^0, ({}_{\theta_i}\Sigma^-)^0, \psi_0) \right)_{i\in\und{2\degk}}$ for the trivial knot such that $\Theta_2$ maps the chain $B^R\left(({}_{\theta_i}\Sigma^+)^0, ({}_{\theta_i}\Sigma^-)^0, \psi_0\right)\cap p_b^{-1}(N_R^0\times N_R^0)$ to $B^R(\Sigma^+_i, \Sigma^-_i, \psi)\cap p_b^{-1}(N_R\times N_R)$ for any $i\in\und{2\degk}$.
\end{lm}

\subsection{Use of admissible propagators to compute \texorpdfstring{$\Z[]$}{Zk}}\label{Section3.2}
\subsubsection{Three reduction lemmas}\label{setting}

Fix the following setting for Section \ref{Section3.2}.
\begin{sett}\label{def-pert}
\begin{itemize}
\item The integer $\degk\geq 2$ is fixed.
\item The real number $R$ of the previous subsection is fixed to some arbitrary value $R\geq k+1$.
\item The numbers $(\theta_i)_{i\in\indices}$ are such that $0\leq\theta_1<\theta_2<\cdots < \theta_{2\degk}<\pi$.
\item For any $i\in\indices$, $(\Sigma_i^\pm)^0= ({}_{\theta_i}\Sigma^\pm)^0$.
\item $\Sigma$ is a fixed Seifert surface for $\psi$.
\item $(\Sigma_i^\pm)_{i\in \indices}$ are pairwise parallel\footnote{They only meet along the knot.} Seifert surfaces for $\psi$, all isotopic to $\Sigma$, 
such that $\Sigma_i^\pm \cap N_{R} = (\Sigma_i^{\pm})^0\cap N_{R}$,
and such that for any $(i,\epsilon)\in\und{2\degk}\times\{\pm\}\setminus\{(1,+)\}$, 
$\Sigma_i^\epsilon\cap E_{1}$ is obtained from $\Sigma_1^+\cap E_{1}$ by pushing it in the positive normal direction, 
so that the order of these surfaces in the positive normal direction is 
$(\Sigma_1^+, \ldots, \Sigma_{2k}^+, \Sigma_1^-, \ldots, \Sigma_{2k}^-)$.
\item With the notations of Lemma \ref{th-prop}, $F_*^0=(A_i, B_i^0)_{i\in\indices}$ is a $\degk$-family of propagators of $(\R^{n+2}, \tau_0)$ in general position for $\psi_0$, such that for any $i\in\indices$, $B_i^0$ is an arbitrarily small perturbation of $B^R((\Sigma^+_i)^0, (\Sigma^-_i)^0,  \psi_0)$.
\item With the notations of Lemma \ref{th-prop}, $F_*=(A_i, B_i)_{i\in\indices}$ is a $\degk$-family of propagators of $(\ambientspace, \tau)$ in general position for $\psi$ such that for any $i\in\indices$, $B_i$ is an arbitrarily small perturbation of $B^R(\Sigma^+_i, \Sigma^-_i,  \psi)$ and $B_i\cap {p_b}^{-1}(N_{R}\times N_{R})$ is the image of $B_i^0\cap {p_b}^{-1}(N_{R}^0\times N_{R}^0)$ under the identification $\Theta_2$. 

\end{itemize}

\end{sett}

For any edge $e$ of a numbered degree $k$ BCR diagram $(\Gamma,\sigma)$ as in Definition \ref{Def-numb}, define the chains $D_{e,\sigma}\subset C_\Gamma(\psi)$ and $D^0_{e,\sigma}\subset C_\Gamma(\psi_0)$ as \[D_{e,\sigma}= \begin{cases} {p_e}^{-1}(A_{\sigma(e)})& \text{if $e$ is internal,} \\ {p_e}^{-1}(B_{\sigma(e)})& \text{if $e$ is external,}\end{cases} \text{ and } D_{e,\sigma}^0=\begin{cases}{p_e}^{-1}(A_{\sigma(e)})&\text{if $e$ is internal,}\\ {p_e}^{-1}(B^0_{\sigma(e)})&\text{if $e$ is external.}\end{cases}\]

\begin{lm}\label{transv1}

If $\Gamma$ has an external edge from an internal vertex to an internal vertex, then, for any numbering $\sigma$, $\bigcap\limits_{f\in E(\Gamma)} D_{f, \sigma}=\emptyset$, and $\bigcap\limits_{f\in E(\Gamma)} D^0_{f, \sigma}=\emptyset$. 
\end{lm}
\begin{proof}
We first ignore the perturbations of the external propagators. 
Let $e= (v,w)$ be an external edge between two internal vertices of a numbered BCR diagram $(\Gamma,\sigma)$. For a configuration $c$ in $ D_{e, \sigma}$, set $c(v)=p_v(c)$ and $c(w)=p_w(c)$, where $p_v$ and $p_w$ are the maps defined in Theorem \ref{cgamma}. Since $v$ and $w$ are internal, we have $G_0(p_e(c))\in\{0\}^2\times\s^{n-1}$.
On the other hand, by Definition \ref{adm-prop} of admissible propagators, since $(c(v), c(w))\in N_1\times N_1$, we have $G_0(p_e(c)) = \pm (\cos(\theta_{\sigma(e)}), \sin(\theta_{\sigma(e)}), \overline 0)$.
Thus, $D_{e,\sigma}=\emptyset$ and $\bigcap\limits_{f\in E(\Gamma)} D_{f, \sigma}=\emptyset$. Similarly, $\bigcap\limits_{f\in E(\Gamma)} D^0_{f, \sigma}=\emptyset$. 

Now, note that the property $D_{e,\sigma}= \emptyset$ is stable under small perturbations 
since the map $p_e$ is continuous and since the propagators are compact. \qedhere
\end{proof}

\begin{lm}\label{transv2}
Let $\Gamma\in \graphes\setminus\{\Gamma_\degk\}$, where $\Gamma_\degk$ is the degree $\degk$ BCR diagram of Figure \ref{Gk}.
For any numbering $\sigma$ of $\Gamma$, $\bigcap\limits_{e\in E(\Gamma)} D_{e, \sigma}=\emptyset$ and $\bigcap\limits_{e\in E(\Gamma)} D^0_{e, \sigma}=\emptyset$. 
\end{lm}

\begin{proof}
Fix a numbering $\sigma$.
If $\Gamma$ has only internal vertices, conclude with Lemma \ref{transv1}. 
If $\Gamma\neq \Gamma_\degk$ and $\sommetsexternes\neq\emptyset$, then $\Gamma$ contains a maximal sequence $(w_1,\ldots, w_p)$ of consecutive external vertices with $p\in\und{k-1}$ like in Figure \ref{Figtr2}. Let $a$ be the bivalent vertex such that there is an external edge from $a$ to $w_1$ and let $b$ be the bivalent vertex such that there is an external edge from $w_p$ to $b$, and note that $a\neq b$.
\begin{figure}[H]
\centering
\begin{tikzpicture}[yscale=-1]
\edgi (-.9,0)-- (-.1,0);
\fill (0,0) \crc (8,0) \crc; 
\draw (0,0) ++ (60:.35) node {$a$} (8,0) ++ (60:.35) node {$b$};
\draw (1,0) \crc (4,0) \crc (7,0) \crc;
\draw (1,0) ++ (30:.45) node {$w_1$} (4,0) ++ (30:.45) node {$w_i$} (7,0) ++ (30:.45) node {$w_p$};
\fill (1,1) \crc (4,1) \crc (7,1) \crc;
\draw (1,1) ++ (30:.35) node {$v_1$} (4,1) ++ (30:.35) node {$v_i$} (7,1) ++ (30:.35) node {$v_p$};
\edge (0.1, 0) -- (0.9, 0);
\draw (0.5, -.3) node{$f_1$};
\draw[dashed] (1.1,0)--(1.8,0);
\draw (1.5, -.3) node{$f_2$};
\draw[dotted] (2.3, 0)-- (2.8, 0);
\draw[dotted] (5.3, 0)-- (5.8, 0);
\edge (3.1, 0)--(3.9,0);
\draw (3.5, -.3) node{$f_i$};
\edge (4.1, 0)--(4.9,0);
\draw (4.5, -.3) node{$f_{i+1}$};
\edge (1, 0.9) -- (1, 0.1);
\draw (.7, .5) node{$\ell_1$};
\edge (4, 0.9) -- (4, 0.1);
\draw (3.7, .5) node{$\ell_i$};
\edge (7, 0.9) -- (7, 0.1);
\draw (6.7, .5) node{$\ell_p$};
\edge (6.1, 0)--(6.9,0);
\draw (6.5, -.3) node{$f_p$};
\edge (7.1, 0)--(7.9,0);
\draw (7.5, -.3) node{$f_{p+1}$};
\draw(8.1,0)-- (8.9,0);
\end{tikzpicture}
\caption{Notations for Lemma \ref{transv2}}
\label{Figtr2}
\end{figure}
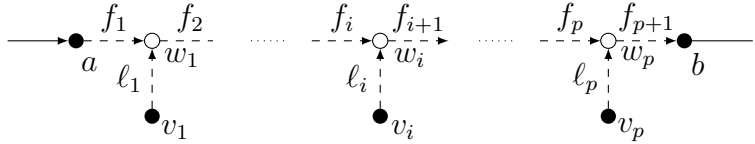

As in the previous proof, we first ignore the perturbations.
Let $c$ in $\bigcap\limits_{e\in E(\Gamma)} D_{e, \sigma}$. For any $i\in\und p$, since $p_{\ell_i}(c) \in B_{\sigma(\ell_i)}$ and $v_i$ is internal, $c(w_i)=p_{w_i}(c)$ lies in the closure $\overline{\Sigma^+_{\sigma(\ell_i)}\cup\Sigma^-_{\sigma(\ell_i)}}$ of $\Sigma^+_{\sigma(\ell_i)}\cup\Sigma^-_{\sigma(\ell_i)}$ in $ C_1(\ambientspace)$. Similarly, since $a$ is internal, $c(w_1)\in \overline{\Sigma^+_{\sigma(f_1)}\cup\Sigma^-_{\sigma(f_1)}}$. Then, $c(w_1)$ lies in the closure $\overline{\psi(\R^n)}$ of $\psi(\R^n)$ in $C_1(\ambientspace)$. The same argument now proves that $c(w_2)$ is in $\overline{\psi(\R^n)}$. By induction, $c(w_i)$ lies in $\overline{\psi(\R^n)}$ for any $i\in\und p$.

By construction of $C_{\Gamma}(\psi)$, $c$ is the limit of configurations $(c_t)_{t\in ]0,1]}$ 
of $C_{\Gamma}^0(\psi)$ when $t$ approaches $0$. 
Since $c(w_i)$ is in $\overline{\psi(\R^n)}$ for any $i\in\und p$ when $t=0$, 
we can assume that $c_t$ maps all the vertices $(w_i)_{i\in\und p}$ in $N_1\subset \R^{n+2}$, 
for any $t\in]0,1]$. 
For any $t\in ]0,1]$, the vector $c_t(b)-c_t(a)$ is the sum of the vectors 
$c_t(w_1)-c_t(a), c_t(w_2)-c_t(w_1), \ldots, c_t(w_p)-c_t(w_{p-1})$, and $c_t(b)-c_t(w_p)$. 
Since the propagators are admissible, 
and since $c(a)$, $c(b)$ and the $(c(w_i))_{i\in\und p}$ are in $N_1$, 
$G_0(c_t(a), c_t(b))$ is a linear combination 
of the vectors $((\cos(\theta_{\sigma(f_i)}), \sin(\theta_{\sigma(f_i)}), \overline 0))_{i\in\und{p+1}}$. 
Thus, $G_0(c_t(a), c_t(b))$ is in $\s^1\times \{0\}^n$ for any $t\in ]0,1]$. 
But since $a$ and $b$ are internal, for any $t\in ]0,1]$, 
$G_0(c_t(a), c_t(b))$ reads $(0, 0, \overline x_t)$ for some $\overline x_t\in \s^{n-1}$. 
This is a contradiction, so $\bigcap\limits_{e\in E(\Gamma)} D_{e, \sigma}=\emptyset$. 
Similarly, $\bigcap\limits_{e\in E(\Gamma)} D^0_{e, \sigma}=\emptyset$.

This property is stable under small enough perturbations since the propagators and configuration spaces are compact
and since the maps $p_e$ are continuous.
\end{proof}
The two above lemmas allow us to reduce our study to the graph $\Gamma_k$. 
The following lemma will help us to study 
the contribution of $\Gamma_k$ in the next subsection.
\begin{lm}\label{mainlemma}

Let $\Gamma_k$ be the BCR diagram of Figure \ref{Gk}.
If $c$ is a configuration of $\bigcap\limits_{e\in E(\Gamma_k)} D_{e,\sigma}$ (resp. of $\bigcap\limits_{e\in E(\Gamma_k)}D^0_{e,\sigma}$), and if there exists $j\in\zk$ such that $c(w_j)\in E_{k+1}$ (resp. $c(w_j)\in E_{k+1}^0$), then $c(w_i)\not\in N_2$ for any $i\in\zk$.
\end{lm}

\begin{proof}It suffices to prove the statement on $\bigcap\limits_{e\in E(\Gamma_k)} D_{e,\sigma}$, the proof for $\bigcap\limits_{e\in E(\Gamma_k)} D^0_{e,\sigma}$ is the same.
Let us first ignore the perturbations, and assume without loss of generality that $j=k$, so that $c(w_k)\in E_{k+1}$.

Let us prove that $c(w_{k-1})\not\in N_k$. Since $v_{k-1}$ and $v_k$ are internal, 
$c(w_{k-1}) \in \overline{\Sigma^+_{\sigma(\ell_{k-1})}  \cup \Sigma^-_{\sigma(\ell_{k-1})}     }$
and $c(w_k)\in \Sigma_{\sigma(\ell_k)}(k+1)$.
Since the surfaces $\Sigma_{\sigma(f_{k-1})}(k+1)$ and $\Sigma_{\sigma(\ell_k)}(k+1)$ are disjoint,
$c(w_k) \not\in \Sigma_{\sigma(f_{k-1})}(k+1)$.
Since $\overline{L_{\theta_{\sigma(f_{k-1})}}(k)}$ and $\overline{\Sigma^+_{\sigma(\ell_{k-1})}  \cup \Sigma^-_{\sigma(\ell_{k-1})}     }$ 
do not intersect in $C_1(\ambientspace)$,
$c(w_{k-1})\not\in \overline{L_{\theta_{\sigma(f_{k-1})}}(k)}$, 
so that $p_{f_{k-1}}(c)= (c(w_{k-1}), c(w_k)) \not\in B_{\sigma(f_{k-1})} \cap (N_k\times E_{k+1})$
and $c(w_{k-1}) \not\in N_k$.
By induction, we prove that $c(w_{i})\not\in N_{i+1}$ for any $i\in \zk$.

Since the set ${p_{w_k}}^{-1}(E_{k+1})\cap\left( \bigcup\limits_{j\in \zk}{p_{w_j}}^{-1}(N_2)\right)$ is compact, the property of the lemma is stable under small perturbations and the lemma is therefore true for small enough perturbations.\qedhere

\end{proof}

\subsubsection{A first formula for \texorpdfstring{$Z_k$}{Zk}}\label{Subs333}

We prove the following lemma until the end of this subsection.
 \begin{lm}\label{contributionGamma1}Label $\Gamma_k$ as in Figure \ref{Gk}.
Fix a pair $(\Ba, \Bb)$ of dual bases of $\overline{H}_*(\Sigma_1^+)$ and set $\Ba=([\bb_i^\dimd])_{\dimd\in\und n, i\in \und{b_\dimd}}$ and $\Bb=([\ba_i^\dimd])_{\dimd\in\und n, i \in \und{b_\dimd}}$. For any $i \in \zk$, set \[\plusun{i}=\begin{cases} i+1 & \text{if $i<k$,}\\ 1 & \text{if $i=\degk$.}\end{cases}\]For any $i\in\zk$, any numbering $\sigma$ of $\Gamma_k$, and any $\hat\epsilon\colon \zk \rightarrow \{\pm1\}$, set $\sigma_{\hat\epsilon}(\ell_i)= \sigma(\ell_i) + (1-\hat\epsilon(i))k$ and \[\epsilon_{\hat\epsilon,\sigma}(i)= \begin{cases}
+1 & \text{if $\sigma_{\hat\epsilon}(\ell_i)< \sigma_{\hat\epsilon}(\ell_{\plusun{i}})$,}\\ -1 & \text{otherwise.}\end{cases}\]

The difference $Z_k(\psi) = Z^{F_*}_k(\psi)-Z^{F_*^0}_k(\psi_0) $ reads\footnote{In the even-dimensional case, the formula holds for $Z^{F_*}_k(\psi)-Z^{F_*^0}_k(\psi_0)$, 
regardless of Theorem \ref{thu*}, which implies that this quantity is $Z_k(\psi)$.}
\[ \Z[](\psi) = \frac1{k(2k)!2^k}\sum\limits_{\sigma\in \mathrm{Num}(\Gamma_k)}\sum\limits_{\hat\epsilon\colon \und k \rightarrow \{\pm\}}\sum\limits_{\dimd\in\und n}\sum\limits_{p \colon \zk\rightarrow\underline{b_\dimd}}
(-1)^{d+k+n}
 \prod\limits_{j\in\zk} \lk\left(\bb_{p(j)}^\dimd, (\ba_{p(\plusun{j})}^{n+1-\dimd})^{\epsilon_{\hat\epsilon,\sigma}(j)}\right)  ,\]
where $\mathrm{Num}(\Gamma_k)$ denotes the set of numberings of $\Gamma_k$.
\end{lm}
As an immediate corollary of Theorem \ref{conn-sum}, $Z_k(\psi_0)= 0$.
For any $(\Gamma, \sigma)\in \graphesnum$, set $\Delta_{\Gamma,\sigma}\Z[]= I_{F_*}(\Gamma,\sigma, \psi) - I_{F_*^0}(\Gamma,\sigma,\psi_0)$, so that \[\Z[](\psi) = \Z[](\psi) -\Z[](\psi_0) = \frac1{(2k)!}\sum\limits_{(\Gamma,\sigma)\in \graphesnum}\Delta_{\Gamma,\sigma}\Z[].\] 
Lemma \ref{transv2} implies that $\Delta_{\Gamma,\sigma}\Z[]=0$ if $\Gamma$ is not isomorphic to $\Gamma_k$. 
Because of the symmetry of $\Gamma_k$, each numbered graph $(\Gamma,\sigma)$ where $\Gamma$ is isomorphic to $\Gamma_k$ yields $k$ numberings of $\Gamma_k$, so that 
\[\Z[](\psi) = \frac1{k(2k)!} \sum\limits_{\sigma\in\Num(\Gamma_k)} \Delta_{\Gamma_k,\sigma}\Z[].
\]

Since $R\geq \degk+1$, $B_i\cap {p_b}^{-1}(N_{\degk+1}\times N_{\degk+1})= \Theta_2(B_i^0\cap{p_b}^{-1}(N_{\degk+1}^0\times N_{\degk+1}^0))$. This yields the following lemma.
\begin{lm}\label{lmenk}
Let $\langle\cdot,\ldots, \cdot\rangle_X$ denote the algebraic intersection of several chains of a manifold $X$ such that their codimensions add up to $\dim(X)$.
Let $X_1(\Gamma_\degk)$ (respectively $X^0_1(\Gamma_\degk))$ denote the subset of $C_{\Gamma_\degk}(\psi)$ (respectively $C_{\Gamma_\degk}(\psi_0)$), whose elements are the configurations that map at least one vertex to $E_{\degk+1}$. 
For any edge $e$ of $\Gamma_k$, set $ D_{e, \sigma}^{(1)}= D_{e,\sigma} \cap X_1(\Gamma_\degk)$ and $D_{e, \sigma}^{(1),0}= D_{e,\sigma}^0 \cap X^0_1(\Gamma_\degk)$. 

The chains $(D_{e,\sigma}^{(1)})_{e\in E(\Gamma_\degk)}$ and $(D^{(1),0}_{e,\sigma})_{e\in E(\Gamma_\degk)}$ are transverse, and
\begin{align*}\Delta_{\Gamma_k,\sigma}\Z[]= \epsilon_{n,k}\langle D^{(1)}_{\ell_1,\sigma},\ldots,  D^{(1)}_{\ell_k,\sigma},&
D^{(1)}_{f_1,\sigma},\ldots,  D^{(1)}_{f_k,\sigma}
\rangle_{X_1(\Gamma_\degk)}\\
&-\epsilon_{n,k}\langle D^{(1),0}_{\ell_1,\sigma}, \ldots, D^{(1),0}_{\ell_k,\sigma},
D^{(1),0}_{f_1,\sigma}, \ldots,  D^{(1),0}_{f_k,\sigma}
\rangle_{X_1^0(\Gamma_\degk)},\end{align*}
where the sign $\epsilon_{n, k}$ is $+1$ if $n$ is odd, and $(-1)^{k-1}\epsilon_{n,k}$ is the signature of the permutation $\sigma_{F,L}=\begin{pmatrix}
f_1 & \ell_1&  f_2 & \ell_2 &  \ldots& \ldots &  f_k&  \ell_k \\
\ell_1 & \ell_2 & \ldots &\ldots  &\ell_k &f_1 & \ldots & f_k\end{pmatrix}$ if $n$ is even.
\end{lm}
\begin{proof}
When $n$ is even, the sign follows from Lemma \ref{orex}.\qedhere
\end{proof}

The rest of this section is devoted to the computation of the above intersection number.

\begin{lm}\label{exprY}
Let $Y(\sigma)$ denote the manifold $\prod\limits_{i\in\zk} \Sigma_{\sigma(\ell_i)}(2)$. Similarly, set $Y^0(\sigma) = \prod\limits_{i\in\zk} \Sigma^0_{\sigma(\ell_i)}(2)$. For any $i\in\zk$, set $Y_{i}(\sigma)= \Sigma_{\sigma(\ell_i)}(2)\times \Sigma_{\sigma(\ell_{\plusun{i}})}(2)$ and $B_{i,\sigma} = B_{\sigma(f_i)}\cap Y_i(\sigma)$ and similarly define $Y_i^0(\sigma)$ and $B^0_{i,\sigma} = B^0_{\sigma(f_i)}\cap Y_i^0(\sigma)$. Let $\pi_i$ denote the projection map $Y(\sigma)\rightarrow Y_i(\sigma)$, and set $\Pii=\pi_i^{-1}(B_{i,\sigma})$. Similarly define $\pi_i^0\colon Y^0(\sigma)\rightarrow Y_i^0(\sigma)$ and $\Pii^0$.

 The chains $(\Pii)_{i\in \zk}$ are transverse, the chains $(\Pii^0)_{i\in \zk}$ are transverse, and \[\Delta_{\Gamma_k,\sigma}Z_k=(-1)^{1+k(n+1)}\frac{1}{2^k}\left( \langle \Pii[{1}], \cdots, \Pii[{k}]\rangle_{Y(\sigma)} - \langle \Pii[{1}]^0, \ldots , \Pii[{k}]^0\rangle_{Y^0(\sigma)}\right).\] 

\end{lm}
\begin{proof}
Let $X_2(\Gamma_\degk)$ denote the set of configurations such that all the external vertices are mapped to $E_2$, and set $D_{e,\sigma}^{(2)}= 2 D_{e,\sigma} \cap X_2(\Gamma_\degk)$. For simplicity, we assume that $D_{e,\sigma}^{(2)}$ is a manifold. 
Similarly define $X_2^0(\Gamma_\degk)$ and $D_{e,\sigma}^{(2),0}$. Lemma \ref{mainlemma} ensures that the intersection of the supports of the $D_{e,\sigma}^{(1)}$ is contained in $\bigcap\limits_{e\in E(\Gamma_\degk)} D_{e,\sigma}^{(2)}$.
Since the propagators are chosen as in Lemma \ref{th-prop},\[\left\langle  (D_{e,\sigma}^{(1)} )_{e\in E(\Gamma_\degk)}\right\rangle_{X_1(\Gamma_\degk)} = \frac1{2^{2\degk}}\left\langle (D_{e,\sigma}^{(2)})_{e\in E(\Gamma_\degk)}\right\rangle_{X_2(\Gamma_\degk)} +  R, \]
where $R$ is independent of the knot.
This implies that
\[\Delta_{\Gamma_\degk,\sigma}\Z[] =\frac1{2^{2\degk}}\left( \left\langle (D_{e,\sigma}^{(2)})_{e\in E(\Gamma_\degk)}\right\rangle_{X_2(\Gamma_\degk)} - \left\langle (D_{e,\sigma}^{(2),0})_{e\in E(\Gamma_\degk)}\right\rangle_{X_2^0(\Gamma_\degk)}\right).\]

Let us check that $\left(\phi\colon c\in \bigcap\limits_{i\in\zk} D_{\ell_i,\sigma}^{(2)} \mapsto (c(w_i))_{i\in\zk}\in Y(\sigma)\right)$ is well-defined and let us see how it acts on the orientation.

If $c\in  \bigcap\limits_{i\in\zk} D_{\ell_i,\sigma}^{(2)}$, then $(c(v_i), c(w_i))\in B_{\sigma(\ell_i)}$ for any $i\in \zk$. Therefore, for any $i\in \zk$, since $v_i$ is internal and because of Definition \ref{adm-prop} of admissible propagators, $c(w_i)\in \Sigma_{\sigma(\ell_i)}(2)$. This implies that $\phi$ is actually valued in $Y(\sigma)$. It is a diffeomorphism since the disk $\mur[1]$ meets the knot in exactly one point.

Let us study how $\phi$ acts on the orientations. Let $n_i(x)$ denote the positive normal direction to $\Sigma_{\sigma(\ell_i)}$ at $x$. 
The normal bundle to $\bigcap\limits_{i\in\zk} D_{\ell_i,\sigma}^{(2)}$ at $c$ is \[
\N_c\left( \bigcap\limits_{i\in\zk} D_{\ell_{i},\sigma}^{(2)}\right)= \prod\limits_{i\in \zk} \left(T_{c(v_i)}\psi(\R^n)\times \R .n_i(c(w_i))\right),\] and we proved in Lemma \ref{orex} that $C_{\Gamma_\degk}(\psi)$ is oriented as \begin{eqnarray*}
T_c C_{\Gamma_\degk}(\psi) & = & (-1)^{k+n(k+1)} \prod\limits_{i\in\zk} \left( T_{c(v_i)} \psi(\R^n) \times T_{c(w_i)} \spamb  \right)\\
& = & (-1)^{k+n(k+1)} \prod\limits_{i\in\zk} \left( T_{c(v_i)} \psi(\R^n) \times \R. n_i(c(w_i)) \times T_{c(w_i)} \Sigma_{\sigma(\ell_i)} \right)\\
&=& \eta_{n,k}
%(-1)^{1+k(n+1)} \epsilon_{n,k} 
\left(\prod\limits_{i\in\zk} \left( T_{c(v_i)} \psi(\R^n) \times \R. n_i(c(w_i)) \right)\right) \times\left( \prod\limits_{i\in\zk} T_{c(w_i)} \Sigma_{\sigma(\ell_i)} \right),
\end{eqnarray*}
where the last equality involves a sign $\eta_{n,k}=\pm1$. If $n$ is odd, then $\eta_{n,k} = -1$.
If $n$ is even, then $\eta_{n,k} = (-1)^k \times \epsilon(\sigma'_{F,L})$, where $
\sigma'_{F, L} = \begin{pmatrix}
f_1 & f_2 & \ldots &\ldots  &f_k &\ell_1 & \ldots & \ell_k\\
f_1 & \ell_1&  f_2 & \ell_2 &  \ldots& \ldots &  f_k&  \ell_k
\end{pmatrix}$.
Note that $\epsilon(\sigma'_{F,L}) = 
(-1)^k \epsilon( \sigma_{F,L})$ with the permutation $\sigma_{F,L}$
of Lemma \ref{lmenk}, 
so that $\eta_{n,k} =\epsilon(\sigma_{F,L}) =(-1)^{k-1} \epsilon_{n,k}$ when $n$ is even. 
This proves that $D_{L,\sigma}=\bigcap\limits_{i\in\zk} D_{\ell_i,\sigma}^{(2)}$ 
is oriented as $(-1)^{1+k(n+1)}\epsilon_{n,k}\phi^{-1}(Y(\sigma))$. 

Let us state without proof the following easy lemma, which we will use in the rest of this proof.

\begin{lm}
Let $P$ and $Q$ be two transverse oriented submanifolds of an oriented manifold $R$. Let $\N^Q(Q\cap P)$ denote the normal bundle of $Q\cap P$ as a submanifold of $Q$. For any $x\in Q\cap P$, \[\N^Q_x(Q\cap P) =\N_x P.\]
\end{lm}
For any $i\in\zk$, the coorientation of the submanifolds $ D_{L,\sigma}\cap D_{f_i, \sigma}^{(2)}$ in $D_{L, \sigma}$ and of the submanifolds $D_{f_i, \sigma}^{(2)}$ in $\confignoeud$ coincide.
%On the other hand, $T_c X_2(\Gamma_{\degk})= \N_c D_{L,\sigma} \times T_c D_{L, \sigma} = (-1)^{k(n+1)} T_c D_{L,\sigma} \times \N_c D_{L, \sigma}$.
%, since $\dim(\N_c D_{L,\sigma})=k(n+1)\mod2$, and $\dim(T_c D_{L,\sigma})=2kn - k(n+1) = k(n+1)\mod 2$.
This yields 
%$ \left\langle  D_{L,\sigma}\cap D_{f_1, \sigma}^{(2)}, \ldots,  D_{L,\sigma}\cap D_{f_\degk, \sigma}^{(2)} 
%\right\rangle_{D_{L,\sigma}} = \langle  D_{L,\sigma}, D^{(2)}_{f_1,\sigma},\ldots,  D^{(2)}_{f_k,\sigma} 
%\rangle_{X_2(\Gamma_\degk)}.$ Therefore,
\begin{align*}
\left\langle\left( D_{e, \sigma}^{(2)}\right)_{e\in E(\Gamma_\degk)}\right\rangle_{X_1(\Gamma_\degk)} 
&= 
\epsilon_{n,k}\langle  D_{L,\sigma},  D^{(2)}_{f_1,\sigma},\ldots,  D^{(2)}_{f_k,\sigma}
\rangle_{X_2(\Gamma_\degk)}\\
&=
\epsilon_{n,k}\left\langle D_{L,\sigma}\cap  D_{f_1, \sigma}^{(2)}, \ldots,  D_{L,\sigma}\cap D_{f_\degk, \sigma}^{(2)}
\right\rangle_{D_{L,\sigma}}\\
&= (-1)^{1+k(n+1)} \left\langle \phi( D_{L,\sigma}\cap D_{f_1, \sigma}^{(2)}), \ldots, 
\phi( D_{L,\sigma}\cap D_{f_\degk, \sigma}^{(2)})
\right\rangle_{Y(\sigma)},
\end{align*}
where the third equality comes from the action of $\phi$ on the orientation.

Now, for any $i\in\zk$, $\phi( D_{L,\sigma}\cap D_{f_i, \sigma}^{(2)})$ is cooriented as $D_{f_i,\sigma}=p_{f_i}^{-1}(B_{\sigma(f_i)})$, i.e. as $B_{\sigma(f_i)}$ in $\configM$. On the other hand, $\Pii$ is cooriented as $B_{\sigma(f_i)}\cap (\Sigma_{\sigma(\ell_i)}(2)\times \Sigma_{\sigma(\ell_{\plusun{i}})}(2))$ in $\Sigma_{\sigma(\ell_i)}(2)\times \Sigma_{\sigma(\ell_{\plusun{i}})}(2)$, i.e. as $B_{\sigma(f_i)}$ in $\configM$.
Because of the $2$ factors in the definition of the $D_{f_i, \sigma}^{(2)}$, this yields \[\left\langle\left( D_{e, \sigma}^{(2)}\right)_{e\in E(\Gamma_\degk)}\right\rangle_{X_2(\Gamma_\degk)}
= (-1)^{1+k(n+1)}{2^\degk}\langle \Pii[1], \cdots, \Pii[\degk] \rangle_{Y(\sigma)}.\]

Similarly, $\left\langle\left( D_{e, \sigma}^{(2),0}\right)_{e\in E(\Gamma_k)}\right\rangle_{X_2^0(\Gamma_\degk)}
= (-1)^{1+k(n+1)}{2^\degk}\langle \Pii[1]^0, \cdots, \Pii[\degk]^0 \rangle_{Y^0(\sigma)}.$ \qedhere
\end{proof}

We are going to define a manifold $\overline {Y(\sigma)}$ without boundary in which the chains $P_{i,\sigma}$ and $P_{i,\sigma}^0$ embed, in order to compute intersection numbers of the previous chains with boundaries in terms of intersection numbers of cycles inside one common manifold.

\begin{lm}\label{hpi}
For $i\in\und{2\degk}$, let $S^\pm_{i}$ denote the gluing of $\Sigma^\pm_{i}\cap E_2$ and $-\left((\Sigma^\pm_{i})^0\cap E_2^0\right)$ along their boundaries, set $S_{i}= S^+_{i}\sqcup S^-_{i}$, and let $S_{i}^{\leq3}$ denote the set of points of $S_{i}$ that come from a point in $N_3$ or $N_3^0$ before the gluing. For any $i\in\zk$, set $\overline{Y_i}(\sigma) = S_{\sigma(\ell_i)}\times S_{\sigma(\ell_{\plusun{i}})}$, and set $\overline Y(\sigma) = \prod\limits_{i\in\zk} S_{\sigma(\ell_i)}$. There exist canonical projection maps $\overline{\pi_i} \colon \overline Y(\sigma) \rightarrow \overline{Y_i}(\sigma)$ for any $i\in\zk$.
The chains $(B_{i,\sigma})_{i\in\zk}$ and $(B_{i,\sigma}^0)_{i\in\zk}$ naturally embed into $\overline {Y_i}(\sigma)$, so that the chains $(P_{i,\sigma})_{i\in\zk}$ and $(P_{i,\sigma}^0)_{i\in\zk}$ naturally embed into $\overline {Y}(\sigma)$. With these notations, \begin{itemize}
\item the boundaries $\partial B_{i,\sigma}$ and $\partial B_{i, \sigma}^0$ lie in $ S_{\sigma(\ell_i)}^{\leq3}\times S_{\sigma(\ell_{\plusun{i}})}^{\leq3}$,
\item for any $i\in\zk$, there exists an $(n+1)$-chain $\hat B_{i, \sigma}$ in $S_{\sigma(\ell_i)}^{\leq3}\times S_{\sigma(\ell_{\plusun{i}})}^{\leq3}$ such that $\partial \hat B_{i,\sigma} = \partial B_{i,\sigma}^0-\partial B_{i,\sigma}$.
\end{itemize} The manifold $S_{\sigma(\ell_i)}^{\leq3}\times S_{\sigma(\ell_{\plusun{i}})}^{\leq3}$ does not depend on the knot. The chains $(\hat{B}_{i,\sigma})_{i\in\zk}$ can be chosen such that they do not depend on the knot, either.

\end{lm}
\begin{proof}Fix $i\in\zk$.
Since the Seifert surfaces are parallel, the chain $B_{i,\sigma}$ does not meet $Y_i(\sigma)\cap(N_2\times E_3)$ or $Y_i(\sigma)\cap(E_3\times N_2)$. 
The chain $\partial B_{i,\sigma}$ is therefore contained in $\partial Y_i(\sigma) \cap (N_3\times N_3)$. 
The same argument proves that $\partial B_{i,\sigma}^0$ is contained in $\partial Y_i^0(\sigma) \cap (N_3^0\times N_3^0)$. Therefore, the chain $Q_{i,\sigma} = \partial B_{i,\sigma}^0-\partial B_{i,\sigma}$ is a cycle of $S_{\sigma(\ell_i)}^{\leq3}\times S_{\sigma(\ell_{\plusun{i}})}^{\leq3}$. Since the propagators are standard inside ${p_b}^{-1}(N_3\times N_3)$, the cycle $Q_{i,\sigma}$ does not depend on the knot.

For any $j\in\zk$, let  $\ell^\pm_{j}$ denote the boundary $\partial (\Sigma_{j}^\pm \cap E_2)$, which is involved in the gluing in the definition of $S^\pm_{j}$, and let $x_{j}^\pm\in\ell^\pm_{j}$. 
Since the product $S_{\sigma(\ell_i)}^{\leq3}\times S_{\sigma(\ell_{\plusun{i}})}^{\leq3}$ retracts onto $(\ell^+_{\sigma(\ell_i)}\sqcup\ell^-_{\sigma(\ell_{i})})\times(\ell^+_{\sigma(\ell_{\plusun{i}})}\sqcup\ell^-_{\sigma(\ell_{\plusun{i}})})$, $H_n(S_{\sigma(\ell_i)}^{\leq3}\times S_{\sigma(\ell_{\plusun{i}})}^{\leq3}) = \mathbb Q^8$, with a basis given by the eight spheres $[\ell_{\sigma(\ell_{i})}^\epsilon\times  x^{\epsilon'}_{\sigma(\ell_{\plusun{i}})}]$ and $ [x^\epsilon_{\sigma(\ell_i)} \times \ell_{\sigma(\ell_{\plusun{i}})}^{\epsilon'}]$ for $\epsilon, \epsilon'\in\{\pm\}$. Let $(s_j)_{j\in\und{8}}$ denote these spheres.

The manifold $S_{\sigma(\ell_i)}^{\leq3}\times S_{\sigma(\ell_{\plusun{i}})}^{\leq3}$ contains $T_1^0 = (S_{\sigma(\ell_i)}^{\leq3}\times S_{\sigma(\ell_{\plusun{i}})}^{\leq3})\cap  (\Sigma^0_{\sigma(\ell_i)}(2)\times\Sigma^0_{\sigma(\ell_{\plusun{i}})}(2) )$ and $T_1 = (S_{\sigma(\ell_i)}^{\leq3}\times S_{\sigma(\ell_{\plusun{i}})}^{\leq3})\cap  (\Sigma_{\sigma(\ell_i)}(2)\times\Sigma_{\sigma(\ell_{\plusun{i}})}(2) )$. $T_1^0$ and $T_1$ are diffeomorphic to each other because the surfaces $\Sigma_j^0$ and $\Sigma_j$ are identical inside $N_3$ for any $j$. Denote by $\Theta_T\colon T_1^0\rightarrow T_1$ the induced diffeomorphism.

The eight spheres $(s_j)_{j\in \und{8}}$ also define 
bases $([s_j])_{j\in\und{8}}$ of $H_n(T_1^0 )$ and $H_n( T_1)$. 
The definition of the spheres $(s_j)_{j\in\und{8}}$ implies that $\Theta_T(s_j) = s_j$ for any $j\in\und{8}$. 
Since the propagators do not depend on the knot inside $N_3\times N_3$, $\Theta_T(\partial B_{i,\sigma}^0) =\partial B_{i,\sigma}$. 
The cycle $\partial B_{i,\sigma}^0$ defines a class in $H_n( T_1^0 )= H_n( S_{\sigma(\ell_i)}^{\leq3}\times S_{\sigma(\ell_{\plusun{i}})}^{\leq3}  )$. 
This class reads $[\partial B_{i,\sigma}^0] = \sum\limits_{j=1}^8\alpha_{i,j}[s_j]$ for some rational numbers $(\alpha_{i,j})_{j\in\und {8}}$.  
Apply $\Theta_T$ to this identity to get $[\partial B_{i,\sigma}] = \sum\limits_{j=1}^8\alpha_{i,j}[s_j]$. 
This implies that $[Q_i]=0$ in $H_n(S_{\sigma(\ell_i)}^{\leq3}\times S_{\sigma(\ell_{\plusun{i}})}^{\leq3} ) $ 
and proves the existence of $\hat{B}_{i,\sigma}$. 

Since the cycle $Q_i\subset S_{\sigma(\ell_i)}^{\leq3}\times S_{\sigma(\ell_{\plusun{i}})}^{\leq3}$ is independent of the knot, the chain $\hat{B}_{i,\sigma}$ can be chosen independently of the knot.\qedhere

\end{proof}

\begin{lm}\label{sigmas}Let $b_{\dimd}$ denote the $\dimd$-th Betti number of $S^+_{1}$.
It is possible to choose two families of cycles $((\bb_{1,j}^\dimd)^+)_{0\leq \dimd \leq n+1,j\in \und{b_\dimd}}$ and $((\ba_{1,j}^\dimd)^+)_{0\leq \dimd \leq n+1,j\in\und{b_\dimd}}$ in $S_{1}^+$ such that:\begin{itemize}

\item For any $\dimd\in\und n$, and any $j\in \und{ b_\dimd}$, $[(\bb_{1,j}^\dimd)^+] = [\bb_j^\dimd]$ and  $[(\ba_{1,j}^\dimd)^+] = [\ba_j^\dimd]$, where the cycles $(\bb_j^\dimd)_{j,\dimd}$ and $(\ba_j^\dimd)_{j,\dimd}$ are defined in Lemma \ref{contributionGamma1}. 

For $d=n+1$, $(\bb_{1,1}^{n+1})^+ = (\ba^{n+1}_{1,1})^+= S_{1}^+$.
For $d=0$, $(\bb_{1,1}^{0})^+ = (\ba^{0}_{1,1})^+$ is a point.

 In particular, for any $\dimd\in \{0,\dots, n+1\}$,
$([(\bb_{1,j}^\dimd)^+])_{j \in \und{b_\dimd}}$ and $([(\ba_{1,j}^\dimd)^+])_{j\in \und{b_\dimd}}$ are bases of $H_\dimd(S_{1}^+)$, and for any $(j,j')\in\und{b_\dimd}^2$, we have the duality relation $\langle [(\bb_{1,j}^{\dimd})^+] ,
[(\ba_{1,j'}^{n+1-\dimd})^+]\rangle_{S^+_{1}} = \delta_{j,j'}$.
\item For any $\dimd\in\und n$, the cycles $((\bb_{1,j}^\dimd)^+)_{j\in \und{b_\dimd}}$ and $((\ba_{1,j}^\dimd)^+)_{j\in \und{b_\dimd}}$ are contained in $\Sigma^+_{1}\cap E_{\degk+1}\subset S^+_{1}$.
\item The point $(\bb_{1,1}^0)^+= (\ba_{1,1}^0)^+$ is in $ \partial (\Sigma^+_{1} \cap E_2)$.
%\item For any $\dimd>\frac{n+1}2$, and any $j\in \und{ b_\dimd}$, $(\ba_{1,j}^{\dimd})^+ = (\bb_{1,j}^{\dimd})^+$, and for any $\dimd<\frac{n+1}2$ and any $j\in\und{b_\dimd}$, $(\bb_{1,j}^{\dimd})^+ = (-1)^d(\ba_{1,j}^{\dimd})^+$.
\end{itemize}

Since all the Seifert surfaces $(\Sigma_i^\pm)_{i\in \und k}$ are obtained from $\Sigma_1^+$ by pushing $\Sigma_1^+$ in the positive normal direction, these families yield similar families $((\bb_{i, j}^\dimd)^\pm)_{0\leq \dimd \leq n+1, j\in \und{b_\dimd}}$ and $((\ba_{i, j}^\dimd)^\pm)_{0\leq \dimd \leq n+1, j\in \und{b_\dimd}}$ in $H_*(S^\pm_{i})$.

\end{lm}\begin{proof}
It is possible to choose two families such that the first three properties hold because the map $H_\dimd(\Sigma^+_{1}\cap E_{k+1})\rightarrow H_\dimd( S^+_{1})$ induced by the inclusion is an isomorphism for $d\in\{0,\ldots, n\}$.\qedhere
%Due to the symmetry (and antisymmetry) properties of the intersection number, we can also choose these chains such that the last property holds.
\end{proof}

\begin{lm}\label{decomp-prop}
For any $i\in\zk$, define the cycle $\overline{B_{i,\sigma}} = B_{i,\sigma}- B^0_{i,\sigma} +\hat B_{i,\sigma}$ of $\overline{Y_i}(\sigma)$. 

Its class in $H_{n+1}(\overline{Y_i}(\sigma))$ reads $\left[\overline{B_{i,\sigma}}\right]=R^B_{i,\sigma}+$
\[\sum\limits_{\dimd\in\und n}\sum\limits_{(p,q,\epsilon,\epsilon')\in(\und{b_\dimd})^2\times\{\pm\}^2} 
(-1)^{n(d+1)} 
\lk \left((\bb^{\dimd}_{\sigma(\ell_i),p})^{\epsilon}, (\ba^{n+1-\dimd}_{\sigma(\ell_{\plusun{i}}),q})^{\epsilon'}\right)
 \left[(\ba^{n+1-\dimd}_{\sigma(\ell_i),p})^{\epsilon} \times(\bb^{\dimd}_{\sigma(\ell_{\plusun{i}}),q})^{\epsilon'}\right],\]
where $R^B_{i,\sigma}$ reads $\sum\limits_{\dimd\in\{0,n+1\}} \sum\limits_{(\epsilon,\epsilon')\in\{\pm\}^2}  \alpha_{\dimd, 1, 1, \epsilon, \epsilon'}^{(i)}
 \left[(\ba^{n+1-\dimd}_{\sigma(\ell_i),1})^{\epsilon} \times(\bb^{\dimd}_{\sigma(\ell_{\plusun{i}}),1})^{\epsilon'}\right]$, with rational coefficients $(\alpha_{\dimd, 1, 1, \epsilon, \epsilon'}^{(i)})_{\dimd\in\{0,n+1\},( \epsilon, \epsilon')\in\{\pm\}^2}$ independent of the knot.
\end{lm}
\begin{proof}
The families of chains $((\bb_{\sigma(\ell_i),p}^\dimd)^+)_{0\leq \dimd \leq n+1, p \in\und{ b_\dimd}}$ and $((\ba_{\sigma(\ell_{\plusun{i}}),p}^\dimd)^+)_{0\leq \dimd \leq n+1, p \in\und{ b_\dimd}}$ induce 
the two following bases of $H_{n+1}( \overline {Y_i}(\sigma))$: 
\[\left(\left[(\ba^{n+1-\dimd}_{\sigma(\ell_i),p})^{\epsilon} \times(\bb^{\dimd}_{\sigma(\ell_{\plusun{i}}),q})^{\epsilon'}\right]\right)_{0\leq \dimd \leq n+1, 1\leq p, q \leq b_{\dimd}, (\epsilon,\epsilon')\in\{\pm\}^2},\] \[\text{and }\left(\left[(\bb^{\dimd}_{\sigma(\ell_i),p})^{\epsilon} \times(\ba^{n+1-\dimd}_{\sigma(\ell_{\plusun{i}}),q})^{\epsilon'}\right]\right)_{0\leq \dimd\leq n+1, 1\leq p, q\leq b_\dimd, (\epsilon,\epsilon')\in\{\pm\}^2}.\] 
These bases are dual is the sense that for any $p,p',q,q',\dimd, \dimd',\epsilon, \epsilon', \eta, \eta'$, 
 \[\left\langle \left[(\ba^{n+1-\dimd}_{\sigma(\ell_i),p})^{\epsilon} \times(\bb^{\dimd}_{\sigma(\ell_{\plusun{i}}),q})^{\epsilon'}\right] , 
\left[(\bb^{\dimd'}_{\sigma(\ell_i),p'})^{\eta} \times(\ba^{n+1-\dimd'}_{\sigma(\ell_{\plusun{i}}),q'})^{\eta'}
\right] \right\rangle_{\overline Y(\sigma)}
 = 
(-1)^{(n+1)\dimd} 
 \delta_{(\dimd, p, q , \epsilon,\epsilon')}^{(\dimd',p',q', \eta, \eta')} ,\] where $\delta_x^y$ is the Kronecker delta.
There exist coefficients such that \[\left[\overline{B_{i,\sigma}}\right]=\sum\limits_{\dimd=0}^{n+1}\sum\limits_{(p,q,\epsilon,\epsilon')\in(\und{b_\dimd})^2\times\{\pm\}^2}  \alpha_{\dimd,p,q,\epsilon, \epsilon'}^{(i)} \left[(\ba^{n+1-\dimd}_{\sigma(\ell_i),p})^{\epsilon} \times(\bb^{\dimd}_{\sigma(\ell_{\plusun{i}}),q})^{\epsilon'}\right] .\]
For any $\dimd\in\und n$, and any $(p,q,\epsilon,\epsilon')\in(\und{b_\dimd})^2\times\{\pm\}^2$, 
\begin{eqnarray*}\alpha_{\dimd,p,q,\epsilon,\epsilon'}^{(i)}&=& 
(-1)^{(n+1)\dimd} 
\left\langle \left[\overline {B_{i,\sigma}}\right] ,\left[(\bb^{\dimd}_{\sigma(\ell_i),p})^{\epsilon} \times(\ba^{n+1-\dimd}_{\sigma(\ell_{\plusun{i}}),q})^{\epsilon'}\right] \right\rangle_{\overline {Y_i}(\sigma) }\\
&=& 
(-1)^{(n+1)\dimd} 
\left\langle B_{i,\sigma} ,\left[(\bb^{\dimd}_{\sigma(\ell_i),p})^{\epsilon} \times(\ba^{n+1-\dimd}_{\sigma(\ell_{\plusun{i}}),q})^{\epsilon'}\right] \right\rangle_{Y_i(\sigma) }\\
&=& (-1)^{(n+1)\dimd} 
\left\langle B_{\sigma(f_i)} ,\left[(\bb^{\dimd}_{\sigma(\ell_i),p})^{\epsilon} \times(\ba^{n+1-\dimd}_{\sigma(\ell_{\plusun{i}}),q})^{\epsilon'}\right] \right\rangle_{\configM} \\
&=& (-1)^{(n+1)\dimd} (-1)^{(n+3)(n+1)}
\left\langle\left[(\bb^{\dimd}_{\sigma(\ell_i),p})^{\epsilon} \times(\ba^{n+1-\dimd}_{\sigma(\ell_{\plusun{i}}),q})^{\epsilon'}\right],B_{\sigma(f_i)} \right\rangle_{\configM} \\
&=& (-1)^{n(d+1)}   
\lk \left((\bb^{\dimd}_{\sigma(\ell_i),p})^{\epsilon}, (\ba^{n+1-\dimd}_{\sigma(\ell_{\plusun{i}}),q})^{\epsilon'}\right),\end{eqnarray*}
where the first equality comes from the duality of the bases above, the second one comes from the second point of Lemma \ref{sigmas}, and the fifth one comes from Lemma \ref{lk}.

Set $R^B_{i,\sigma}=\sum\limits_{\dimd\in\{0,n+1\}} \sum\limits_{(\epsilon,\epsilon')\in\{\pm\}^2}  \alpha_{\dimd, 1, 1, \epsilon, \epsilon'}^{(i)}
 \left[(\ba^{n+1-\dimd}_{\sigma(\ell_i),1})^{\epsilon} \times(\bb^{\dimd}_{\sigma(\ell_{\plusun{i}}),1})^{\epsilon'}\right]$.
The duality allows us to compute the coefficients that appear in $R^B_{i,\sigma}$, too. 
First, $\alpha_{0,1,1, \epsilon,\epsilon'}^{(i)} = 
\langle [\overline{B_{i,\sigma}} ], 
[ (\bb_{\sigma(\ell_{i}),1}^0)^{\epsilon}\times S_{\sigma(\ell_{\plusun{i}})}^{\epsilon'}]
 \rangle_{\overline {Y_i}(\sigma)}$. 
 The chain $(\bb_{\sigma(\ell_{i}),1}^0)^{\epsilon}\times S_{\sigma(\ell_{\plusun{i}})}^{\epsilon'}$ 
 is contained in 
 $\partial (\Sigma_{\sigma(\ell_{i})}^{\epsilon}\cap E_2) \times S_{\sigma(\ell_{\plusun{i}})}^{\epsilon'} $,
 so it only meets $B_{i, \sigma}$ 
 inside $\partial(\Sigma^\epsilon_{\sigma(\ell_i)}\cap E_2)\times \Sigma_{\sigma(\ell_{\plusun{i}})}^{\epsilon'} $. 
Let us prove that $B_{i,\sigma}\cap(\partial(\Sigma^\epsilon_{\sigma(\ell_i)}\cap E_2)\times \Sigma_{\sigma(\ell_{\plusun{i}})}^{\epsilon'} )$ lies in $p_b^{-1}(N_3\times N_3)$.
If a configuration in this intersection was in $\partial N_2\times E_3$, 
it would be in $( \mur[2] \times \Sigma_{\sigma(f_i)}(3)) \cup ( L_{\theta_{\sigma(f_i)}}(2)\times E_3)$. 
Since $\Sigma_{\sigma(f_i)}(3) \cap \Sigma_{\sigma(\ell_{\plusun{i}})}(2) = \emptyset$ and $\Sigma_{\sigma(\ell_{i})}(2)\cap L_{\theta_{\sigma(f_i)}}(2)=\emptyset$, this is impossible.
Therefore,  $\langle B_{i,\sigma} , (\bb_{\sigma(\ell_{i}),1}^0)^{\epsilon}\times S_{\sigma(\ell_{\plusun{i}})}^{\epsilon'} \rangle_{S_{\sigma(\ell_i)} \times S_{\sigma(\ell_{\plusun{i}})}}$ only counts intersection points in $p_b^{-1}(N_3\times N_3) $. 
By construction, this implies that this intersection number does not depend on the knot. 
Similarly, $\langle B^0_{i,\sigma} ,(\bb_{\sigma(\ell_{i}),1}^0)^{\epsilon}\times S_{\sigma(\ell_{\plusun{i}})}^{\epsilon'} \rangle_{S_{\sigma(\ell_i)} \times S_{\sigma(\ell_{\plusun{i}})}}$ does not depend on the knot, and $\langle \hat{B}_{i,\sigma} ,(\bb_{\sigma(\ell_{i}),1}^0)^{\epsilon}\times S_{\sigma(\ell_{\plusun{i}})}^{\epsilon'} \rangle_{S_{\sigma(\ell_i)} \times S_{\sigma(\ell_{\plusun{i}})}}$ does not depend on the knot because of Lemma \ref{hpi}. This proves that $\alpha_{0,1,1, \epsilon,\epsilon'}^{(i)}$ does not depend on the knot.
The same argument proves that $\alpha_{n+1,1,1,\epsilon,\epsilon'}^{(i)}$ does not depend on the knot, so that the coefficients of $R^B_{i,\sigma}$ do not depend on the knot.
\end{proof}
Lemma \ref{decomp-prop} yields the following lemma.

\begin{lm}\label{decomppii}
%% scénario 1 {
Set $\overline\eta_i= 1$ for $i\in\und{k-1}$, and $\overline\eta_k= (-1)^{(k+1)(n+1)}$.
%%%%}
Let $J$ denote the set of tuples $(\dimd,p,q,\hat\epsilon)$ such that $\dimd\in\und n$, $(p,q)\in (\und{b_{\dimd}})^2$, and $\hat\epsilon$ is a map $\hat\epsilon\colon\zk\rightarrow \{\pm\}$. For any $i\in\zk$, set $\overline{\Pii} = \Pii- \Pii^0 + \overline{\pi_i}^{-1}(\hat B_{i,\sigma})$.
With these notations, \begin{multline*}
[\overline{\Pii}] = R_{i,\sigma}
+
 \sum\limits_{(\dimd,p,q,\hat\epsilon)\in J}
 (-1)^{n(d+1)}
\lk\left((\bb_{\sigma(\ell_i), p}^{\dimd})^{\hat\epsilon(i)}, (\ba_{\sigma(\ell_{\plusun{i}}),q}^{n+1-\dimd})^{\hat\epsilon(\plusun{i})}\right) \\
\left[
%%%% scénario 1{
\overline\eta_i.
%%%%%}
(\ba_{\sigma(\ell_i), p}^{n+1-\dimd})^{\hat\epsilon(i)} \times(\bb_{\sigma(\ell_{\plusun{i}}),q}^{\dimd})^{\hat\epsilon(\plusun{i})}\times \prod\limits_{j\not\in\{i,\plusun{i}\}} S_{\sigma(\ell_j)}^{\hat\epsilon(j)}\right]
,\end{multline*}
where $R_{i,\sigma}$ reads \[R_{i,\sigma}=\sum\limits_{\dimd\in\{0,n+1\}}\sum\limits_{\hat\epsilon\colon\und k \rightarrow \{\pm\}}\alpha_{\dimd, \hat\epsilon}^{(i)}
\left[ 
%%%% scénario 1{
\overline\eta_i.
%%%%%}
(\ba_{\sigma(\ell_i), 1}^{n+1-\dimd})^{\hat\epsilon(i)} \times(\bb_{\sigma(\ell_{\plusun{i}}),1}^{\dimd})^{\hat\epsilon(\plusun{i})}\times\prod\limits_{j\not\in\{i,\plusun{i}\}} S_{\sigma(\ell_j)}^{\hat\epsilon(j)}\right],\]
with coefficients $(\alpha_{\dimd, \hat\epsilon}^{(i)})_{\dimd,\hat\epsilon}$ that do not depend on the knot.
\end{lm}

\begin{proof}
%We have $\overline{\Pii} = \overline{\pi_i}^{-1}(\overline{B_{i,\sigma}})$, so Lemma \ref{decomp-prop} implies \begin{multline*}\overline{\Pii} =R_{i,\sigma}\\
%+ \sum\limits_{\dimd\in\und n}\sum\limits_{(p,q,\epsilon,\epsilon')\in\und{b_{\dimd}}^2\times\{\pm\}^2}
%(-1)^{n(d+1)}\lk\left((\bb_{\sigma(\ell_i), p}^{\dimd})^{\epsilon},
% (\ba_{\sigma(\ell_{\plusun{i}}),q}^{n+1-\dimd})^{\epsilon'}\right) 
%\left[{\overline{\pi}_i}^{-1}\left((\ba_{\sigma(\ell_i), p}^{n+1-\dimd})^{\epsilon} \times(\bb_{\sigma(\ell_{\plusun{i}}),q}^{\dimd})^{\epsilon'}\right)\right],\end{multline*}
%with $R_{i,\sigma}$ as in the lemma with $\alpha^{(i)}_{d,\hat\epsilon}=\alpha^{(i)}_{d, 1, 1, \hat\epsilon(i), \hat\epsilon(\plusun{i})}$ for $d\in\{0,n+1\}$.
%For any $\dimd\in\und n$, any $p$ and $q$ in $\und{b_{\dimd}}$, 
%and any $(\epsilon, \epsilon')\in\{\pm\}^2$, 
%\begin{multline*}{\overline{\pi}_i}^{-1}\left((\ba_{\sigma(\ell_i), p}^{n+1-d})^{\epsilon} \times(\bb_{\sigma(\ell_{\plusun{i}}),q}^{d})^{\epsilon'}\right)\\
%               = \sum\limits_{\hat\epsilon\colon\zk\setminus\{i,\plusun{i}\}\rightarrow\{\pm\}^2} 
%                      \eta_{\hat\epsilon}\left((\ba_{\sigma(\ell_i), p}^{n+1-d})^{\epsilon} \times(\bb_{\sigma(\ell_{\plusun{i}}),q}^{d})^{\epsilon'}\times\prod\limits_{j\not\in\{i,\plusun{i}\}} S_{\sigma(\ell_j)}^{\hat\epsilon(j)}\right)\end{multline*} 
%                      for some signs $(\eta_{\hat\epsilon})_{\hat\epsilon}$. 
%and any $\hat\epsilon\colon\zk\rightarrow\{\pm\}^2$ 
If $i<k$, then
the chain $(\ba_{\sigma(\ell_i), p}^{n+1-d})^{\hat\epsilon(i)} \times(\bb_{\sigma(\ell_{\plusun{i}}),q}^{d})^{\hat\epsilon(\plusun{i})}
\times \prod\limits_{j\not\in\{i,\plusun{i}\}} S_{\sigma(\ell_j)}^{\hat\epsilon(j)}$ is cooriented by $(-1)^{n+1-\dimd}\N(\ba_{\sigma(\ell_i), p_i}^{n+1-\dimd})^{\hat\epsilon(i)} \times\N(\bb_{\sigma(\ell_{\plusun{i}}),q_i}^{\dimd})^{\hat\epsilon(\plusun{i})}$.
If $i=k$, then the coorientation of
$(\ba_{\sigma(\ell_k), p}^{n+1-d})^{\hat\epsilon(k)} \times(\bb_{\sigma(\ell_{1}),q}^{d})^{\hat\epsilon(1)}
\times \prod\limits_{j=2}^{k-1} S_{\sigma(\ell_j)}^{\hat\epsilon(j)}$
 is  $(-1)^{(k+1)(n+1)}$ $ (-1)^{n+1-\dimd}
 \N(\ba_{\sigma(\ell_k), p_k}^{n+1-\dimd})^{\hat\epsilon(k)} \times\N(\bb_{\sigma(\ell_{1}),q_k}^{\dimd})^{\hat\epsilon(1)}$ because 
$S_{\sigma(\ell_k)}^{\hat\epsilon(k)}\times S_{\sigma(\ell_1)}^{\hat\epsilon(1)} \times \prod\limits_{j=2}^{k-1} S_{\sigma(\ell_j)}^{\hat\epsilon(j)}$ is oriented as $(-1)^{(k+1)(n+1)} \prod\limits_{j\in\und k} S_{\sigma(\ell_j)}^{\hat\epsilon(j)}$.
\qedhere
%, 
%the signs $(\eta_{\hat\epsilon})_{\hat\epsilon}$ are all $+1$. The lemma follows.\qedhere

\end{proof}

\begin{lm}\label{interpii}
Let $J'$ denote the set of tuples $(\dimd, p, q, \hat\epsilon)$ such that $\dimd\in\{0,\ldots, n+1\}$, $(p,q)\in (\und{b_\dimd})^2$, and $\hat\epsilon\colon\zk\rightarrow\{\pm\}$. (We have $J\subset J'$.)
For $(\dimd_i, p_i, q_i, \hat\epsilon_i)_{i\in\zk}\in(J')^{\zk}$,
\begin{multline*}\left\langle\left(
\left[
%%%% scénario 1{
\overline\eta_i.
%%%%%}
(\ba_{\sigma(\ell_i), p_i}^{n+1-\dimd_i})^{\hat\epsilon_i(i)} \times(\bb_{\sigma(\ell_{\plusun{i}}),q_i}^{\dimd_i})^{\hat\epsilon_i(\plusun{i})}\times \prod\limits_{j\not\in\{i,\plusun{i}\}} S_{\sigma(\ell_j)}^{\hat\epsilon_i(j)}
\right]\right)_{i \in \zk}
\right\rangle_{\overline Y(\sigma)}
\\= \begin{cases} (-1)^{(n+1-d_1)(kn+1)}
&\text{if \(\dimd_1=\ldots = \dimd_k\), \(\hat\epsilon_1=\ldots=\hat\epsilon_k\), and for any \(i\in\zk\), \(q_{i} = p_{\plusun{i}}\),}\\
0
&\text{otherwise.} 
\end{cases}\end{multline*}
\end{lm}

\begin{proof}If we do not have $\hat\epsilon_1=\ldots=\hat\epsilon_k$ the intersection is empty. If we do not have $\dimd_1=\ldots = \dimd_k$, there exists an integer $i\in\zk$ such that $\dimd_{i}<\dimd_{\plusun{i}}$ and the chains $(\bb_{\sigma(\ell_{\plusun{i}}), q_i}^{\dimd_i})^{\hat\epsilon_i(i)}$ and $(\ba_{\sigma(\ell_{\plusun{i}}), p_{\plusun{i}}}^{n+1-\dimd_{\plusun{i}}})^{\hat\epsilon_{\plusun{i}}(i)}$ do not intersect, up to small perturbations, so the intersection number of the lemma is zero. Let us now assume $\hat\epsilon_1=\ldots=\hat\epsilon_k=\hat\epsilon$ and $\dimd_1=\ldots = \dimd_k=\dimd$.
The chain $
%%%% scénario 1{
\overline\eta_i
%%%%%}
(\ba_{\sigma(\ell_i), p_i}^{n+1-\dimd})^{\hat\epsilon(i)} \times(\bb_{\sigma(\ell_{\plusun{i}}),q_i}^{\dimd})^{\hat\epsilon(\plusun{i})}\times \prod\limits_{j\not\in\{i,\plusun{i}\}} S_{\sigma(\ell_j)}^{\hat\epsilon(j)}$ 
is cooriented by $(-1)^{n+1-\dimd}\N(\ba_{\sigma(\ell_i), p_i}^{n+1-\dimd})^{\hat\epsilon(i)} \times\N(\bb_{\sigma(\ell_{\plusun{i}}),q_i}^{\dimd})^{\hat\epsilon(\plusun{i})}$. 
If the intersection
of the lemma is non-empty,
it consists on the $(x_1, \ldots, x_k)$ such that $x_i \in 
(\bb_{\sigma(\ell_{i}),q_{\moinsun{i}}}^{\dimd})^{\hat\epsilon(i)} \cap
(\ba_{\sigma(\ell_{i}), p_{i}}^{n+1-\dimd})^{\hat\epsilon(i)}$ for any $i\in\und k$.
The normal bundle at $(x_1, \ldots, x_k)$ of the intersection of the lemma is oriented as \begin{align*} &\prod\limits_{i\in\zk}\left((-1)^{n+1-\dimd} \N(\ba_{\sigma(\ell_i), p_i}^{n+1-\dimd})^{\hat\epsilon(i)} \times\N(\bb_{\sigma(\ell_{\plusun{i}}),q_i}^{\dimd})^{\hat\epsilon(\plusun{i})} \right)\\
=&(-1)^{\degk (n+1-\dimd)}
\N(\ba_{\sigma(\ell_1), p_1}^{n+1-\dimd})^{\hat\epsilon(1)}\times
\left(\prod\limits_{i=1}^{\degk-1}\left( \N(\bb_{\sigma(\ell_{\plusun{i}}),q_{i}}^{\dimd})^{\hat\epsilon(\plusun{i})}\times
\N(\ba_{\sigma(\ell_{\plusun{i}}), p_{\plusun{i}}}^{n+1-\dimd})^{\hat\epsilon(\plusun{i})} \right) \right) \times
\N(\bb_{\sigma(\ell_1), q_k}^{\dimd})^{\hat\epsilon(1)}
\\
=& (-1)^{\degk(n+1- d)}(-1)^{(n+1-d)( (k-1)(n+1) + d)} \N(\bb_{\sigma(\ell_1), q_k}^{\dimd})^{\hat\epsilon(1)}\times \N(\ba_{\sigma(\ell_1), p_1}^{n+1-\dimd})^{\hat\epsilon(1)}\\
& \ \ \ \ \ \ \ \ \ \ \ \ \ \ \ \ \ \ \ \ \ \ \ \ \ \ \ \ \ \ \ \ \ \ \ \ \ \ \ \ \ \ \ \ \  \times
\prod\limits_{i=1}^{\degk-1}\left(
\N(\bb_{\sigma(\ell_{\plusun{i}}),q_{i}}^{\dimd})^{\hat\epsilon(\plusun{i})}\times\N(\ba_{\sigma(\ell_{\plusun{i}}), p_{\plusun{i}}}^{n+1-\dimd})^{\hat\epsilon(\plusun{i})} \right)\\
=& (-1)^{(n+1-d)(kn+1)} 
\prod\limits_{i\in\zk}\left(\N(\bb_{\sigma(\ell_{i}),q_{\moinsun{i}}}^{\dimd})^{\hat\epsilon(i)}
 \times\N(\ba_{\sigma(\ell_{i}), p_{i}}^{n+1-\dimd})^{\hat\epsilon(i)}
\right),
\end{align*}
where $\moinsun{i}$ is the element of $\und k$ such that $\plusun{(\moinsun{i})}=i$.
This proves that the intersection number of the lemma is $(-1)^{(n+1-d)(kn+1)}\prod\limits_{i\in\zk}
\left\langle (\bb_{\sigma(\ell_{i}),q_{\moinsun{i}}}^{\dimd})^{\hat\epsilon(i)},
(\ba_{\sigma(\ell_{i}), p_{i}}^{n+1-\dimd})^{\hat\epsilon(i)}\right\rangle_{S_{\sigma(\ell_i)}^{\hat\epsilon(i)}}$.
\qedhere 
\end{proof}

\begin{lm}\label{formule1}
For any numbering $\sigma$ of $\Gamma_k$, \[\Delta_{\Gamma_k, \sigma}Z_k = \frac{1}{2^k}
\sum\limits_{d\in\und n} \sum\limits_{\hat\epsilon \colon \zk \rightarrow \{\pm\}} \sum\limits_{p\colon \zk \rightarrow \und{b_d}}
(-1)^{d+k+n}
\prod\limits_{i\in\zk} \lk \left((\bb_{\sigma(\ell_i), p(i)}^{d})^{\hat\epsilon(i)}, (\ba_{\sigma(\ell_{\plusun{i}}), p(\plusun{i})}^{n+1-d})^{\hat\epsilon(\plusun{i})}\right)
.\]
\end{lm}

\begin{proof}
Lemmas \ref{decomppii} and \ref{interpii} imply that \begin{align*}\langle& \overline{\Pii[1]}, \cdots, \overline{\Pii[k]}\rangle_{\overline Y(\sigma)}\\&=
\sum\limits_{\dimd\in \und n} \sum\limits_{\hat\epsilon \colon \und k \rightarrow \{\pm\}} \sum\limits_{p\colon \und k \rightarrow \und{b_{\dimd}}}
(-1)^{(n+1-d)(kn+1) +kn(d+1)} \prod\limits_{i\in\zk}\lk \left((\bb_{\sigma(\ell_i), p(i)}^{\dimd})^{\hat\epsilon(i)}, (\ba_{\sigma(\ell_{\plusun{i}}), p(\plusun{i})}^{n+1-\dimd})^{\hat\epsilon(\plusun{i})}\right) + \rho_1,\end{align*} where $\rho_1 = \langle R_{1,\sigma} , \ldots, R_{k, \sigma}\rangle_{\overline Y(\sigma)}$ does not depend on the knot. 
For any $i\in \und k$ and $j\in \{0,1,2\}$, set \[\Pii[i](j) = \begin{cases} \Pii & \text{if $j=0$,}\\ \Pii^0 & \text{if $j=1$,}\\ \widehat{\Pii}={\overline{\pi_i}}^{-1}(\hat B_{i,\sigma}) & \text{if $j=2$,} \end{cases}\] so that $\overline\Pii = \Pii(0)-\Pii(1)+\Pii(2)$. 
By transversality, $\bigcap\limits_{i\in\zk}\mathrm{Supp}(\overline\Pii) \subset \bigcap\limits_{i\in\zk}\big(\Int(\Pii(0))\sqcup\Int(\Pii(1))\sqcup\Int(\Pii(2))\big)$, 
so that \[\left\langle \overline{\Pii[1]}, \cdots, \overline{\Pii[k]}\right\rangle_{\overline Y(\sigma)}= 
\sum\limits_{j\colon\und k \rightarrow \{0,1,2\}} (-1)^{j(1)+\ldots +j(k)} \left\langle (\Pii(j(i)))_{i\in\zk} \right\rangle_{\overline Y(\sigma)}
.\]

For $r\in [3, R-1]$, let $S_i^{\leq r}$ denote the set of points of $S_i$ that come from a point of $N_r$ or $N_r^0$ in the gluing that defines $S_i$ in Lemma \ref{hpi}. Let us prove that for a configuration $c\in \bigcap\limits_{i\in\zk}\Pii(j(i))$, if $c(w_p)\in S_{\sigma(\ell_p)}^{\leq r}$ for some $r\in [3, R-1]$ and $p\in\und k$, then $ c(w_{\plusun{p}})\in S_{\sigma(\ell_{\plusun{p}})}^{\leq r+1}$. 
\begin{itemize}
\item If $j(p)=0$, then $(c(w_p), c(w_{\plusun{p}}))\in B_{\sigma(f_p)}$. 
If $c(w_{\plusun{p}})$ were in $E_{r+1}$, 
then we would have $(c(w_p), c(w_{\plusun{p}}))\in {p_b}^{-1}(N_r\times E_{r+1})$, so $c(w_p)\in L_{\theta_{\sigma(f_p)}}(r)$ or $c(w_{\plusun{p}})\in \Sigma_{\sigma(f_p)}(r+1)$. 
Since the Seifert surfaces $(\Sigma_j^\epsilon)_{j\in \und k, \epsilon = \pm}$ are pairwise parallel, 
this is impossible and $c(w_{\plusun{p}})\in \Sigma_{\sigma(\ell_{\plusun{p}})}(2)\cap N_{r+1}\subset S_{\sigma(\ell_{\plusun{p}})}^{\leq r+1}$.
\item If $j(p)=1$, we similarly prove that $c(w_{\plusun{p}})\in \Sigma^0_{\sigma(\ell_{\plusun{p}})}(2)\cap N^0_{r+1}\subset S_{\sigma(\ell_{\plusun{p}})}^{\leq r+1}$.
\item If $j(p)=2$, since $\hat B_{p,\sigma}= \overline{\pi_i}(\widehat{P_{p,\sigma}} ) \subset S_{\sigma(\ell_{p})}^{\leq 3}\times S_{\sigma(\ell_{\plusun{p}})}^{\leq 3} $ and $r\geq 2$, 
we have $c(w_{\plusun{p}}) \in S_{\sigma(\ell_{\plusun{p}})}^{\leq 3}\subset S_{\sigma(\ell_{\plusun{p}})}^{\leq r+1}$
\end{itemize}
Since $R\geq k+1$, a finite induction proves that if $j$ takes the value $2$, 
then the intersection number $\big\langle (\Pii(j(i)))_{i\in\zk}\big\rangle_{\overline Y(\sigma)}$ 
only counts configurations in $\prod\limits_{i\in\zk} S_{\sigma(\ell_i)}^{\leq R}$, 
where these chains are independent of the knot. This implies that 
\[\langle \overline{\Pii[1]}, \cdots, \overline{\Pii[ k]}\rangle_{Y(\sigma)}= 
\langle \Pii[1], \cdots, \Pii[k]\rangle_{Y(\sigma)}+ (-1)^k \langle \Pii[1]^0, \cdots, \Pii[ k]^0\rangle_{Y(\sigma)} + \rho_2,
\] where $\rho_2$ is independent of the knot. 
The quantity $\rho_3=-\rho_2+ ((-1)^{\degk+1}-1) \langle \Pii[1]^0, \cdots, \Pii[k]^0\rangle_{Y(\sigma)}$ 
does not depend on the knot. Lemma \ref{exprY} reads
\begin{eqnarray*}
\Delta_{\Gamma_k, \sigma}Z_k &=& \frac{(-1)^{k(n+1)+1}}{2^k}\left( \langle \Pii[1], \cdots, \Pii[k]\rangle_{Y(\sigma)} - \langle \Pii[1]^0, \ldots , \Pii[ k]^0\rangle_{Y^0(\sigma)}\right)\\
&=& \frac{(-1)^{k(n+1)+1}}{2^k}\left( \langle \overline{\Pii[ 1]}, \cdots, \overline{\Pii[k]}\rangle_{Y(\sigma)} + \rho_3 \right)\\
&=& \frac{1}{2^k}
\sum\limits_{\dimd\in\und n} \sum\limits_{\hat\epsilon \colon \zk \rightarrow \{\pm\}} \sum\limits_{p\colon \zk \rightarrow \und{b_\dimd}}
(-1)^{d+k+n}
 \prod\limits_{i\in\zk} \lk \left((\bb_{\sigma(\ell_i), p(i)}^{\dimd})^{\hat\epsilon(i)}, (\ba_{\sigma(\ell_{\plusun{i}}), p(\plusun{i})}^{n+1-\dimd})^{\hat\epsilon(\plusun{i})}\right) \\&&
 +(-1)^{k(n+1)+1}\frac{\rho_3+\rho_1}{2^k}.
\end{eqnarray*}
If $\psi$ is the trivial knot, this formula yields $\rho_1+\rho_3=0$. This concludes the proof of Lemma \ref{formule1}.
\end{proof}
\begin{proof}[Proof of Lemma \ref{contributionGamma1}]
Note that  
for any $(\sigma,i)\in\mathrm{Num}(\Gamma_\degk)\times\zk$, 
if $\Sigma^+$ denote the surface obtained from $\Sigma$ by pushing $\Sigma$ in the positive normal direction, 
then $(\Sigma_{\sigma(\ell_i)}^{\hat\epsilon(i)}, \Sigma_{\sigma(\ell_{\plusun{i}})}^{\hat\epsilon(\plusun{i})})$ 
is isotopic to $ (\Sigma, \Sigma^+)$ if $\epsilon_{\hat\epsilon, \sigma(i)}=+1$ 
and to $ (\Sigma^+, \Sigma)$ if $\epsilon_{\hat\epsilon, \sigma(i)}=-1$.
Therefore, Lemma \ref{formule1} and the definition of the Seifert surfaces in Setting \ref{def-pert} 
imply Lemma \ref{contributionGamma1}.
\end{proof}

\subsection{Proof of Theorem \ref{th0}}

Lemma \ref{contributionGamma1} can be rephrased as follows in terms of Seifert matrices. 

\begin{lm}\label{defpoids}
Fix a pair $(\Ba,\Bb)$ of dual bases of $\overline{H_*}(\Sigma)$, and set $\Ba=([\bb_i^\dimd])_{0\leq \dimd \leq n, i\in\und{b_\dimd}}$ and $\Bb=([\ba_i^\dimd])_{0\leq \dimd \leq n, i\in\und{ b_\dimd}}$. 
For any $\dimd\in\und n$, define the matrices $V_\dimd^\pm = V_\dimd^\pm(\Ba,\Bb)$ as in Definition \ref{bases duales}.
For any numbering $\sigma$ of $\Gamma_\degk$ and any map $\hat\epsilon\colon \zk \rightarrow \{\pm\}$, let $\epsilon_{\hat\epsilon, \sigma}$ be defined as in Lemma \ref{contributionGamma1}, 
and let $N(\hat\epsilon, \sigma)$ be the number of integers $i\in\zk$ 
such that $\epsilon_{\hat\epsilon, \sigma}(i)=+1$.
For any $\nu\in \{0, \ldots, \degk\}$, set \[\poids_{\degk,\nu}^{(0)}=\frac{1}{2^\degk(2\degk)!}
\Card\left(\{(\hat\epsilon,\sigma)\in \{\pm\}^{\zk}\times \mathrm{Num}(\Gamma_\degk)\mid N(\hat\epsilon, \sigma) = \nu\}\right).\]
 With these notations, \[\Z[](\psi) = \frac1k\sum\limits_{\nu=1}^{\degk-1}\sum\limits_{\dimd=1}^n
(-1)^{d+k+n}
\poids^{(0)}_{\degk,\nu}\Tr((V_\dimd^+)^{\nu}(V_\dimd^-)^{\degk-\nu}).\]

\end{lm}\begin{proof}
Note that for any $k\geq2$, $\poids_{\degk,0}^{(0)}=\poids_{\degk,\degk}^{(0)}=0$.\qedhere
\end{proof}
In order to prove Theorem \ref{th0}, it suffices to prove the following lemma and to use 
Lemma \ref{symkn}, which implies that the above formula vanishes when $k-n$ is even, and allows us 
to replace $d+k+n$ with $d+1$ in the signs of Lemma \ref{defpoids}.

\begin{lm}\label{lambda10}
For any $\degk\geq2$ and any $\nu\in\und{\degk}$, $\poids^{(0)}_{\degk, \nu}= \poids_{\degk, \nu}$ with the notations of Theorem \ref{th0}.
\end{lm}
\begin{proof}
For any $(\hat\epsilon,\sigma) \in \{\pm\}^{\zk} \times \Num(\Gamma_\degk)$, define \[F_0(\hat\epsilon, \sigma)\colon i \in \zk \mapsto \sigma(\ell_i)+(1-\epsilon(i))k\in \und{4\degk},\] and let $F_1(\hat\epsilon,\sigma)\in\Sym_{\zk}$ be the permutation such that for any $i\in\zk$, \[F_1(\hat\epsilon,\sigma)(i) = 1+ \Card\left(\{j\in\zk \mid F_0(\hat\epsilon, \sigma)(j)<F_0(\hat\epsilon, \sigma)(i)\}\right) .\] By definition, $N(\hat\epsilon, \sigma)$ is the number of elements $i\in\zk$ such that $F_0(\hat\epsilon,\sigma)(i)<F_0(\hat\epsilon,\sigma)(\plusun{i})$. Since for any $i\in\zk$, \[\left(F_0(\hat\epsilon,\sigma)(i)<F_0(\hat\epsilon,\sigma)(\plusun{i})\right) \Leftrightarrow \left(F_1(\hat\epsilon,\sigma)(i)<F_1(\hat\epsilon,\sigma)(\plusun{i})\right),\] $N(\hat\epsilon, \sigma)= N_1(F_1(\hat\epsilon,\sigma))$, where for any $\rho\in \Sym_{\zk}$, \[N_1(\rho)=\Card(\{i\in\zk\mid \rho(i)<\rho(\plusun{i})\}).\] Let $\rho\in\Sym_{\zk}$ act on $(\hat\epsilon, \sigma)\in\{\pm\}^{\zk} \times \Num(\Gamma_\degk)$ as \[\rho\cdot (\hat\epsilon,\sigma) = (\hat\epsilon\circ \rho^{-1}, \sigma_\rho),\text{ where for any $i\in\zk$,}\begin{cases}
\sigma_\rho(f_i) = \sigma(f_i),\\
\sigma_\rho(\ell_i)= \sigma(\ell_{\rho^{-1}(i)}).\end{cases}\]
With these notations, $F_1(\rho\cdot (\hat\epsilon,\sigma)) = F_1(\hat\epsilon,\sigma)\circ \rho^{-1}$. This implies that all the fibers of $F_1$ have same cardinality $\frac{2^k(2k)!}{k!}$, so that \[\poids_{\degk,\nu}^{(0)} = \frac{1}{\degk!}\Card\left(\{\sigma\in\Sym_{\zk} \mid N_1(\sigma) =\nu\}\right).\]

Let $\sigma_+\in\Sym_{\zk}$ denote the permutation such that $\sigma_+(i)=\plusun{i}$ for any $i\in\zk$.
The subgroup $G$ generated by $\sigma_+$ in $\Sym_{\zk}$ is cyclic of order $\degk$. Let $G$ act on $\Sym_{\zk}$ in such a way that $\sigma_+\cdot \sigma = \sigma \circ (\sigma_+)^{-1}$ for any $\sigma\in\Sym_{\zk}$. We have $N_1(\sigma)=N_1(\sigma_+\cdot \sigma)$ for any $\sigma\in\Sym_{\zk}$, and each orbit is of cardinality $\degk$. The subgroup $G'=\{\sigma\in\Sym_{\zk}\mid \sigma(k)=k\}$ contains exactly one element of each orbit. For any $\sigma\in G'$, $N_1(\sigma) = 1+  N_2(\sigma_{|\und{\degk-1}})$, where for any $\sigma'\in\Sym_{\und{\degk-1}}$, $N_2(\sigma') = \Card(\{ i \in \und{\degk-2}\mid \sigma'(i)<\sigma'(i+1)\})$
 Therefore, \[\poids_{\degk,\nu}^{(0)} = \frac1{(\degk-1)!}\Card\left(\{\sigma' \in\Sym_{\und{\degk-1}} \mid N_2(\sigma') =\nu-1\}\right)=\poids_{\degk,\nu}.\qedhere\]
\end{proof}
\section{Proof of Theorem \ref{Reidth}}\label{Section6}
\subsection{A generating series for the \texorpdfstring{$(\poids_{\degk,\nu})_{k\geq2, \nu\in\und{\degk-1}}$}{lambda k , nu}}

In this section, we prove the following result for the coefficients $(\poids_{\degk,\nu})_{k\geq2, \nu\in\und{\degk-1}}$ of Theorem \ref{th0}.

\begin{lm}\label{lm5.1}
For any $\degk\geq2$, set $L_\degk(X) = \sum\limits_{\nu=1}^{\degk-1} \poids_{\degk,\nu} X^\nu$, and set $L_1(X) = \frac{X+1}2$. 
The formal power series $L(X, Y) = \sum\limits_{\degk\geq1} L_\degk(X) Y^{\degk-1}\in\QQ[[X,Y]]$ reads 
\[ L(X,Y) = \frac{1-X}2\frac{1+X\exp((1-X)Y)}{1-X\exp((1-X)Y)}.\] 
\end{lm}
In order to prove Lemma \ref{lm5.1}, we first obtain an induction formula for the coefficients $(\poids_{\degk,\nu})_{k\geq2, \nu\in\und{\degk-1}}$ in Lemma \ref{lm5.2}. We next derive an induction formula for the polynomials $(L_\degk)_{\degk\geq1}$ in Lemma \ref{lm5.3}, and a differential equation satisfied by $L(X,Y)$ in Lemma \ref{lm5.4}.
\begin{lm}\label{lm5.2}
Extend the definition of the coefficients $(\poids_{\degk,\nu})_{k\geq2, \nu\in\und{\degk-1}}$ to $(\degk, \nu)\in\mathbb Z\times\mathbb Z$ by setting $\poids_{\degk, \nu}= 0$ when $\degk\leq 1$ or $\nu\not\in\{1,\ldots, \degk-1\}$.
For any $\degk\geq 3$, \[(\degk-1)\poids_{\degk,\nu} = \poids_{\degk-1,\nu}+\poids_{\degk-1,\nu-1} +\sum\limits_{r=2}^{\degk-2}\sum\limits_{p\geq0} \poids_{r,p}\poids_{k-r, \nu-p}.\]
\end{lm}
\begin{proof}
By definition, for any $\degk\geq3$ and any $\nu \in \und{\degk-1}$, $\poids_{\degk,\nu} = \frac1{(\degk-1)!} \Card(\{\sigma\in\Sym_{\degk-1} \mid N(\sigma) = \nu-1
\})$, where $N(\sigma) = \Card\{i\in \und{\degk-2} \mid \sigma(i)<\sigma(i+1)\}$.

Let $\sigma\in\Sym_{\degk-1}$, and set $r_\sigma = \sigma^{-1}(\degk-1)$, $I_\sigma = \{1,\ldots, r_\sigma-1\}$, and $J_\sigma=\{r_\sigma+1, \ldots, \degk-1\}$. Let $i_\sigma\colon \sigma(I_\sigma)\rightarrow I_\sigma$ and $j_\sigma\colon \sigma(J_\sigma)\rightarrow J_\sigma$ denote the two only such maps that are strictly increasing bijections. The permutation $\sigma$ induces two permutations $\sigma_1=i_\sigma\circ\sigma_{|I_\sigma}\in\Sym_{I_\sigma}$ and $\sigma_2=j_\sigma\circ\sigma_{|J_\sigma}\in\Sym_{J_\sigma}$. \begin{itemize}
\item If $r_\sigma=\degk-1$, $N(\sigma_1) = N(\sigma)-1$ and $\sigma_2\in\Sym_\emptyset$.
\item If $r_\sigma = 1$, $N(\sigma_2) = N(\sigma)$, and $\sigma_1\in\Sym_\emptyset$.
\item If $r_\sigma\in\{2,\ldots, \degk-2\}$, then $N(\sigma) = 1+ N(\sigma_1)+N(\sigma_2)$, since the elements $i\in\und{k-2}$ such that $\sigma(i)<\sigma(i+1)$ are taken into account in $N(\sigma_1)$ if $i<r_\sigma-1$, in $N(\sigma_2)$ if $i \geq r_\sigma+1$, and since $\sigma(r_\sigma-1) < \sigma(r_\sigma)$ and $\sigma(r_\sigma) > \sigma(r_\sigma+1)$.\end{itemize}
Now, note that $\sigma$ is equivalent to the data of $(r_\sigma, \sigma(I_\sigma), \sigma_1,\sigma_2)$, and that  
there are $\binom{\degk-2}{r_{\sigma}-1}$ possible choices of $\sigma(I_\sigma)$
for a given $r_\sigma$. This yields the induction formula of the lemma.\qedhere
\end{proof}
%This allows us to obtain an induction formula on the polynomials $(L_\degk(X))_{\degk\geq1}$. 

\begin{lm}\label{lm5.3}
The polynomial $L_2(X)$ is $X$, and for any $\degk \geq 3$, 
\[L_\degk(X) = \frac1{\degk-1}\sum\limits_{r=1}^{\degk-1}L_r(X) L_{\degk-r}(X).\]
\end{lm}
\begin{proof}
The first point of the lemma is immediate since $\poids_{2,1} = 1$.
For $\degk\geq3$,
\begin{eqnarray*}
(\degk-1) L_\degk(X) &=& \sum\limits_{\nu\in\und{\degk-1}}(\degk-1) \poids_{\degk,\nu} X^\nu \\
&=& \sum\limits_{\nu\in\und{\degk-1}} \left(\poids_{\degk-1,\nu}+\poids_{\degk-1,\nu-1} 
+\sum\limits_{r=2}^{\degk-2}\sum\limits_{p\geq0} \poids_{r,p}\poids_{k-r, \nu-p}\right) X^\nu \\
&=& L_{\degk-1}(X) + X L_{\degk-1}(X) + \sum\limits_{r=2}^{\degk-2}\sum\limits_{\nu\in\und{\degk-1}}\sum\limits_{p\geq0} \poids_{r,p}X^p\poids_{k-r, \nu-p} X^{\nu-p}\\
&=& (X+1)L_{\degk-1}(X) + \sum\limits_{r=2}^{\degk-2}L_r(X) L_{\degk-r}(X)\\
&=&\sum\limits_{r=1}^{k-1}L_r(X) L_{\degk-r}(X),
\end{eqnarray*}
since $L_1(X)L_{\degk-1}(X) + L_{\degk-1}(X)L_1(X) = (X+1)L_{\degk-1}(X)$.
\end{proof}

%Eventually, this induction formula gives rise to a differential equation on $L(X,Y)$. 

\begin{lm}\label{lm5.4}
$L(X,Y)$ satisfies the differential equation \[\frac{\partial L}{\partial Y}(X,Y) = \left(L(X,Y)\right)^2- \left(\frac{1-X}2\right)^2.\]
\end{lm}

\begin{proof}
Indeed, \begin{eqnarray*}
\frac{\partial L}{\partial Y}(X,Y) &= & \sum\limits_{\degk\geq2} L_\degk(X) (\degk-1)Y^{\degk-2}\\ 
%&=&  L_2(X)+ \sum\limits_{\degk\geq3}\sum\limits_{r=1}^{\degk-1}L_r(X) L_{\degk-r}(X)Y^{\degk-2}\\
&=&  L_2(X)+ \sum\limits_{\degk\geq3}\sum\limits_{r=1}^{\degk-1}L_r(X)Y^{r-1} L_{\degk-r}(X)Y^{\degk-r-1}\\
&=&  L_2(X)+\sum\limits_{\degk\geq1} t_k(X) Y^k,
\end{eqnarray*}
where $(L(X,Y))^2= \sum\limits_{\degk\geq0} t_k(X) Y^k$. Since $L(X,Y) = \frac{(X+1)}2 + \sum\limits_{\degk\geq2}L_\degk(X)Y^{k-1} $, we have $t_0(X) = \frac{(X+1)^2}4$, and 
\begin{eqnarray*}
\frac{\partial L}{\partial Y}(X,Y) &=& X + L(X,Y)^2 - \frac{(X+1)^2}4
= L(X,Y)^2- \left(\frac{1-X}2\right)^2.\qedhere
\end{eqnarray*}
\end{proof}
\begin{proof}[Proof of Lemma \ref{lm5.1}]
Since $|L_k(x)|\leq 1$ for any $x\in [-1,1]$, $L(X,Y)$ defines a power series that converges at least on $]-1,1[^2$.
Fix $x\in \left]0,\frac12\right[$, and set $u_x(t) = L(x, t)$ for any $t\in]-1,1[$. The function $u_x$ satisfies the equation $u_x' = \left(u_x\right)^2- \left(\frac{1-x}2\right)^2$. Set $a= \frac{1-x}2$, and note that \[\int_0^t\frac{u_x'(t)}{(u_x(t))^2-a^2}\d t  =\frac1{2a}\left(\Ln\left(\frac{u_x(t)-a}{u_x(t)+a}\right)-\Ln\left(\frac{u_x(0)-a}{u_x(0)+a}\right)\right),\] so that, for any $t$, \[\Ln\left(\frac{(u_x(t)-a)(u_x(0)+a)}{(u_x(t)+a)(u_x(0)-a)}\right) = 2at.\]
Since $u_x(0) = \frac{x+1}2$ and $a=\frac{1-x}2$, this yields the formula of Lemma \ref{lm5.1}. Both sides of the formula of Lemma \ref{lm5.1} are power series with a convergence domain containing a disk around $(0,0)$, so that the formula also holds for the formal power series.
\end{proof}
\subsection{The formula with the Reidemeister torsion}\label{SectionReid}
%We prove the following lemma, which is equivalent to Theorem \ref{Reidth}.
\begin{lm}\label{lemmefinal}
For any virtually rectifiable null-homologous long knot $\psi$ of an asymptotic $\mathbb K$-homology $\R^{n+2}$, 
\[\sum\limits_{\degk\geq2}Z_\degk(\psi)h^\degk = (-1)^n \Ln( \mathcal T_\psi(e^{h})) .\]

\end{lm}\begin{proof}
Let $\Sigma$ be a Seifert surface for $\psi$, let $(\Ba,\Bb)$ be a pair of dual bases of $\overline H_*(\Sigma)$, and set $V_\dimd^\pm=V_\dimd^\pm(\Ba,\Bb)$. Corollary \ref{Zklemma} yields \begin{eqnarray*}
\sum\limits_{\degk\geq2 }Z_\degk(\psi)h^\degk
&=&
\sum\limits_{d\in\und n}\sum\limits_{\degk\geq2}\sum\limits_{\nu\in\und{\degk-1}}
(-1)^{d+1}
\poids_{\degk,\nu}\frac{h^\degk}{\degk}
\Tr\left((V_\dimd^+)^\nu(V_\dimd^-)^{\degk-\nu}\right)\\
& = & \sum\limits_{\dimd\in\und n} (-1)^{\dimd+1} \Tr\left(
M(hV_\dimd^+, hV_\dimd^-)
\right),\end{eqnarray*}
where $M(X,Y) = \sum\limits_{\degk\geq 2} \sum\limits_{\nu\in\und{\degk-1}} \frac1\degk \poids_{\degk,\nu} X^\nu Y^{\degk-\nu}$. Note that \[ M(X,Y) = \int_0^Y \left( L(XY^{-1},T)-L_1(XY^{-1})\right) \d T.\] 
Lemma \ref{lm5.1} and basic integral calculus yield 
\begin{eqnarray*}
M(X,Y) &=& \int_0^Y\left( \frac{1-XY^{-1}}2\frac{1+XY^{-1}\exp( (1-XY^{-1})T )}{1-XY^{-1}\exp( (1-XY^{-1})T )} - \frac{XY^{-1}+1}2 \right)\d T\\
&=& \int_0^Y\left( \frac{1-XY^{-1}}2 \left( 1 + \frac{2X\exp((1-XY^{-1})T)}{Y-X\exp((1-XY^{-1})T)}\right) - \frac{XY^{-1}+1}2 \right) \d T\\
&=& \int_0^Y\left( - XY^{-1} + \frac{X(1-XY^{-1})\exp((1-XY^{-1})T)}{Y-X\exp((1-XY^{-1})T)} \right)\d T \\
&=& -X - \Ln\left(\frac{Y-X\exp(Y-X)}{Y-X}\right).%\\
%&=& -\Ln\left(\frac{Y\exp(X)-X\exp(Y)}{Y-X}\right).
\end{eqnarray*}
For any square matrices $A,B$ with $B-A=I$ and any $h$ arbitrarily small, this yields 
\[M(hA,hB) =  - hA - \Ln(B-e^{h}A),\]
where $\Ln(C)$ is defined for $C-I$ sufficiently small as $\sum\limits_{k\geq1} \frac{(-1)^{k-1}}k (C-I)^k$.
Therefore, \begin{eqnarray*}
M(hV_d^+, hV_d^-)
&=& 
 - hV_\dimd^+ -\Ln\left(V_\dimd^- - e^{h}V_\dimd^+ \right) \\
&=& 
- \frac{h}2I_{b_d} - hV_\dimd^+ - \Ln\left(e^{-\frac{h}2}V_\dimd^- - e^{\frac{h}2}V_\dimd^+ \right)
\\
&=& 
- \frac{h}2(V_\dimd^+ + V_\dimd^-) - \Ln\left(e^{-\frac{h}2}V_\dimd^- - e^{\frac{h}2}V_\dimd^+ \right)
,\end{eqnarray*}
so that \begin{eqnarray*}\sum\limits_{\degk\geq2}Z_\degk(\psi)h^\degk & = & \sum\limits_{\dimd\in\und n} 
(-1)^{\dimd}\Tr\left(\frac{h}2(V_\dimd^+ + V_\dimd^-) + \Ln\left(e^{-\frac{h}2}V_\dimd^- - e^{\frac{h}2}V_\dimd^+ \right)
\right)\\
&=& \sum\limits_{\dimd\in\und n}(-1)^{\dimd}\left(\Delta_{d,\Sigma}'(1) h +  \Ln( \Delta_{d,\Sigma}(e^{-h}) )\right) ,
\end{eqnarray*}
where the second equality uses
the formula $\Tr(\Ln(I+H)) = \Ln(\det(I+H))$ for $H\in\mathcal M_n(\mathbb C)$ sufficiently small.
Since $\torsion'(1)=0$, and $\torsion(t^{-1}) = (\torsion(t))^{(-1)^{n+1}}$, the lemma follows.\qedhere
\end{proof}
\section{On virtual rectifiability}\label{Section4}
In this section, $n$ is an integer $\geq2$.
\subsection{Proof of Lemma \ref{rect-lemma}}\label{Section4-1}
\begin{lm}\label{relev}

Let $(h_t)_{0\leq t \leq 1}$ be a homotopy of maps $(\R^n, \R^n\setminus \mathbb D^n) \rightarrow (\inj, \iota_0)$ with $h_0(\R^n)=\{\iota_0\}$. 
Let $\mathcal G$ denote the space of smooth maps from $\R^n$ to $GL_{n+2}(\R)$ that map $\R^n\setminus\mathbb D^n$ to $I_{n+2}$.

There exists a continous map $\big(t\in [0,1]\mapsto g_t\in\mathcal G\big)$, such that for any $(t,x)\in [0,1]\times\R^n$, $h_t(x)= g_t(x)\circ h_0(x)$, and such that, for any $x\in \R^n$, $g_0(x) = I_{n+2}$.
\end{lm}

\begin{proof} 

Set $g_0(x) = I_{n+2}$ for any $x$. 
Endow $\R^{n+2}$ with its canonical Euclidean structure and let $P_{t,x}$ denote the orthogonal complement of $h_t(x)(\R^n)$ in $\R^{n+2}$. Let $\pi_{t,x}$ denote the orthogonal projection on $P_{t,x}$.
Set $f(t, s, x) = \min\limits_{z\in P_{s,x}, ||z|| = 1} || \pi_{t,x}(z)|| $. Since $f$ maps the complement of the compact $[0,1]^2\times \mathbb B^n$ to $1$, it is uniformly continuous. Fix $\delta>0$ so that for any $(t,s,x)$ and $(t', s', x')$ with $|t-t'|+|s-s'| + ||x-x'||< \delta$, 
$| f(t', s', x')-f(t, s, x)|<\frac12$, 
and for any $j \in \mathbb N$, set $t_j = \min (j\frac{\delta}2, 1)$. We are going to define $g_t$ on each $[t_j,t_{j+1}]$.

Note that $P_{0, x} \cap P_{t, x}^{\perp}=\{0\}$ if and only if $f(t, 0, x) >0$. Since $f(0,0,x)=1$ for any $x$, we have $P_{0, x} \cap P_{t, x}^{\perp}=\{0\}$ for $t\in [0,t_1]$. 
For $t\in [0,t_1]$, define $g_t(x)$ by the following formula: $$\forall z=(z_1,z_2,\overline z) \in\R^{n+2}=\R\times\R\times \R^n, g_t(x)(z_1,z_2,\overline z) = \pi_{t,x}(\pi_{0, x}(z)) + h_t(x)(\overline z)$$
Since $P_{0, x} \cap P_{t, x}^{\perp}=\{0\}$, $\pi_{t,x}$ defines an isomorphism from $P_{0,x}$ to $P_{t,x}$. Thus, $g_t(x)$ is an isomorphism. 
For $t\in [t_k , t_{k+1}]$ and $x\in\R^n$ define $g_t(x)$ so that $$\forall z=(z_1,z_2,\overline z) \in\R^{n+2}, g_t(x)(z_1,z_2,\overline z) = \pi_{t,x}(\pi_{t_k, x}(\cdots \pi_{t_0, x}(z) \cdots)) + h_t(x)(\overline z).$$
Since $f(t,  t_k, x) \geq f(t_k, t_k, x) -\frac 12= \frac12$, the above method proves that $g_t(x)$ is an isomorphism. This defines a family $(g_t)_{0\leq t \leq 1}$ as required by the lemma. 
\end{proof}

\begin{proof}[Proof of Lemma \ref{rect-lemma}.]
%Let us first assume $n\geq3$.
Let $\tau_e$ be a parallelization such that the class $[\iota(\tau_e,\psi)]$ of Lemma \ref{obstructionlemma}
is zero, so that there exists $(h_t)_{0\leq t \leq 1}$ as in Lemma \ref{relev} 
with $h_1= \iota(\tau_e,\psi)$.
Let $(\tilde g_t)_{0\leq t \leq 1}$ be a smooth approximation of the map $(g_t)_{0\leq t \leq 1}$ of Lemma \ref{relev}, such that for any $x\in\R^n$, $\tilde g_0(x)= I_{n+2}$ and $h_1(x)=\iota(\tau_e,\psi)(x) = \tilde g_1(x)\circ h_0(x)$. Assume without loss of generality that $\left( t\in[0,1]\mapsto \tilde g_t\in \mathcal G \right) $ is constant on a neighborhood of $\{0,1\}$.
Take a tubular neighborhood $N$ of $\psi(\R^n)$ and identify $N$ with $\psi(\R^n)\times \mathbb D^2\subset \psi(\R^n)\times \mathbb C$ with coordinates $(\psi(x), r e^{i\theta})$.
For any $y=(\psi(x), r e^{i\theta})\in N$, set $\tau^0_y =  (\tau_e)_y\circ\tilde g_{1-r}(x)$. This defines a map $\tau^0\colon N\times \R^{n+2}\rightarrow TN$, which extends to a map $\tau\colon \ambientspace\times \R^{n+2}\rightarrow T\ambientspace$, by setting ${\tau}_y = (\tau_e)_y$ when $y\not\in N$.
This construction ensures that $\iota(\tau, \psi)= \iota_0$, and $\tau$ is a parallelization of $\ambientspace$.\qedhere

\end{proof}
\subsection{Proof of Lemma \ref{threct}}
We use the following Bott periodicity theorem, which is proved in \cite{[Bott2]}. 
\begin{theo}\label{Bt}[Bott]
For any $k\geq 0$, and any $N \geq 1$,$$\pi_N(\mathrm{SO}(N+2+k),I_{N+2+k}) = \left\{\begin{array}{lll} 0 & \text{if $N\equiv 2, 4, 5$ or $6 \mod 8$,}\\
\mathbb Z/2\mathbb Z& \text{if $N\equiv 0$ or $1 \mod 8$,}\\
\mathbb Z& \text{if $N\equiv 3$ or $7 \mod 8$.}\end{array}\right.$$
\end{theo}

\subsubsection{Case \texorpdfstring{$n\equiv5 \mod 8$}{n = 5 mod 8}}

\begin{cor}\label{cor-un}
Suppose $n\equiv 5\mod 8$, and let $\ambientspace$ be an asymptotic homology $\R^{n+2}$.
If $\ambientspace$ is parallelizable, then all long knots $\psi\colon \R^n\hookrightarrow \punct M$ are rectifiable.
Therefore, for any long knot $\psi$ in a (possibly non-parallelizable) asymptotic homology $\R^{n+2}$, $\psi\sharp\psi$ is rectifiable. 
 
\end{cor}
\begin{proof}
 As stated in Lemma \ref{lmpin}, $\pi_n(\inj, \iota_0)=\pi_n(SO(n+2), I_{n+2})$. Since $n\equiv 5 \mod 8$, $\pi_n(SO(n+2),I_{n+2}) = 0$.
Then $\pi_n(\inj, \iota_0) = 0$, and, if $\ambientspace$ is parallelizable, 
the rectifiability obstruction $\iota(\psi)$ of Definition \ref{rectifiabilitydef} is trivial.

In the non-parallelizable case, $\ambientspace\sharp\ambientspace$ is parallelizable because of Proposition \ref{conn-sum2}, and the previous argument applies to $\psi\sharp \psi$.
\end{proof}

\subsubsection{Case \texorpdfstring{$n\equiv1 \mod 8$}{n = 1 mod 8} and connected sum of long knots}\label{diese}

\begin{lm}\label{lmfin}
When $n\equiv1\mod 8$, for any long knot $\psi$ in a parallelizable asymptotic homology $\R^{n+2}$, the connected sum $\psi\sharp \psi$ is rectifiable.
Therefore, for any long knot $\psi$ in a (possibly non-parallelizable) asymptotic homology $\R^{n+2}$, the connected sum $\psi\sharp \psi\sharp\psi\sharp\psi $ is rectifiable.
\end{lm}

\begin{proof}If $n=1$, it is true by definition.
Now assume $n\geq2$ and $n\equiv1\mod8$, and let $(\punct M, \tau)$ be a parallelized asymptotic homology $\R^{n+2}$, let $(\punct M\sharp \punct M, \tau\sharp \tau)$ be the induced connected sum, and fix a long knot $\psi\colon \R^n\hookrightarrow \punct M$.
Since $\psi \sharp \psi$ is defined by stacking two copies of the knot, $\iota(\tau\sharp \tau, \psi\sharp \psi)$ is the map defined by stacking two copies of $\iota(\tau, \psi)$. In terms of homotopy classes in $[(\R^n, \R^n\setminus \mathbb D^n), (\inj, \iota_0)]= \pi_n(\inj, \iota_0)$, this yields $[\iota(\tau\sharp \tau, \psi\sharp \psi)] = 2. [\iota(\tau, \psi)] $. Lemma \ref{lmpin} and Theorem \ref{Bt} imply that $\pi_n(\inj, \iota_0)= \mathbb Z/2\mathbb Z$. This yields $[\iota(\tau\sharp \tau, \psi\sharp \psi)]= 0$, so $\psi\sharp\psi$ is rectifiable.

In the non-parallelizable case, $\ambientspace\sharp\ambientspace$ is parallelizable because of Proposition \ref{conn-sum2}, and the previous argument applies to $\psi\sharp\psi$.\qedhere

\end{proof}

Since $\pi_n(SO(n+2), I_{n+2}) = \mathbb Z$ for $n\equiv 3 \mod 4$, the same method implies that $\psi\sharp\psi$ is virtually rectifiable if and only if $\psi\sharp\psi$ is rectifiable (otherwise the class $\iota(\psi\sharp \psi)$ of Definition~\ref{defiota} has infinite order). This argument together with Corollary \ref{cor-un} and Lemma \ref{lmfin} yields the following remark.

\begin{rmq}
Let $\ambientspace$ be an odd-dimensional asymptotic homology $\R^{n+2}$ and let $\psi$ be a long knot of $\ambientspace$. The long knot $\psi$ is virtually rectifiable if and only if $\psi\sharp \psi\sharp\psi \sharp \psi$ is rectifiable.
\end{rmq}

\subsubsection{Even-dimensional case}

Note that we have to keep the parallelizability hypothesis
in the following lemma, since Proposition \ref{conn-sum2} may not extend to the even-dimensional case\footnote{There is an obstruction in $\pi_{n+1}(SO(n+2), I_{n+2})$, which is not a torsion group when $n$ is even.}.
\begin{lm}\label{lmeven}
When $n$ is even, for any long knot $\psi$ in a parallelizable asymptotic homology $\R^{n+2}$, the connected sum $\psi\sharp \psi$ is rectifiable.
Furthermore, if $n$ is an even integer such that $n\not\in 8\mathbb Z$, then $\psi$ itself is rectifiable.
\end{lm}
\begin{proof}
For $n\geq4$, this follows from the same arguments as in the previous subsections, using Lemma \ref{lmpin} and Theorem \ref{Bt}.
Let us assume $n=2$. 
Lemma \ref{lmS} implies that the normal bundle to $\psi(\R^2)$ admits a trivialisation 
$\nu \colon \psi(\R^2)\times\R^2 \rightarrow \N \psi(\R^2)$ such that for any $(x,u) \in \R^2\times \R^2$, 
if $||x||\geq1$, then $\nu(\psi(x), u) = (u, 0, 0) \in \N_{(0,0,x)}\psi(\R^2)$.
For $x\in \R^2$, let $g(x)$ denote the map $\big( u\in \R^{4} \mapsto
\tau_{\psi(x)}^{-1}(\nu(\psi(x), u_1, u_2) + T_x\psi(u_3,u_4) ) \in \R^{4} \big)$.
This yields a homotopy class $[g]$ in the trivial group
$[(\R^2, \R^2\setminus \mathbb D^2), (GL^+(\R^4), I_4)]= \pi_2(SO(4),I_4)$.
Let $(g_t)_{t\in[0,1]}$ be a homotopy between $g$ and the constant map with value $I_4$ among maps such that $g_t(\R^2\setminus\mathbb D^2)=\{I_4\}$.
For any $t\in[0,1]$ and any $x\in\R^2$, set $h_t(x)=g_t(x)\circ \iota_0$.
This yields a homotopy between $\iota(\tau, \psi)$ and the constant map with value $\iota_0$, and thus proves that $\iota(\tau)$ is trivial.
\end{proof}
\section{Construction of admissible propagators}\label{Section 3}
\subsection{Preliminary setting}\label{conventions}

In this section, we prove Lemma \ref{th-prop}. It suffices to prove the following result.
\begin{lm}\label{th-prop0}
Fix a rectifiable null-homologous long knot $\psi\colon \R^n \hookrightarrow \ambientspace$ of an 
asymptotic $\mathbb K$-homology $\R^{n+2}$. 
Let the neighbordhoods $N_R$ and $N_R^0$ and the diffeomorphism $\Theta\colon N_R^0 \rightarrow N_R$ be as in Section \ref{Section31} 
and fix a parallelisation $\tau$ as in Lemma \ref{gauss-rect}. 
Fix two real numbers $\theta 
\in \R$, and $R\geq 3$, and fix isotopic
Seifert surfaces $\Sigma^\pm$ with disjoint interiors such that $\Theta$ maps $({}_\theta\Sigma^\pm)^0\cap N_R^0$ to $\Sigma^\pm\cap N_R $,
with the notations of Definition \ref{adm-prop}. 

Under these assumptions, there exist $R$-admissible propagators for $(\Sigma^+, \Sigma^-, \psi)$ as in 
Definition \ref{adm-prop}.
Furthermore, it is possible to choose $R$-admissible propagators $B$ (for $(\Sigma^+, \Sigma^-, \psi)$) 
and $B_0$ (for $( ({}_\theta\Sigma^+)^0, ({}_\theta\Sigma^-)^0, \psi_0)$) such that $\Theta_2(B_0 \cap 
p_b^{-1}(N_{R}^0\times N_{R}^0) )= B\cap p_b^{-1}(N_{R}\times N_{R})$, where $\Theta_2\colon C_2(N_R^0)
\rightarrow C_2(N_R)$ is the diffeomorphism induced by $\Theta\colon N_R^0\rightarrow N_R$.

\end{lm}
Lemma \ref{lmS} guarantees the existence of Seifert surfaces, neighborhoods and diffeomorphism as in the lemma.
From now on, we assume without loss of generality that $\theta=0$ and $R=3$, and we prove 
Lemma \ref{th-prop0} until the end of Section \ref{Section 3}.

We fix Seifert surfaces $\Sigma^\pm$ as in the lemma. 
We use $\Theta$ to identify a neighborhood $N_3$ of the knot with the neighborhood $N_3^0$ of the trivial knot in $\R^{n+2}$ defined as the union of the cylinder $\{x \in\R^{n+2} \mid {x_1}^2+{x_2}^2\leq 9\}$ and the complement of the open ball of center $0$ and radius $\frac{2.3^2}3=6$. 
Since $\tau$ is chosen as in Lemma \ref{gauss-rect}, $(G_\tau)_{|(\partial \configM) \cap p_b^{-1}(N_3\times N_3)}$ extends to a smooth map $G_0\colon p_b^{-1}(N_3\times N_3)\rightarrow \s^{n+1}$, which is the restriction of the Gauss map of $C_2(\R^{n+2})$ to  ${p_b}^{-1}(N_3\times N_3)={p_b}^{-1}(N_3^0\times N_3^0)\subset C_2(\R^{n+2})$. 

 Define the following subsets: \[X_0 = p_b^{-1}(N_1\times N_1) \ \ \ \ 
X_1 =p_b^{-1}\left( \bigcup\limits_{r\in[1,2]} E_{r+1}\times N_r \right) \ \ \ \ 
X_2 =p_b^{-1}\left( \bigcup\limits_{r\in[1,2]}  N_r\times E_{r+1}\right) , \]
\[ Y_1= p_b^{-1}( (N_2\cap E_1) \times N_1), \ \ \ \ \ \ \ Y_2= p_b^{-1}(N_1\times(N_2\cap E_1)), \] \ \ \ 
\[ X = X_0\cup X_1\cup X_2, \ \ \ \ \ \ \ Y=Y_1\cup Y_2, \ \ \ \ \ \ 
\text{and } W = \overline{\configM \setminus (X\cup Y) }.\]
Figure \ref{Fig2} shows this decomposition of $\configM$, where $X$ is in black, $Y$ in gray, and $W$ in white.
\begin{figure}[H]
\centering
\figureXYZ
\caption{The used decomposition of $\configM$}
\label{Fig2}
\end{figure}

Let $(e_i)_{1\leq i \leq n+2}$ denote the canonical basis of $\R^{n+2}$. The disks $\mur[r]$ and the half-lines $\Dr{\pm}{r}$ are defined in Notation \ref{notations}.
Define the following chains in $X$: \begin{itemize}
\item $B_{X_0} ={G_0}^{-1}(\{e_1\})\cap {p_b}^{-1}(N_1\times N_1)$
\item $B_{X_1}  = \bigcup\limits_{r\in[1,2]}\left((\Sigma^-\cap E_{r+1})\times \mur[r]+ 
%caspair
(-1)^{n+1}
\overline{{p_b}^{-1}(E_{r+1}\times\Dr{+}{r}) } \right)$
\item $B_{X_2} = \bigcup\limits_{r\in[1,2]}\left(\mur[r]\times( \Sigma^+\cap E_{r+1})
%caspair
-
\overline{{p_b}^{-1}(\Dr{-}{r}\times E_{r+1})}\right) $
\end{itemize}
Note that $B_{X_0}$ is naturally oriented. 
We orient $B_{X_1}$ such that the inclusions 
$(\Sigma^-\cap E_{r+1})\times \mur[r]+ (-1)^{n+1}\overline{p_b^{-1}(E_{r+1}\times\Dr{+}{r}) }\hookrightarrow B_{X_1}$ 
preserve the orientation for any $r\in [1,2]$. 
The chain $B_{X_2}$ is similarly oriented. 
We are going to define a chain $B\subset \configM$ such that $\partial B = {G_\tau}^{-1}(\{e_1\})$ and $B\cap \mathrm{int}(X_i)= \mathrm{int}(B_{X_i})$ for $i\in\{0,1,2\}$.

In $N_3= N_3^0\subset \R^{n+2}$, use the coordinates $x= (x_1,x_2, \overline x)\in\R\times\R\times\R^n$. In these coordinates, and for any $r\in [1, 3]$, $N_r= \{ (x_1, x_2,\overline x) \mid {x_1}^2+{x_2}^2\leq r^2 \text{ or } ||x||\geq \frac{18}r\}$. 

We also define the coordinates $x=(x_1, x_2, h_x.\Dir_x)$, such that $h_x\in \R_+$ and $\Dir_x\in \s^{n-1}$. This will help us in drawing the next figures in $\R^2\times \R^+$ with $\Dir_x$ fixed in $\s^{n-1}$. For example, Figure \ref{Fig3} depicts $N_r$.

\begin{figure}[H]
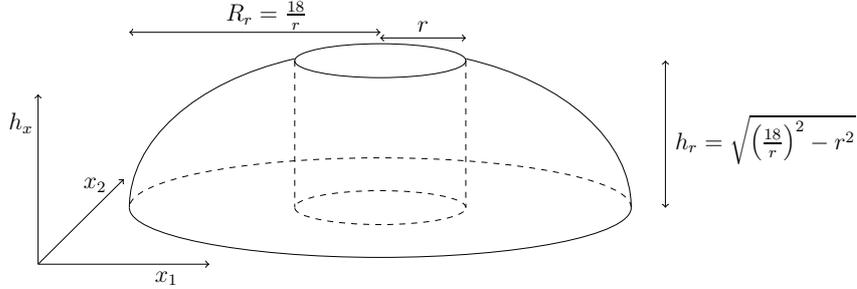

 \centering
 \figurea
\caption{The neighborhood $N_r$.}
\label{Fig3}
\end{figure}
Set $R_r= \frac{18}r$ and $h_r= \sqrt{{R_r}^2-r^2}$,
so that $\partial N_r = \partial_cN_r\cup \partial_sN_r$, with $\partial_cN_r = \{(x_1, x_2, h_x.\Dir_x) \mid {x_1}^2+{x_2}^2=r^2,h_x\leq h_r\}$ and $\partial_sN_r = \{(x_1, x_2, h_x.\Dir_x) \mid {x_1}^2+{x_2}^2= {R_r}^2-{h_x}^2,h_x\leq h_r\}$.

We are going to define a chain $B_{Y_1}\subset Y_1$ such that $\partial(B_{X_0}+B_{Y_1}+B_{X_1}) \subset \partial (X_0\cup Y_1\cup X_1)$. For any $y\in N_3$, define \[D^0(y, -e_1)= \{x\in N_1\mid\text{there exists $t\geq 0$ such that }  x =y-t.e_1\},\] so that $B_{X_0} = \{(x, y) \mid y \in N_1, x \in D^0(y, -e_1)\}$. We orient $D^0(y, -e_1)$ with $\d t$, so that $B_{X_0}$ is oriented by $-\d t \wedge \d y$.

\begin{lm}\label{bordX0}
The boundary $\partial B_{X_0}$ splits into three pieces $G_\tau^{-1}(\{e_1\}) \cap (\partial\configM\cap p_b^{-1}(N_1\times N_1))$, $\partial_{1} B_{X_0} = B_{X_0}\cap p_b^{-1}(\partial N_1\times N_1)$ and $\partial_2 B_{X_0} =B_{X_0} \cap p_b^{-1}(N_1\times \partial N_1)$. 
The piece $\partial_1 B_{X_0}$ is $\{(x, y) \mid x \in D^0(y, -e_1)\cap \partial N_1, y \in N_1\}$. For any $y \in N_1$, define the following points :
\begin{itemize}
\item if $h_y \leq h_1$, $ |y_2|\leq\sqrt{{R_1}^2-{h_y}^2}$, and $y_1\geq \sqrt{{R_1}^2-{y_2}^2-{h_y}^2}$, $x_s^-(y)$ and $x_s^+(y)$ are the two\footnote{They coincide when $|y_2|=\sqrt{{R_1}^2 - {h_y}^2}$.} intersection points of $D^0(y,- e_1)$ with the sphere $\{ x \mid ||x||  = R_1\}$,
\item if $h_y\leq h_1$, $|y_2| \leq 1$, and $y_1\geq \sqrt{{R_1}^2-{y_2}^2-{h_y}^2}$, $x_c^-(y)$ and $x_c^+(y)$ are the two\footnote{They coincide when $|y_2|=1$.} intersection points of $D^0(y,- e_1)$ with the cylinder $\{ x \mid x_1^2+x_2^2=1\}$,
\item if $h_y \leq h_1$, and ${y_1}^2+{y_2}^2\leq 1$, 
$x_c^-(y)$ and $x_s^-(y)$ are the two\footnote{They coincide when $h_y = h_1$.} intersection points of $D^0(y,- e_1)$ with the cylinder or the sphere as above.
\end{itemize}

More precisely, we have the formulas \[\begin{array}{ll}
x_c^\pm(y) = \left(\pm \sqrt{1- {y_2}^2} , y_2, \overline y\right) & \ \ \ \ \ \ \ \ \ \ 
x_s^\pm(y)=\left(\pm \sqrt{{R_1}^2-{y_2}^2 - {h_y}^2}, y_2, \overline y\right)\\
%x_\Sigma^s(y) =  \left(-\sqrt{25-{h_y}^2}, 0, \overline y \right)& \ \ \ \ \ \ \ \ \ \ 
%x_\Sigma^c(y) = \left(-\sqrt{1-{y_2}^2}, y_2, \overline y\right)
\end{array},\]

when they make sense.

For any $y\in N_1$,\[D^0(y, -e_1) \cap \partial N_1=
\begin{cases}\emptyset & \text{if $h_y > h_1$ or $|y_2|>\sqrt{{R_1}^2-{h_y}^2}$}\\
\empty &\ \ \  \text{or ( $y_1<0$ and $|| y || >{R_1} $),}\\
\{x_s^-(y), x_s^+(y)\} & \text{if $h_y \leq h_1$, $1<|y_2|\leq\sqrt{{R_1}^2-{h_y}^2}$ } \\
\empty &\ \ \  \text{  and $y_1\geq \sqrt{{R_1}^2-{y_2}^2-{h_y}^2}$,}\\
\{ x_c^-(y), x_c^+(y), x_s^-(y), x_s^+(y)\} & \text{if $h_y \leq h_1$, $|y_2|\leq 1$,}\\
\empty &\ \ \ \text{  and $y_1\geq \sqrt{{R_1}^2-{y_2}^2-{h_y}^2}$,}\\
\{x_s^-(y), x_c^-(y) \} & \text{if $h_y \leq h_1$ and ${y_1}^2+{y_2}^2 \leq1$.}
\end{cases}
\]
Figure \ref{Fig4} depicts the four possible cases, with the conventions of Section \ref{conventions}. 
\begin{figure}[H]
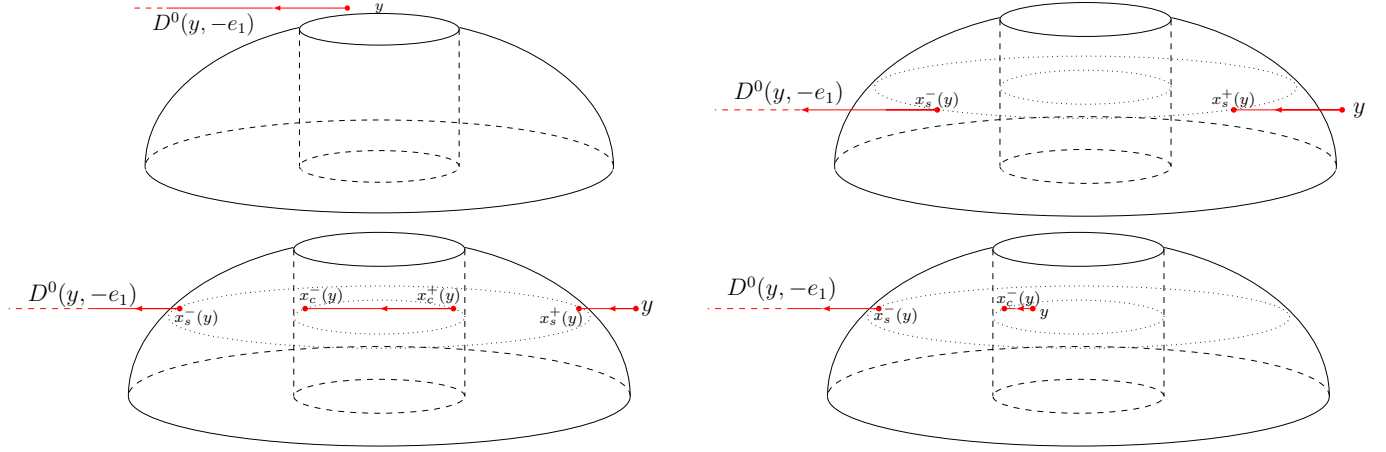

\figureba \ \ 
\figurebd \\
\figurebc \ \  
\figurebb \\
\caption{The four possible cases of Lemma \ref{bordX0}}
\label{Fig4}
\end{figure}
\end{lm}

The piece $\partial_1B_{X_0}$ splits into six faces, which are \begin{itemize}
\item the faces $\partial_{s, o}^\pm B_{X_0} = \{ (x_s^\pm (y), y) \mid h_y \leq h_1, |y_2|\leq\sqrt{{R_1}^2-{h_y}^2}, y_1\geq \sqrt{{R_1}^2-{y_2}^2-{h_y}^2}\}$, oriented by $\mp \d y$, 
\item the faces $\partial_{c, o}^\pm B_{X_0} = \{ (x_c^\pm (y), y) \mid h_y \leq h_1, |y_2|\leq1, y_1\geq \sqrt{{R_1}^2-{y_2}^2-{h_y}^2}\}$, oriented by $\pm \d y$,
\item the face $\partial_{c, i} B_{X_0} = \{(x_c^- (y), y) \mid h_y \leq h_1,{ y_1}^2+{y_2}^2\leq 1\}$, oriented by $-\d y$, \item the face $\partial_{s, i} B_{X_0} = \{(x_s^-(y), y) \mid h_y \leq h_1, {y_1}^2+{y_2}^2\leq 1\}$, oriented by $+\d y$.\end{itemize}

\begin{lm}\label{bx1}
The boundary of $B_{X_1}$ is the union of\begin{itemize}
\item the face $\partial_\ell B_{X_1} = \partial(\Sigma^- \cap E_2) \times \mur[1]$,
\item the face\footnote{The union $(-1)^n\bigcup\limits_{r\in [1, 2] }\partial (\Sigma^- \cap E_{r+1})\times \partial \mur[r]$ is oriented as $(-1)^n[1,2]\times \partial(\Sigma^-\cap E_{r+1}) \times \partial\mur[r]$.} $\partial_{\ell, \mu} B_{X_1} =(-1)^n \bigcup\limits_{r\in [1, 2] }\partial (\Sigma^- \cap E_{r+1})\times \partial \mur[r]$,
\item the face $\partial_{\mu} B_{X_1} = (-1)^{n+1}(\Sigma^-\cap E_3) \times \partial\mur[2]$,
\item the face $\partial_{E,L} B_{X_1} = - \bigcup\limits_{ r \in [1, 2]}  (\partial E_{r+1} )\times\{( R_r, 0, \overline 0)\}$,
\item the face $\partial_{E} B_{X_1} = %caspair
(-1)^{n+1}
\partial E_2 \times \Dr{+}{1}$,
\item the face\footnote{In $\ambientspace$, $\partial \Dr{+}{2}$ reduces to the point $({R_2},0,\overline 0)$ with a negative sign.} $\partial_{L} B_{X_1} = -E_3\times \partial \Dr{+}{2}$,
\item the face $\partial_\infty B_{X_1} = {G_\tau}^{-1}(\{e_1\}) \cap {p_b}^{-1}(E_2\times\{\infty\})$.
\end{itemize}
\end{lm}

Among these faces, $\partial_\ell B_{X_1}$ and $\partial_E B_{X_1}$ are contained in $\partial Y_1$.
We are going to extend the half-line $D^0(y, -e_1)$ inside $E_1$ in order to cancel the faces of $\partial_1 B_{X_0}$ and these faces $\partial_\ell B_{X_1}$ and $\partial_E B_{X_1}$. The goal of Section \ref{41} is to obtain the following lemma.

\begin{lm}\label{BY1}
There exists a chain $B_{Y_1}\subset Y_1$ such that the codimension $1$ faces of $B_{Y_1}$ are:
\begin{itemize}
\item the faces $-\partial_{c, i} B_{X_0}$, $-\partial_{s, i} B_{X_0}$, $-\partial_{s,o}^\pm B_{X_0}$, and $-\partial_{c, o}^\pm B_{X_0}$,
\item the faces $-\partial_\ell B_{X_1} $ and $-\partial_EB_{X_1}$,
\item the face $\partial_\infty B_{Y_1}=\{(x, y  = \infty, u = e_1) \mid x \in N_2\cap E_1\}
= {G_\tau}^{-1}(\{e_1\}) \cap {p_b}^{-1}((N_2\cap E_1)\times \{\infty\})$, which is oriented by $-\d x$,
\item faces $(\partial_iB_{Y_1})_{1\leq i \leq 3}$, which are contained in $p_b^{-1}((N_2\cap E_1)\times \partial N_1)\subset \partial W$ and which are described in Lemmas \ref{chaine1}, \ref{chaine2}, and \ref{chaine3}.
\end{itemize}

\end{lm} 
\subsection{Construction of the chain \texorpdfstring{$B_{Y_1}$}{BY1}}\label{41}
\subsubsection{\texorpdfstring{Cancellation of the faces $\partial _{c, i}B_{X_0}$ and $\partial_{s,i} B_{X_0}$}{Cancellation of the two first faces}}\label{411}

In this section, set $\mathcal Y_c = \{y\in N_1 \mid 0<h_y \leq h_1, {y_1}^2+{y_2}^2\leq 1\}$. Let $y$ be a point of $\mathcal Y_c$. 

If $h_y \geq h_2$, define $D^1(y, -e_1) = \{y - t e_1 \in N_2\cap E_1  \mid t >0 \}$, and orient it by $\d t$.
If $h_y \leq h_2$, define $D^1(y, -e_1)$ as the union of the following oriented arcs.

\begin{itemize}
\item The line segment $L_c^-(y)\subset N_2$ from $x_c^-(y)$ to $\partial N_2$ with direction $-e_1$. Let $x_L^c(y)= (-2 \cos(\eta_c), 2\sin(\eta_c), \overline y)$ be the intersection point of this line with $\partial N_2$ (with $\eta_c \in [-\frac{\pi}2, \frac{\pi}2]$).
\item The circular arc from $x_L^c(y)$ to $x_\Sigma^c(y) = \left(-2, 0, \overline y\right)$ given by the formula $\big(t\in [0,1]\mapsto (-2\cos((1-t)\eta_c), 2\sin( (1-t)\eta_c) , \overline y)\big)$.
\item The arc of longitude $\{x \in \partial (\Sigma^-\cap E_2) \mid h_x \geq h_y , \Dir_x=\Dir_y\}$, from $x_\Sigma^c(y)$ to $x_\Sigma^s(y) =  \left(-\sqrt{{R_2}^2-{h_y}^2}, 0, \overline y \right)$.
\item The circular arc from $x_\Sigma^s(y)$ to the point $x_L^s(y)= \left(- \sqrt{{R_2}^2 - {h_y}^2 - {y_2}^2}, y_2, \overline y\right)$, given by $ \left( t\in[0,1]\mapsto \left(-\sqrt{{R_2}^2 - {h_y}^2} \cos(t\eta_s), \sqrt{{R_2}^2 - {h_y}^2}\sin(t\eta_s), \overline y\right)\right)$, where $\eta_s\in [-\frac\pi2,\frac\pi2]$ satisfies $x_L^s(y)=\left(-\sqrt{{R_2}^2 - {h_y}^2} \cos(\eta_s), \sqrt{{R_2}^2 - {h_y}^2}\sin(\eta_s), \overline y\right)$.
\item The line segment $L_s^-(y)\subset N_2$ from $x_L^s(y)$ to $x_s^-(y)$, which has direction $-e_1$.
\end{itemize}

Figure \ref{Fig5} depicts the curve $D(y,- e_1) = D^0(y, -e_1)\cup D^1(y, -e_1)$, where the dotted part on the right is not in the plane but in the longitude.

\begin{figure}[H]
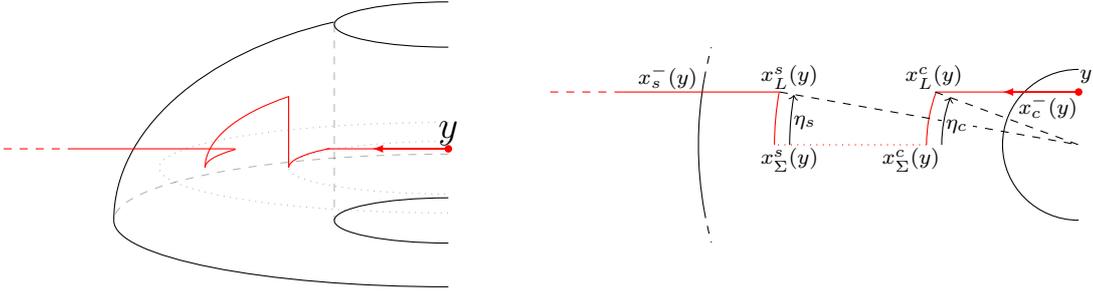

\centering
\figureca\figurecb\caption{The curve $D(y, -e_1)$ in $N_3$ (left) or in $\Pi_y=\{x\mid \overline x =\overline y\}$ (right)}\label{Fig5}

\end{figure}

\begin{lm}\label{chaine1}Set $B_c = \overline{\{(x, y) \mid y \in \mathcal Y_c, x \in D^1(y, -e_1)\}}$ and orient it by $-\d t \wedge \d y$, where $\d t$ represents the orientation of $D^1(y, -e_1)$. The codimension $1$ faces of $B_c$ are
\begin{itemize}
\item the face $-\partial_{c,i}B_{X_0}$,
\item the face $-\partial_{s,i}B_{X_0}$,
\item the face $\partial_1 B_{Y_1}=\{(x, y) \mid y \in \partial_c N_1, x \in D^1(y, -e_1)\}$, oriented by $\d t \wedge \Omega(\partial N_1)$, 
where $\Omega(\partial N_1)$ denotes the orientation of $\partial N_1$,
\item the face $-\partial_\ell B_{X_1}=-\partial (\Sigma^-\cap E_2) \times \mur[1]$.
\end{itemize}
\noindent If $(x,y)\in B_c$, and if $y\in \psi(\R^n)$, then $x\in\Sigma^-$.

\end{lm}

\begin{proof}
The first two faces and their orientations directly follow from the construction of $D^1(y, -e_1)$. 
The next two faces correspond to ${y_1}^2+{y_2}^2=1$, and $h_y=0$, respectively. The face corresponding to $h_y= h_1$ is of codimension $2$, since $D^1(y, -e_1)$ reduces to a point. 
Note the cancellation at $h_y=h_2$ since $x_\Sigma^s(y)=x_\Sigma^c(y)$ and $x_L^s(y)=x_L^c(y)$ for such a $y$.
\end{proof}
\subsubsection{Cancellation of \texorpdfstring{$\partial_{c,o}^+B_{X_0}$, $\partial_{c,o}^-B_{X_0}$, $\partial_{s, o}^{+} B_{X_0}$, and $\partial_{s, o}^{-} B_{X_0}$}{the four next faces}}\label{412}

Set $\ys = \left\{y\in N_1 \ \Big|\ h_y \leq h_1 , 0<|y_2| \leq \sqrt{{R_1}^2-{h_y}^2}, y_1\geq \sqrt{{R_1}^2-{h_y}^2-{y_2}^2} \right\}$.

In this section, for any $y\in\ys$, we are going to extend $D^0(y, -e_1)$ to a curve $D(y, -e_1)$ in $N_2$ such that $\partial D(y, -e_1) =-\{y\}$ in $\ambientspace$. In order to do so, we will connect $x_s^-(y)$ and $x_s^+(y)$, and, when they exist, we will connect $x_c^-(y)$ and $x_c^+(y)$.

We split $\ys$ in three parts $\ys^1 = \left\{y\in \ys\ \Big|\ h_2\leq h_y \text{ or } |y_2| \geq \sqrt{{R_2}^2-{h_y}^2}\right\}$, 
$\ys^2 = \left\{y \in \ys\ \Big|\ h_y \leq h_2,2< |y_2| \leq \sqrt{{R_2}^2-{h_y}^2}\right\}$, 
and $\ys^3 = \{y \in \ys\mid h_y \leq h_2, 0<|y_2|\leq 2\}$.

\textbf{First case: $y\in \ys^1$}

In this case, the half-line starting at $y$ with direction $-e_1$ is contained in $N_2$, so we set $D(y, -e_1)= \{x \in N_2 \mid x= y - t e_1 \text{for some $t\in\R^+$}\}$.

\textbf{Second case: $y\in \ys^2$}

In this case, the half-line $\{x \in N_2 \mid x= y - t e_1 \text{for some $t\in\R^+$}\}$ meets $\partial_s N_2$ in two\footnote{They coincide if $|y_2|=\sqrt{{R_2}^2-{h_y}^2}$.}
 points $x_{s,2}^\pm(y)$ as in Figure \ref{Fig6}. 
Let $\gamma_s(y)$ denote the circular arc contained in the half-circle $\partial_s N_2 \cap \{x \mid x_2y_2>0, \overline x = \overline y\}$ from $x_{s,2}^+(y)$ to $x_{s,2}^-(y)$. 
The line $D(y, -e_1)$ is the union of $\{x \in N_2 \mid x= y - t e_1 \text{for some $t\in\R^+$}\}$ and $\gamma_s(y)$.

\textbf{Third case: $y\in \ys^3$}

In this case, the half-line $\{x \in N_2 \mid x= y - t e_1 \text{for some $t\in\R^+$}\}$ meets $\partial_s N_2$ in two points $x_{s,2}^\pm(y)$ and meets $\partial_c N_2$ in two \footnote{They coincide if $|y_2|=2$.} points $x_{c,2}^\pm(y)$ as in Figure \ref{Fig6}. 
Let $\gamma_s(y)$ be defined as in the previous case, and let $\gamma_c(y)$ be the circular arc from $x_{c, 2}^-(y)$ to $x_{c, 2}^+(y)$ in the half-circle $\partial_c N_2 \cap \{x \mid x_2y_2>0, \overline x = \overline y\}$. 
The line $D(y, -e_1)$ is the union of $\{x \in N_2 \mid x= y - t e_1 \text{for some $t\in\R^+$}\}$, $\gamma_c(y)$, and $\gamma_s(y)$.

Figure \ref{Fig6} depicts the curves $D^1(y, -e_1)=D(y, -e_1)\cap (N_2\cap E_1) $ in the plane $\Pi_y=\{x\mid\overline x = \overline y\}$ for different values of $y$ in $\ys$. The two plain circles depict the boundary of $N_1$ and the two dotted circles depict the boundary of $N_2$. The orientations are given in the picture by the arrows.

\begin{figure}[H]
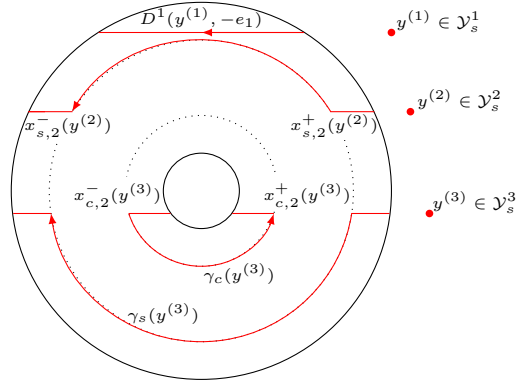

\centering
\figured
\caption{The curves $D^1(y, -e_1)$ in $ \Pi_y $.}
\label{Fig6}
\end{figure}

\begin{lm}\label{chaine2}
Set $B_s^{(1)} = \overline{\{(x, y) \mid y\in \ys, x\in D^1(y, -e_1)\}}\subset Y_1$, and orient it with $-\d t \wedge \d y$, where $\d t$ is the orientation of the lines $D^1(y, -e_1)$.
Set $\ys^0 = \{y \mid h_y \leq h_2,y_1\geq \sqrt{{R_1}^2-{h_y}^2}, y_2=0\}$.

The codimension $1$ faces of $B_s^{(1)}$ are
\begin{itemize}
\item the faces $-\partial_{s, o}^{+} B_{X_0}$, $-\partial_{s, o}^{-} B_{X_0}$, $-\partial_{c,o}^+B_{X_0}$, and $-\partial_{c,o}^-B_{X_0}$,
\item the face $\partial_\infty B_{Y_1}$ of Lemma \ref{BY1},
\item the face $\partial_2 B_{Y_1}=\overline{\left\{(x, y ) \ \Big| \ y \in \mathcal Y_s, y_1 =\sqrt{{R_1}^2-{h_y}^2-{y_2}^2}, x \in D^1(y, -e_1)\right\}}$, which is contained in ${p_b}^{-1}((N_2\cap E_1)\times \partial N_1)$, and which is oriented by 
$- \d t \wedge \d y_2 \wedge \d h_y \wedge \d \Dir_y $,
\item the face $\partial_3 B_s^{(1)}=\{(x, y) \mid y \in \ys^{0} , x \in C_s(y)\cup C_c(y)\}$ where $C_s(y)$ denotes $\Pi_y\cap\partial_s N_2$, oriented as a direct circle of the plane $\Pi_y$ (i.e. as the boundary of a disk), and $C_c(y)$ denotes the intersection $\Pi_y\cap\partial_c N_2$, with the opposite orientation. $\partial_3 B_s^{(1)}$ is oriented by $\d  t \wedge \d y_1 \wedge \d h_y \wedge \d \Dir_y$, where $\d t$ is the orientation of the circle in which $x$ lies.
\end{itemize}

\end{lm}

\begin{proof}
The first four faces follow from the fact that the line $D(y, -e_1)$ extends $D^0(y, -e_1)$.
When $y\in \mathcal Y_s^{3}$ and $h_y=h_2$, $\gamma_s(y)$ and $\gamma_c(y)$ cancel each other, so that there is no face corresponding to $h_y = h_2$. 
Note that there is no discontinuity when $|y_2| = \sqrt{{R_1}^2-{h_y}^2}$ since $\gamma_s(y)$ reduces to a point. There is no discontinuity when $|y_2|=2$ either since $\gamma_c(y)$ reduces to a point.

When $y_1$ approaches infinity, we obtain the face $\partial_\infty  B_{Y_1}$, when $y_1$ goes to $\sqrt{{R_1}^2-{h_y}^2-{y_2}^2}$, we obtain the face $\partial_2B_{Y_1}$, and when $y_2$ approaches $0$, we obtain the face $\partial_3 B_s^{(1)}$.
\end{proof}

\subsubsection{Cancellation of the face \texorpdfstring{$\partial_3 B_s^{(1)}$}{partial3Bs1}}\label{413}
For any $y\in \ys^{0}$ such that $h_y> 0$, define $A(y)$ as the annulus $\{x \in \partial N_2 \mid \Dir_x = \Dir_y,  h_x\geq h_y\}$ and orient it so that its boundary is $C_s(y)\cup  C_c(y)$.

\begin{lm}\label{chaine3}
Set $B_s^{(2)}= \overline{\{(x,y)\mid y \in \ys^{0}, h_y >0, x \in A(y)\}}$, and orient this chain by $-\Omega(A(y)) \wedge \d y_1 \wedge \d h_y \wedge \d \Dir_y$, where $\Omega(A(y))$ denotes the orientation of the annulus $A(y)$ in which $x$ lies. The codimension $1$ faces of $B_s^{(2)}$ are 
\begin{itemize}
\item the face $ -\partial_3 B_s^{(1)}$,
\item the face $ (-1)^n\partial E_2\times \Dr{+}{1}= - \partial_EB_{X_1}$,
\item the face $\partial_3 B_{Y_1}=\overline{\{(x,y) \ \Big| \ 
 0<h_y \leq h_2,y_1=\sqrt{{R_1}^2-{h_y}^2}, y_2=0, x \in A(y)\}}$, oriented by $\Omega(A(y)) \wedge \d h_y \wedge \d \Dir_y$, and contained in $p_b^{-1}((N_2\cap E_1)\times \partial N_1)$.
\end{itemize}

\end{lm}
\begin{proof}
The face $-\partial_3 B_s^{(1)}$ corresponds to the boundary of $A(y)$.
The face $-\partial_EB_{X_1}$ appears when $h_y$ approaches zero. The face corresponding to $h_y = h_2$ is of codimension $2$, since $A(y)$ degenerates to a circle.
When $y_1=\sqrt{{R_1}^2-{h_y}^2}$, we obtain the face $\partial_3B_{Y_1}$, and when $y_1$ approaches infinity, we obtain a face contained in $\{(x, y=\infty, u=e_1)\mid x \in \partial N_2\}$, thus of codimension at least two. 
\end{proof}

\subsubsection{Proof of Lemma \ref{BY1} and definition of the chain in \texorpdfstring{$X\cup Y$}{X cup Y}}

Set $B_{Y_1} = B_c + B_s^{(1)} + B_s^{(2)} $. The chain $B_{Y_1}$ satisfies the conditions of Lemma \ref{BY1}.

Let $S\colon C_2(N_3) \rightarrow  C_2(N_3)$ and $T\colon\configM \rightarrow \configM$ be the smooth maps defined on the interiors of their respective domains by the formulas $S(x,y) = (-x, -y)$ and $T(x, y) = (y, x)$, and set $B_{Y_2} = (-1)^n ST(B_{Y_1})$ and $B_{X\cup Y} = B_{X_0} + B_{Y_1}+ B_{X_1} + B_{Y_2} + B_{X_2}$. 
\begin{lm}\label{deltaZ}
Let $G_\tau$ be the map of Definition \ref{configM}.

The chain $\partial B_{X\cup Y}- G_\tau^{-1}(\{e_1\})$ defines a cycle $\delta_W$ of $\partial W\subset W$.
\end{lm}
\begin{proof}For any $1\leq i \leq 3$, set 
$\partial_i B_{Y_2} = (-1)^nST(\partial_i B_{Y_1})$. 
Set $\partial_L B_{X_2}=- \partial L_0^-(2)\times E_3$ and 
$\partial_\mu B_{X_2}=(\partial \mur[2])\times (\Sigma^+\cap E_3)$. 
Set $\partial_{\ell,\mu}B_{X_2}=(-1)^nST(\partial_{\ell, \mu} B_{X_1})$, 
and $\partial_{E,L}B_{X_2}= (-1)^nST(\partial_{E, L} B_{X_1})$, 
where the faces $\partial_{\ell, \mu} B_{X_1}$ 
and $\partial_{E, L} B_{X_1}$ are defined in Lemma \ref{bx1}.

The boundary of $B_{X\cup Y}$ is the union of  \begin{itemize}
\item the faces $(\partial_i B_{Y_1})_{1\leq i \leq 3}$ and $(\partial_i B_{Y_2})_{1\leq i \leq 3}$,
\item the faces $\partial_{\ell, \mu} B_{X_i}$, $\partial_\mu B_{X_i}$, $\partial_{E,L}B_{X_i}$, $\partial_L B_{X_i}$ for $i\in \{1,2\}$,
\item the face $ G_\tau^{-1}(\{e_1\}) \cap (X\cup Y)= ({G_\tau}_{|( \partial \configM )\setminus\unitaire E_1})^{-1}(\{e_1\})$.
\end{itemize}

All the previous faces except the last one are in $\partial W$. 
Making the difference with $G_\tau^{-1}(\{e_1\})$ replaces the last part with $- G_\tau^{-1}(\{e_1\}) \cap W=-({G_\tau}_{|\unitaire E_1})^{-1}(\{e_1\})$, which is contained in $\partial W$. \qedhere

\end{proof}

\subsection{Extension of the chain to \texorpdfstring{$W$}{W}}\label{Section43}
\subsubsection{Construction of \texorpdfstring{$B_W$}{BW} up to Lemma \ref{DZ}}
In this section, we prove that the cycle $\delta_W$ of Lemma \ref{deltaZ} is null-homologous in $W$.

\begin{lm}\label{deltaZ=0}
There exists a chain $\P_W\subset W$ such that $\partial \P_W= -\delta_W$.

\end{lm}
\begin{proof}[Proof of Lemma \ref{th-prop0} assuming Lemma \ref{deltaZ=0}.]

Let $\P_W$ be like in the lemma, so that $\partial (\P_W+\P_{X\cup Y})= G_\tau^{-1}(\{e_1\})$. Set $B^T = \frac12\left( \P_W+ \P_{X\cup Y} + T( \P_W+\P_{X\cup Y}) \right)$, so that $\partial B^T = \frac12G_\tau^{-1}(\{-e_1, e_1\})$. Note also that if $c=(x,y)\in B^T$, and if $(x,y)\in\overline{{p_b}^{-1}(\psi(\R^n)\times\ambientspace)}$, the definition of $(B_{X_i})_{1\leq i \leq 3}$ and the construction of $B_c$ in Lemma \ref{chaine1} imply that $y$ lies in the closure $\overline{\Sigma^-\cup\Sigma^+}$ of $\Sigma^-\cup\Sigma^+$ in $C_1(\spamb)$. 

 This proves the first assertion of Lemma \ref{th-prop0}. It remains to prove that admissible propagators can be chosen standard in $p_b^{-1}(N_3\times N_3)$ as stated in the second part of Lemma \ref{th-prop0}.

The previous work with the trivial knot and the surfaces $( ({}_0\Sigma^+)^0, ({}_0\Sigma^-)^0)$, yields an admissible propagator $B_0$ for $( ({}_0\Sigma^+)^0, ({}_0\Sigma^-)^0, \psi_0)$.

Set $W_2= p_b^{-1}(N_3\times N_3) \cap W$ and $W_3=\overline{W\setminus W_2}$. Set $\P_{W_2}^T = B_0\cap W_2$.

The chain $\delta_{W_3}^T=\frac12(\delta_W+T(\delta_W)) + \partial \P_{W_2}^T$ is a cycle of $W_3$, which is null-homologous in $W$ because of Lemma \ref{deltaZ=0}. 
Since $W_3$ is a deformation retract of $W$, $\delta_{W_3}^T$ is a null-homologous cycle in $W_3$. Therefore, there exists $\P_{W_3}$ in $W_3$ such that $\partial \P_{W_3}^T =- \delta_{W_3}^T$. Since $T(\delta_{W_3}^T)= \delta_{W_3}^T$, choose $\P_{W_3}^T$ such that $T(\P_{W_3}^T) = \P_{W_3}^T$.

Set $B= \frac12(\P_{X\cup Y} + T(\P_{X\cup Y}))+ \P_{W_2}^T+\P_{W_3}^T$. Since the boundary of $\frac12(\P_{X\cup Y} + T(\P_{X\cup Y}))+ \P_{W_2}^T $ is $\frac12G_\tau^{-1}(\{-e_1,e_1\})  +\frac12(\delta_W+T(\delta_W))+ \partial \P_{W_2}^T$, the chain $B$ is as requested by Lemma \ref{th-prop0}.\qedhere

\end{proof}
The rest of this section is devoted to the proof of Lemma \ref{deltaZ=0}.

Set $ W_1=p_b^{-1}(E_1\times E_1)$. 
Note that $W \hookrightarrow W_1$ is a homotopy equivalence. In order to prove Lemma \ref{deltaZ=0}, it suffices to prove that the class $[\delta_W]\in H_{n+2}(W_1)$ is null. Lemma \ref{deltaZ=0} directly follows from the following two lemmas.

\begin{lm}
Let $M_W$ denote the cycle $p_b^{-1}(\{(x,x) \mid x\in \partial \mur[2]\})= \unitaire\ambientspace_{|\partial\mur[2]}$. 

With these notations, $H_{n+2}(W_1) = \mathbb Q.[M_W]$.
\end{lm}
\begin{proof}
$W_1$ is nothing but $C_2(E_1)$, and $E_1$ is homotopic to the complement of $\psi(\R^n)\cup\{\infty\}\subset \s^{n+2}$. 
Let $\Delta_{E_1}$ denote the diagonal of ${E_1}^2$. The construction of the configuration space $C_2(E_1)$ implies that $C_2(E_1)$ has the homotopy type of its interior ${E_1}^2\setminus\Delta_{E_1}$, so that $H_{n+2}(W_1) \cong H_{n+2}({E_1}^2\setminus \Delta_{E_1})$.

Alexander duality implies that $H_*(E_1)$ is non trivial only in degree $0$ and $1$, and that $H_1(E_1)$ is generated by $[\partial \mur[2] ]$. 
Thus, $H_*({E_1}^2) = 0$ for $*> 2$ and the long exact sequence associated to ${E_1}^2\setminus \Delta_{E_1}\hookrightarrow {E_1}^2$ yields an isomorphism from $H_{n+3}({E_1}^2, {E_1}^2\setminus \Delta_{E_1})$ to $H_{n+2}({E_1}^2\setminus \Delta_{E_1})$.

The excision theorem yields an isomorphism between $H_{n+3}({E_1}^2, {E_1}^2\setminus \Delta_{E_1})$ and $ H_{n+3}( \mathcal N(\Delta_{E_1}) , \mathcal N(\Delta_{E_1}) \setminus \Delta_{E_1})$, where $\mathcal N(\Delta_{E_1})$ denotes a tubular neighborhood of $\Delta_{E_1}$.
Since $\punct M$ is parallelizable, $\mathcal N(\Delta_{E_1})$ is diffeomorphic to the trivial disk bundle $\Delta_{E_1}\times \mathbb D^{n+2}$, and \begin{eqnarray*}H_{n+3}(\mathcal N(\Delta_{E_1}), \mathcal N(\Delta_{E_1}) \setminus \Delta_{E_1}) &\cong& H_{n+3}(\Delta_{E_1}\times \mathbb D^{n+2}, \Delta_{E_1} \times (\mathbb D^{n+2}\setminus \{0\}))\\ &\cong& H_1(\Delta_{E_1}) \otimes H_{n+2}(\mathbb D^{n+2}, \partial \mathbb D^{n+2})\\
&\cong& H_1(\Delta_{E_1}) \otimes  H_{n+1}(\partial \mathbb D^{n+2})\\
&=& \mathbb Q.[\partial\mur[2]]\otimes[\partial\mathbb D^{n+2}].\end{eqnarray*} 
Therefore, $H_{n+2}(W_1)\cong \mathbb Q.[\partial\mur[2]]\otimes[\partial\mathbb D^{n+2}]$.
This identification maps $[M_W]$ to the generator $\pm[\partial \mur[2]] \otimes [\partial\mathbb D^{n+2}]$.
\end{proof}

\begin{lm}\label{DZ}
There exists an $(n+2)$-chain $D_W$, with $\partial D_W\subset \partial W_1$, such that: \begin{itemize}
\item $D_W$ is dual to $M_W$: $\langle D_W, M_W\rangle_{W_1} =\pm 1$. %(This implies that $H_{n+2}(W_1,\partial W_1 ) = \mathbb Q. [D_W]$).
\item The intersection number $\langle D_W, \delta_W\rangle_{W_1} $ is zero.
\end{itemize}

\end{lm}

Since this lemma implies that $[\delta_W]=0\in H_{n+2}(W_1)$, it implies Lemma \ref{deltaZ=0}. 
We are left with the proof of Lemma \ref{DZ}.

We will construct the chain $D_W= D_1+D_2+D_3$ as the sum of a chain $D_1$ defined in Lemma \ref{D1}, a chain $D_2$ defined in Lemma \ref{D2}, and a chain $D_3$ defined in Lemma \ref{D3}.

\subsubsection{Construction of the chain \texorpdfstring{$D_1$}{D1}}

Fix a Seifert surface $\Sigma'^0$, whose interior does not meet those used in the construction of the chain $\P_{X\cup Y}$, such that $\Sigma'^0\cap N_3 = \{(r\cos(\frac\pi6), r\sin(\frac{\pi}6), \overline x)\mid \overline x\in \R^n, r\geq0\}\cap N_3$.
If $n$ is even, 
we assume that $\mathcal T_{{\Sigma'}^0}'(1)=0$, which is always possible after attaching some trivial $1$-handles to a given Seifert surface $\Sigma_{init}$ since Lemma \ref{dernierlemme} guarantees that 
$\mathcal T_{\Sigma_{init}}'(1)$ is an integer.
Let $\Sigma'$ denote $\Sigma'^0\cap E_1$. 
Fix an embedding $\phi\colon[-1,1]\times \Sigma'\rightarrow E_1$, such that $\phi(0,x) = x$ for any $x\in\Sigma'$. This allows us to define a normal vector $n_x=\frac{ (\frac{\partial\phi}{\partial t} )(0, x)}{||(\frac{\partial\phi}{\partial t} )(0, x)||}$ for any $x\in \Sigma'$. 
Let $\Sigma'^+$ denote the parallel surface $\phi(\{1\}\times \Sigma')$, and, for any $x\in \Sigma'$, let $x^+$ denote the associated point $\phi(1,x)$ in $\Sigma'^+$. 
Assume without loss of generality that $\Sigma'^+\cap N_3= \{(r\cos(\frac\pi3), r\sin(\frac{\pi}3), \overline x)\mid \overline x\in \R^n,r\geq0 \}\cap (E_1\cap N_3)$.

\begin{lm}\label{D1}
Set $D_1 = \overline{p_b^{-1}(\{(\phi(0,x), \phi(t,x)) \mid (t,x)\in ]0,1]\times \Sigma' \})}$. The closure adds the configurations $(x, x , [n_x])\in  \partial_\Delta\configM$ where $x\in \Sigma'$. The intersection $D_1\cap M_W$ consists of the configuration $c= (x_0, x_0, [n_{x_0}])$ where $x_0$ is the unique intersection point of $\partial\mur[2]$ and $\Sigma'$. Orient $D_1$ as $[0,1]\times \Sigma'$.

The boundary of $D_1$ is the union of three codimension $1$ faces: 
\begin{itemize}
\item the face $\Delta(\Sigma',\Sigma'^+) = \{(x, x^+)\mid x\in \Sigma'\}$, oriented as $\Sigma'$,
\item the face $\partial_1D_1 = \{(x, x, [n_x]) \mid x\in \Sigma'\}$, oriented as $-\Sigma'$,
\item the face $\partial_2D_1= \overline{\{(x, \phi(t,x)) \mid 0< t\leq 1, x\in \partial\Sigma'\}}$, oriented as $-[0,1]\times\partial\Sigma'$.
\end{itemize}
Furthermore, the last two faces are contained in $\partial W_1$.
\end{lm}
\begin{proof}This is a direct check.
\end{proof} Our chain $D_W$ will be defined from $D_1$ by gluing other pieces in order to cancel the face $\Delta(\Sigma', \Sigma'^+)$, which is not contained in $\partial W_1$.
\subsubsection{Construction of the chain \texorpdfstring{$D_2$}{D2}}

Let $S'$ denote the closed surface obtained by gluing a disk $\mathbb D^{n+1}$ and $\Sigma'$ along their boundaries. The surface $S'$ is oriented as $\Sigma' \cup - \mathbb D^{n+1}$. Let $S'\times S'^+$ denote the product of two copies of $S'$, where the coordinates read $(x, y^+)$, so that $\Sigma'\times \Sigma'^+\subset W_1$ naturally embeds into $S'\times S'^+$. Set $\Delta(S', S'^+) = \{(x, x^+)\mid x\in S'\}$, and orient it as $S'$. 
\begin{nt}\label{sigmas2}
Choose two families $(\bb_i^\dimd)_{0\leq \dimd \leq n+1, i\in\und{b_\dimd}}$ and $(\ba_i^\dimd)_{0\leq \dimd \leq n+1, i\in\und{b_\dimd}}$ of cycles of $S'$ such that
\begin{itemize}
\item for any $\dimd \in\{0,\ldots, n+1\}$, the families $([\bb_i^\dimd])_{i \in \und{b_\dimd}}$ and $([\ba_i^\dimd])_{i\in \und{b_\dimd}}$ are two bases of $H_\dimd(S')$,
\item for any $\dimd \in\{0,\ldots, n+1\}$, and any $(i,j) \in (\und{b_\dimd})^2$, $\langle [\bb_i^\dimd] , [\ba^{n+1-\dimd}_j]\rangle_{S'} = \delta_{i,j}$,
\item for any $\dimd\in \und n$ and any $i\in \und{b_\dimd}$, the chains $\bb_i^\dimd$ and $\ba_i^\dimd$ are contained in $\Sigma'\cap E_3$,
\item the chains $\bb_1^0$ and $\ba_1^0$ are two distinct points of $\partial\Sigma'$, and $\bb_1^{n+1}=\ba_1^{n+1}= S'$.
%\item For any $\dimd>\frac{n+1}2$, and any $j\in \und{ b_\dimd}$, $\ba_{j}^{\dimd} = \bb_{j}^{\dimd}$, and for any $\dimd<\frac{n+1}2$ and any $j\in\und{b_\dimd}$, $\bb_{j}^{\dimd} = (-1)^{nd}\ba_{j}^{\dimd}$.
\end{itemize} 
Such a choice is possible as in Lemma \ref{sigmas}, and the previous chains induce similar families $((\bb_i^\dimd)^+)_{0\leq \dimd \leq n+1, i\in\und{ b_\dimd}}$ and $((\ba_i^\dimd)^+)_{0\leq \dimd \leq n+1, i\in \und{b_\dimd}}$ in ${S'}^+$. 
\end{nt}
\begin{lm}\label{315}
We have the following equality in $H_{n+1}(S'\times S'^+)$: 
\[ [ \Delta(S', S'^+)] = \sum\limits_{\dimd=0}^{n+1}\sum\limits_{i\in\und{b_\dimd}}(-1)^{nd} [\ba_i^\dimd\times (\bb_i^{n+1-\dimd})^+].\]
\end{lm}
\begin{proof}
The Künneth formula implies that $H_{n+1}(S'\times S'^+)$ admits the two families $([\ba_i^\dimd\times (\bb_j^{n+1-\dimd})^+])_{0\leq\dimd \leq n+1,  (i,j) \in(\und{b_\dimd})^2}$ and $([\bb_i^\dimd\times (\ba_j^{n+1-\dimd})^+])_{0\leq\dimd \leq n+1,  (i,j) \in(\und{b_\dimd})^2}$ as bases. For any $(\dimd,\dimd')\in\{0,\ldots, n+1\}^2$, any $(i,j)\in(\und{b_\dimd})^2$, and any $ ( i',j')\in(\und{b_{\dimd'}})^2$, we have the following duality property:
\[\langle [\ba_i^\dimd\times (\bb_j^{n+1-\dimd})^+] , [\bb_{i'}^{n+1-\dimd'}\times (\ba_{j'}^{\dimd'})^+]\rangle_{S'\times{S'}^+} =(-1)^{(n+1)(d'+1)}\delta_{i,i'}\delta_{j,j'}\delta_{\dimd,\dimd'}\]

There exist coefficients such that $[ \Delta(S', S'^+)] = \sum\limits_{\dimd=0}^{n+1}\sum\limits_{i=1}^{b_\dimd} \sum\limits_{j=1}^{b_\dimd}\alpha_{i,j}^\dimd [\ba_i^\dimd\times (\bb_j^{n+1-\dimd})^+]$, and the duality property above and the definition of $\Delta(S', S'^+)$ yield \begin{multline*}
\alpha_{i,j}^\dimd =(-1)^{(n+1)(d+1)} \left\langle [\Delta(S', S'^+) ] , [\bb_i^{n+1-\dimd}\times (\ba_j^{\dimd})^+]\right\rangle_{S'\times S'^+ }\\=(-1)^{(n+1)(d+1)} (-1)^{n+1-d}\delta_{i,j}.\qedhere\end{multline*}
\end{proof}

Let $\mathbb D$ denote the $(n+1)$-disk $\mathbb D^{n+1}$, which we glued to $\Sigma'$ above.
To express $\Delta(\Sigma',\Sigma'^+)$ from this lemma, we study $\Delta(\mathbb D, \mathbb D^+)= \Delta(S', S'^+)-\Delta(\Sigma',\Sigma'^+)$.

\begin{lm}\label{ddelta} 

There exists a chain $D_\delta$ in $\partial \Sigma'\times\partial\Sigma'^+$ such that the chain 
$c_{\delta,1}= D_\delta -
 \ba_1^0 \times \mathbb D^+ 
- \mathbb D\times (\bb_1^0)^+  + \Delta(\mathbb D, \mathbb D^+)$ 
is a null-homologous cycle of $\mathbb D\times \mathbb D^+$.
\end{lm}
\begin{proof}
Let $\s$ (respectively $\s^+$) denote the sphere that bounds the disk $\mathbb D$ (respectively $\mathbb D^+$), which we glued to $\Sigma'$ (respectively $\Sigma'^+$). Note that $\s= - \partial \Sigma'$. Without loss of generality,
assume that $\ba_1^0$ is the North Pole $P_N$ of the sphere $\s$, and that $\bb_1^0$ is the South Pole $P_S$. Similarly define $P_S^+$ and $P_N^+$.

For any $x\in \s\setminus \{P_N\}$, there exists a unique shortest geodesic parametrized with constant speed $(y^+_x(t))_{0\leq t \leq 1}$ on the sphere $\s^+$ going from $x^+$ to $P_S^+$. Set $D_\delta= \overline{\{(x,y^+_x(t)) \mid 0\leq t \leq 1, x\in \s\setminus\{P_N\}\}}$, and orient it as $[0,1]\times( \s\setminus \{P_N\})$.

The boundary of $D_\delta$ is the union of three codimension $1$ faces:\begin{itemize}
\item the face $\{P_N\}\times \s^+$,
\item the face $+\s\times\{P_S^+\}$,
\item the face $-\partial \Delta(\mathbb D, \mathbb D^+)$.
\end{itemize}
The first face appears when $x$ approaches $P_N$, the second one when $t=1$, and the third one when $t=0$. This implies that $c_{\delta,1}$ is a cycle. Since $H_{n+1}(\mathbb D\times \mathbb D^+)=0$, $c_{\delta,1}$ is null-homologous.
\end{proof}

\begin{lm}\label{D2}
There exists a chain $D_2\subset \Sigma'\times\Sigma'^+\subset W_1$, such that $D_2\cap \partial W_1\subset \partial D_2$, and such that

\[\partial D_2 = \left(\sum\limits_{\dimd\in\und n}\sum\limits_{i\in\und{b_\dimd}}
(-1)^{nd}
 \ba_i^\dimd\times (\bb_i^{n+1-\dimd})^+\right) + \ba_1^0\times \Sigma'^+ + \Sigma'\times (\bb_1^0)^+ + D_\delta- \Delta(\Sigma',\Sigma'^+).\]
\end{lm}

\begin{proof}

Let $c_{\delta,2} = c_{\delta,1} - \Delta(S', S'^+) + \sum\limits_{\dimd=0}^{n+1}\sum\limits_{i\in\und{b_\dimd}} 
(-1)^{nd}
 \ba_i^\dimd\times (\bb_i^{n+1-\dimd})^+$.

Lemma \ref{315} and \ref{ddelta} imply that $c_{\delta,2}$ is null-homologous in $S'\times {S'}^+$.
Now $c_{\delta,2}$ reads 
\[c_{\delta,2} = 
\left(\sum\limits_{\dimd\in\und n}\sum\limits_{i\in\und{b_\dimd}}
(-1)^{nd} 
\ba_i^\dimd\times (\bb_i^{n+1-\dimd})^+\right) 
+ \ba_1^0\times \Sigma'^+
+\Sigma'\times (\bb_1^0)^+ 
+ D_\delta- \Delta(\Sigma',\Sigma'^+).\]
Therefore, $c_{\delta, 2}$ is a cycle of $\Sigma'\times \Sigma'^+$ and the class $[c_{\delta, 2}]$ is null in $H_{n+1}(S'\times S'^+)$.
The Künneth formula proves that $([\ba_i^\dimd\times (\bb_j^{n+1-\dimd})^+])_{1\leq \dimd \leq n, 1\leq i,j \leq b_\dimd}$ is a basis of $H_{n+1}(\Sigma'\times \Sigma'^+)$. Since it is a subfamily of the basis $([\ba_i^\dimd\times (\bb_j^{n+1-\dimd})^+])_{0\leq \dimd \leq n+1, 1\leq i,j \leq b_\dimd}$ of $H_{n+1}(S'\times S'^+)$,
 the map $H_{n+1}(\Sigma'\times\Sigma'^+)\rightarrow H_{n+1}(S'\times S'^+)$ induced by the inclusion is injective, and $\left[c_{\delta, 2}\right]=0$ in $H_{n+1}(\Sigma'\times \Sigma'^+)$.\qedhere
\end{proof}
At this point, $\partial(D_1+ D_2) $ is the sum of a chain contained in $\partial W_1$ and the chain $\left(\sum\limits_{\dimd\in\und n}\sum\limits_{i\in\und{b_\dimd}} 
(-1)^{nd}
\ba_i^\dimd\times (\bb_i^{n+1-\dimd})^+\right)$, which is not contained in $\partial W_1$. It remains to define the chain $D_3$ in order to cancel $\left(\sum\limits_{\dimd \in \und n}\sum\limits_{i\in\und{b_\dimd}}
(-1)^{nd}
 \ba_i^\dimd\times (\bb_i^{n+1-\dimd})^+\right)$.

\subsubsection{Construction of the chain \texorpdfstring{$D_3$}{D3}}

Recall that the unit normal bundle to the diagonal of $\ambientspace\times\ambientspace$ has been identified with the unit tangent bundle $U\ambientspace$ of $\ambientspace$, and that it is a piece of $\partial \configM$.
\begin{lm}\label{D3}
There exists a chain $D_3\subset p_b^{-1}(E_3\times E_3)$, which meets $\partial p_b^{-1}(E_3\times E_3)$ only along $\partial D_3$, and such that $\partial D_3$ is the union of
\begin{itemize}
\item the faces $ (-1)^{nd+1}\ba_i^\dimd\times(\bb_i^{n+1-\dimd})^+$, for $\dimd\in\und n$ and $i\in\und{b_\dimd}$,
\item a finite collection of fibers $\epsilon(x_i).U_{x_i}M\subset \partial W_1 \cap\partial\configM$, for $1\leq i \leq m$.
\end{itemize}
Furthermore, $\sum\limits_{i=1}^m  \epsilon( x_i) = (-1)^{n+1}\frac{\chi(\Sigma')-1}2$.
\end{lm}
\begin{proof}
Recall that $H_*(E_3) = H_*(\s^1)$.
Therefore, for any $\dimd\in\{2,\ldots, n\}$ and any $i\in \und{b_\dimd}$, there exists $\bA_i^{\dimd+1}\subset E_3$ such that $\partial\bA_i^{\dimd+1}= \ba_i^\dimd$.
For any $i\in\und{b_1}$, the class of $\ba_i^1$ in $H_1(E_3)=\mathbb Q$ is given by its intersection number with any Seifert surface for $\psi$. Since $\ba_i^1$ does not meet $\Sigma'^+$, 
there exists a chain $\bA_i^2$ of $E_3$ such that $\partial \bA_i^2 = \ba_i^1$.
%For any $i\in \und{b_1}$, there exists $(\bA_i^{2})^0\subset \punct{M}$ such that $\partial(\bA_i^{2})^0= \ba_i^1$. 
%Since \begin{align*}\langle (\bA_i^2)^0, \psi(\R^n)\cup\{\infty\}\rangle_M &=\langle (\bA_i^2)^0, \partial (\Sigma'^+\cup\{\infty\})\rangle_M\\ 
%&=[\partial( (\bA_i^2)^0\cap ({\Sigma'}^+\cup\{\infty\}))]-\langle a_i^1, {\Sigma'}^+\cup\{\infty\}\rangle_M \\& =0,\end{align*}
%the chain $(\bA_i^{2})^0$ meets the knot in an even number of points $(x_1,\ldots, x_{2r})$ such that $x_{2i}$ and $x_{2i+1}$ have opposite signs. Cut $(\bA_i^{2})^0$ along a disk $\delta_i$ around each of these points, and glue an annulus $[0,1]\times \s^1$ between $\partial\delta_{2i}$ and $\partial\delta_{2i+1}$ for each $i$, so that the obtained chain $\bA_i^2$ does not meet the knot and the boundary of $\bA_i^2$ is $\ba_i^1$. It can be assumed that $\bA_i^2$ is contained in $E_3$. 

Without loss of generality, assume that the chains $(\bA_i^{\dimd+1})_{i,\dimd}$ have been chosen such that $\bA_i^{\dimd+1}$ and $(\bb_i^{n+1-\dimd})^+$ are transverse for any $i$ and $\dimd$. Set $K_{i,\dimd}= \bA_i^{\dimd+1}\cap (\bb_i^{n+1-\dimd})^+$ and $K=\bigcup\limits_{\dimd\in\und n} \bigcup\limits_{i\in\und{b_\dimd}} K_{i,\dimd}$.
Define \[D_3= 
\sum\limits_{\dimd\in\und n} \sum\limits_{i\in \und{b_\dimd}}
(-1)^{nd+1}
\overline{p_b^{-1}\left(\left(\bA_i^{\dimd+1} \times (\bb_i^{n+1-\dimd})^+\right)\setminus (\Delta\cap {K_{i,d}}^2)\right)}
,\]
so that \[
\partial D_3 = - \sum\limits_{\dimd\in\und n} \sum\limits_{i\in \und{b_\dimd}}  (-1)^{nd}\ba_i^\dimd\times(\bb_i^{n+1-\dimd})^+ +
\sum\limits_{x\in K} \epsilon(x).U_{x}M,
\] where $(-1)^{(n+1)d+1}\epsilon(x)$ is the sign of the intersection point. For $d\in\und n$ and $i\in\und{b_\dimd}$, \begin{multline*}\sum\limits_{x\in K_{i,\dimd}} \epsilon(x) 
=  (-1)^{(n+1)d+1}\left\langle\bA_i^{\dimd+1}, (\bb_i^{n+1-\dimd})^+\right\rangle_{\ambientspace}\\
=  (-1)^{(n+1)d+1}\lk\left( \ba_i^\dimd, (\bb_i^{n+1-\dimd})^+\right) 
=  (-1)^{d}[V_{n+1-\dimd}^-]_{i,i},\end{multline*}
where $V_{n+1-\dimd}^-=V_{n+1-\dimd}^-(\Ba, \Bb)$ as in Definition \ref{bases duales} with $\Ba = ([\bb_i^\dimd])_{i,\dimd}$ and $\Bb = ([\ba_i^\dimd])_{i,\dimd}$, so that
\[\sum\limits_{x\in K}  \epsilon(x) = \sum\limits_{\dimd\in\und n} \sum\limits_{i\in\und{b_\dimd}} \sum\limits_{x\in K_{i,\dimd}}\epsilon (x)= \sum\limits_{\dimd\in\und n} (-1)^{d} \Tr( V_{n+1-\dimd}^-).\]
Conclude with Lemma \ref{dernierlemme} and the fact that $\mathcal T_{\Sigma'}'(1)=0$.\qedhere

\end{proof}
\subsubsection{End of the proof of Lemma \ref{DZ}}
Set $D_W = D_1+D_2+D_3$. By construction, $\partial D_W$ is the union of: \begin{itemize}
\item the faces $\partial_1 D_1$ and $\partial_2 D_1$ of Lemma \ref{D1},
\item the faces $D_\delta$, $\ba_1^0\times \Sigma'^+$, and $\Sigma'\times(\bb_1^0)^+$ of Lemma \ref{D2},
\item the faces $\epsilon(x_i) U_{x_i}M$ of Lemma \ref{D3}.
\end{itemize}
All these faces are contained in $\partial W_1$.

Let us check that $\langle D_W, M_W\rangle_{W_1} = \pm 1$. We already saw that $\langle D_1, M_W\rangle_{W_1} = \pm 1$. 
Since $D_2\subset \Sigma'\times \Sigma'^+$, the chain $D_2$ does not meet $M_W$, which is contained in the diagonal.
Since $D_3\subset{p_b}^{-1}(E_3\times E_3)$, $\langle D_3, M_W\rangle_{W_1} = 0$, and $\langle D_W, M_W\rangle_{W_1} = \pm 1$.

It remains to check that $\langle D_W, \delta_W\rangle_{W_1} =0$ for the cycle $\delta_W$ of Lemma \ref{deltaZ}.
Note that $[\delta_W] = [\delta_W-\partial(\P_{X\cup Y} \cap W_1)]$. 
The cycle $\delta_W'=\delta_W -\partial(\P_{X\cup Y} \cap W_1)$ is the union of:\begin{itemize}
\item the faces $\partial_{\mu,1} \P= (-1)^{n+1}(\Sigma^-\cap E_2)\times (\partial \mur[1])$ and $\partial_{\mu,2} \P= (\partial \mur[1])\times(\Sigma^+\cap E_2)$,
\item the faces $-\partial \Dr{-}{1}\times E_2$ and $-E_2\times \partial \Dr{+}{1}$,
\item the faces $(\partial_i\P_{Y_1})_{1\leq i \leq 3}$ and $(\partial_i\P_{Y_2})_{1\leq i \leq 3}$,
\item the face $-G_\tau^{-1}(\{e_1\})\cap W_1$.
\end{itemize}

The faces $\partial \Dr{-}{1}\times E_2$, $E_2\times \partial \Dr{+}{1}$, $\partial_{\mu,1} B$ and $\partial_{\mu,2} B$ cannot meet $D_W$: indeed, they are contained in $\partial N_1\times E_2$ or $E_2\times\partial N_1$, so they do not meet $D_1$ or $D_3$, and since the points $\partial\Dr{\pm}{1}$ are on the Seifert surfaces $\Sigma^+$ and $\Sigma^-$, and since $\Sigma'$ and $\Sigma'^+$ do not meet the surfaces $\Sigma^\pm$, these faces do not meet $D_2$, either. 

For $1\leq i \leq 3$, let us study the intersection of the faces $\partial_i\P_{Y_1}\subset p_b^{-1}( (N_2\cap E_1) \times \partial N_1)$ and $\partial_i\P_{Y_2}\subset p_b^{-1}(\partial N_1\times (N_2\cap E_1) )$ with $D_W$:\begin{itemize}
\item They cannot meet $D_3$, which is contained in $p_b^{-1}(E_3\times E_3)$. 

\item They could meet $D_1$ along $D_1\cap \partial (X\cup Y) = \partial_2 D_1$, which is composed of configurations where the two points are in $\partial N_1$. The choice of the longitudes $\partial\Sigma'$ and $\partial \Sigma'^+$, and the description of these faces in Lemmas \ref{chaine1}, \ref{chaine2}, \ref{chaine3} imply that any configuration $c=(x,y)\in \partial N_1\times \partial N_1$ in one of these faces is such that $\frac{y-x}{||y-x||} = e_1$. Figure \ref{Fig7} shows that this never happens when $(x,y)\in \partial_2D_1$.
Similarly, the faces $(\partial_i B_{Y_j})_{i\in\und{3}, j\in\und{2}}$ meet
neither $D_\delta$ nor $\Delta(\Sigma', {\Sigma'}^+)$.
\begin{figure}[h]
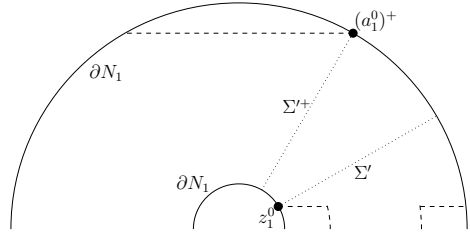

\centering
\figuredz
\caption{Dotted line: The surfaces $\Sigma'$ and $\Sigma'^+$ inside $N_3\cap \Pi_x$ for any $x$ such that $h_x< h_1$. Dashed line: The points $x\in N_2\cap E_1$ such that $(x,(a_1^0)^+)$ or $(z_1^0, x)$ lies in $\partial_2B_{Y_i}$ or $\partial_1B_{Y_i}$.
}\label{Fig7}
\end{figure}
\item Eventually, they could meet $D_2$ along $\ba_1^0\times \Sigma'^+$ and $\Sigma'\times(\bb_1^0)^+$, which would necessarily happen inside $\ba_1^0\times (\Sigma'^+\cap( N_2\cap E_1))$ or $(\Sigma'\cap(N_2\cap E_1))\times(\bb_1^0)^+$. Assume without loss that $\ba_1^0 = (\cos(\frac\pi6), \sin(\frac\pi6), \overline 0)$ and that $(\bb_1^0)^+= (18\cos(\frac\pi3), 18\sin(\frac\pi3), \overline 0)$. In this case, we get no intersection points, as it can be seen on Figure \ref{Fig7}. 
\end{itemize}
Therefore, these faces do not meet $D_W$.

We are left with the proof that $\langle D_W, -G_\tau^{-1}(\{e_1\})\cap W_1\rangle=0$. We will use the following lemma since this intersection is contained in the faces of $D_W$.

\begin{lm}
Let $P$ be an oriented manifold with boundary, let $Q$ be a submanifold of $P$, and let $R$ be a submanifold of $\partial P$ such that $\dim(R) + \dim(Q) =\dim(P)$. Assume that \begin{itemize}
\item the submanifold $Q$ meets $\partial P$ only along its boundary: $Q\cap\partial P \subset \partial Q$, and this intersection is transverse,
\item the submanifolds $Q\cap\partial P$ and $R$ are transverse in $\partial P$.
\end{itemize}
The submanifolds $Q$ and $R$ are transverse in $P$ and $\langle Q, R\rangle_P =  \langle \partial Q\cap \partial P , R\rangle_{\partial P} $.
\end{lm}\begin{proof}
The lemma follows from a direct computation.\qedhere
\end{proof}
The only configurations of $D_W$ where the two points collide with $u= \tau_x(e_1)$ are: \begin{itemize}
\item Those coming from the faces $\epsilon(x_i)U_{x_i}M\subset \partial D_3.$ Their contribution is \begin{eqnarray*} \langle D_3, -G_\tau^{-1}(\{e_1\})\cap W_1 \rangle_{W_1}&=&\left\langle \sum\limits_{i=1}^m \epsilon(x_i) U_{x_i}M , -G_\tau^{-1}(\{e_1\})\cap W_1\right\rangle_{\partial W_1}\\
&=&- \sum\limits_{i=1}^m\epsilon(x_i)\\
&=&(-1)^{n+1}\frac{1-\chi(\Sigma')}2\end{eqnarray*} 

\item Those coming from $\partial_1 D_1$. Assume without loss of generality that $e_1$ is a regular value of the map $\phi_n\colon x\in\Sigma'\mapsto \tau_x^{-1}(n_x)\in \s^{n+1}$. Their contribution is \begin{eqnarray*}\langle D_1, -G_\tau^{-1}(\{e_1\})\cap W_1\rangle_{W_1}&=& \langle\partial_1 D_1, -G_\tau^{-1}(\{e_1\})\cap W_1\rangle_{\partial W_1}\\ 
&=&+\deg_{e_1}(\phi_n),\end{eqnarray*} where $\deg_y(\phi_n)$ is the differential degree of $\phi_n$ at $y$, and where the plus sign comes from the fact that the face $\partial_1D_1$ of Lemma \ref{D1} is oriented as $- \Sigma'$.
\end{itemize}

The proof of Lemma \ref{DZ} is now completed by the following lemma, since $\chi(\Sigma')=1$ if $n$ is even.

\begin{lm}\label{lm419}
Let $\phi_n$ be the map $x\in \Sigma'\mapsto \tau_x^{-1}(n_x)\in\s^{n+1}$. 

The differential degree of $\phi_n$ may be extended to the constant map on $\s^{n+1}$ with value $\frac{\chi(\Sigma')-1}2$.
\end{lm}

\begin{proof}
For any $x\in \Sigma'\cap N_3$, $\phi_n(x)= (\cos(\frac{2\pi}3),  \sin(\frac{2\pi}3), \overline 0)$. All the boundary of $\Sigma'$ is mapped by $\phi_n$ to one point in $\s^{n+1}$. This implies that the differential degree of $\phi_n$ does not depend on the chosen regular value in $\s^{n+1}$. Assume that $\phi_n$ admits $-e_1$ and $e_1$ as regular values, without loss of generality.

For any $x\in \Sigma'$, define the projection $X(x)$ of $\tau_x(e_1)$ on $T_x\Sigma'$ along the direction $n_x$ (which is the only vector of $T_x\Sigma'$ that can be expressed as $\tau_x(e_1)-\lambda n_x$ for some $\lambda\in\R$).
This defines a tangent vector field $X$ on $\Sigma'$, whose zeros are the points such that $\phi_n(x) = \pm e_1$. Around such a zero $z$, $\phi_n$ is a local diffeomorphism from a disk around $z$ to a disk inside $\s^{n+1}$. In this setting, the index $i(X, z)$ of the zero is $+1$ if and only if this local diffeomorphism preserves the orientation. This implies that $\sum\limits_{z  \text{ zero of X}} i(X, z) = \deg_{e_1}(\phi_n)+\deg_{-e_1}(\phi_n)$. Since $\deg(\phi_n)$ does not depend on the regular value, $\deg(\phi_n) = \frac12\sum\limits_{z \text{ zero of X}} i(X,z)$.

Let $\mathbb D\subset \Sigma'$ be the set $\{(r\cos(\frac\pi6), r\sin(\frac\pi6), \overline x) \mid  1\leq r \leq 2,\overline x \in \R^n\}\cap E_1$, as depicted in Figure \ref{Fig8}. This is an $(n+1)$-disk on which $X$ takes a constant value $X_0\neq 0$. Change the vector field $X$ on $\mathbb D$ so that it keeps the same value on $\partial \mathbb D\setminus \partial \Sigma'$ but is going outwards on all $\mathbb D\cap \partial \Sigma'$. The obtained vector field $X'$ is going outwards on $\partial \Sigma'$ and $X'_{|\mathbb D}$ is going outwards on $\partial \mathbb D$ as in Figure \ref{Fig8}. The zeros of $X'$ are the union of those of $X$ with same indices (which are in $\Sigma'\setminus \mathbb D$) and those of $X'_{|\mathbb D}$.
In this setting, Poincaré-Hopf theorem (see for example \cite[Section 6, p 35]{[Milnor]}) yields $\sum\limits_{z \text{ zero of $X'$}} i(X',z)=\chi(\Sigma'\cup \mathbb D)= \chi(\Sigma')$, and $\sum\limits_{z \text{ zero of $X'_{|\mathbb D}$}} i(X',z)= \chi(\mathbb D)=1$.

The difference of these two formulas gives $\sum\limits_{z \text{ zero of $X$}} i(X,z)= \chi(\Sigma')-1$, and implies the lemma.\qedhere

\begin{figure}[H]
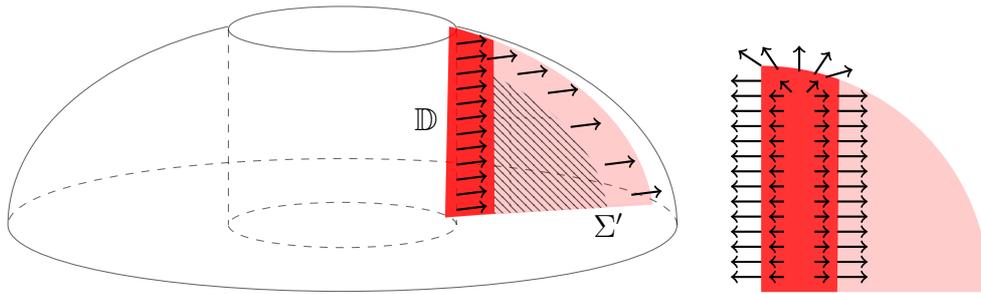
\label{fig4}
\centering
\figuresigma\ \ \ \figuresigmab
\caption{Left: The surface $\Sigma'$ with the darker disk $\mathbb D$, and the vector field $X$. The hashed area depicts $\Sigma'\cap E_2$, which is not necessarily a disk as in the picture. Right: The modified field $X'$ on $\mathbb D$, which points outwards on the boundary.}
\label{Fig8}
\end{figure}
\end{proof}

	\bibliographystyle{alpha}
	\bibliography{Brouillon}
\end{document}